\documentclass{amsart}
\usepackage{amsthm,mathtools}
\usepackage[all]{xy}

\theoremstyle{plain}
\theoremstyle{definition}
\newtheorem{theorem}{Theorem}[section]
\newtheorem{lemma}[theorem]{Lemma}
\newtheorem{proposition}[theorem]{Proposition}
\newtheorem{corollary}[theorem]{Corollary}

\newtheorem{definition}[theorem]{Definition}
\newtheorem{example}[theorem]{Example}

\newtheorem{remark}[theorem]{Remark}

\usepackage{amssymb,amsmath,tabularx,graphicx,float,placeins}
\usepackage{tikz}
\usepackage{mathrsfs, amstext, amsfonts, graphicx, stmaryrd, bbm, mathabx, array}
\usetikzlibrary[calc,intersections,through,backgrounds,arrows,decorations.pathmorphing]

\def\Bbf{\mathbf{B}}

\def\Acal{\mathcal{A}}

\def\Fcal{\mathcal{F}}

\def\Tcal{\mathcal{T}}

\def\ds{\displaystyle}

\def\ora{\overrightarrow}
\def\ov{\overline}

\def\pr{\prime}

\def\ra{\rightarrow}

\def\setm{\setminus}

\def\wtil{\widetilde}

\DeclareMathOperator{\Bic}{Bic}  
\DeclareMathOperator{\Con}{Con}  
\DeclareMathOperator{\con}{con}  
\DeclareMathOperator{\Cor}{Cor}  
\DeclareMathOperator{\Cov}{Cov}  
\DeclareMathOperator{\Endpt}{Endpt} 
\DeclareMathOperator{\id}{id}  
\DeclareMathOperator{\JI}{JI}
\DeclareMathOperator{\Kr}{Kr}  
\DeclareMathOperator{\MI}{MI}
\DeclareMathOperator{\NCP}{NCP}  
\DeclareMathOperator{\Seg}{Seg}  

\begin{document}

\title{Oriented flip graphs and noncrossing tree partitions}
\author{Alexander Garver}
\author{Thomas McConville}

\thanks{The first author was supported by a Research Training Group, RTG grant DMS-1148634.}

\begin{abstract}

The purpose of this paper is to understand the lattice properties of posets of torsion pairs in the module category of representation-finite gentle algebras called tiling algebras, as introduced by Coelho Simoes--Parsons. We present a combinatorial model for torsion pairs in the module category of tiling algebras using polyogonal subdivisions (equivalently, partial triangulations) of a convex polygon. We combine this model with lattice-theoretic techniques to classify 2-term simple-minded collections in bounded derived categories of tiling algebras. As a consequence, we obtain a characterization of \textbf{c}-matrices for any quiver mutation-equivalent to a type $A$ Dynkin quiver. 

Our model is developed using the tree that is dual to a given polygonal subdivision. Given such a tree, we introduce a simplicial complex of noncrossing geodesics supported by the tree which we call the noncrossing complex. The facets of the noncrossing complex have the structure of an oriented flip graph. Special cases of these oriented flip graphs include the Tamari order, type $A$ Cambrian orders, and oriented exchange graphs of quivers mutation-equivalent to a type $A$ Dynkin quiver. We prove that the oriented flip graph is a polygonal, congruence-uniform lattice. To do so, we express the oriented flip graph as a lattice quotient of a lattice of biclosed sets.

The facets of the noncrossing complex have an alternate ordering known as the shard intersection order. We prove that this shard intersection order is isomorphic to a lattice of noncrossing tree partitions. The oriented flip graph inherits a cyclic action from its congruence-uniform structure. On noncrossing tree partitions, this cyclic action generalizes the classical Kreweras complementation on noncrossing set partitions. We show that the data of a noncrossing tree partition and its Kreweras complement is equivalent to a 2-term simple-minded collection of the associated tiling algebra.
\end{abstract}

\maketitle

\tableofcontents

\section{Introduction}

Triangulations of marked surfaces provide an incredibly useful model for studying the combinatorics and representation theory related to cluster algebras \cite{fomin2008cluster}. The arcs on a surface are in bijection with the cluster variables, the triangulations are in bijection with clusters, and moving between two triangulations by flipping an arc corresponds to performing a single mutation on the corresponding clusters. In addition, compatibility of two cluster variables means their corresponding arcs are noncrossing. On the representation theory side, the additive categorification of cluster algebras \cite{bmrrt, ccs1} has been described using marked surfaces (see \cite{ccs1, brustle2011cluster, qiu2013cluster}).

More recent work (see \cite{by14} and the many references therein) has shown that exchange graphs of many cluster algebras can be modeled using many different representation theoretic objects related to certain Jacobian algebras $\Lambda$ \cite{dwz1}. In particular, the poset of functorially finite torsion pairs in the module category of $\Lambda$ and the poset of 2-term simple-minded collections in the bounded derived category of $\Lambda$ are isomorphic to the oriented exchange graph \cite{bdp} of the cluster algebra defined by the quiver of $\Lambda$. We remark that any simple-minded collection in the bounded derived category of $\Lambda$ can be realized as the set of simple objects in the heart of a bounded $t$-structure on the bounded derived category of $\Lambda$ obtained via Happel--Reiten--Smal{\o} tilting \cite{happel1996tilting}, as is shown in \cite{koenig2014silting}.

The connection between these geometric and representation theoretic objects can be seen as follows. Let $(\textbf{S}, \textbf{M})$ be a marked surface, and add to it a certain \textbf{lamination} $\mathcal{L}$ (see \cite{fomin2012cluster}). By choosing a triangulation $\Delta$ of the laminated marked surface, one computes the \textbf{shear coordinates} to obtain the \textbf{c}-matrix $C$ of the cluster corresponding to $\Delta$. The row vectors of $C$ are the \textit{signed} dimension vectors of the objects in a 2-term simple-minded collection. In this way, \textbf{c}-matrices of the cluster algebra associated to $(\textbf{S}, \textbf{M}, \mathcal{L})$ are in bijection with the 2-term simple-minded collections of the Jacobian algebra $\Lambda$ associated with $(\textbf{S}, \textbf{M}, \mathcal{L})$. The torsion pair corresponding to $C$ has torsion part (resp. torsion-free part) generated (resp. cogenerated) by indecomposable $\Lambda$-modules $M$ with $\underline{\text{dim}}(M)$ (resp. $-\underline{\text{dim}}(M)$) a row vector of $C$ (see \cite{by14} for a proof that these maps are bijections when $\Lambda$ is a finite dimensional Jacobian algebra).

It is natural to ask if the correspondences just described make sense when one considers polygonal subdivisions (equivalently, partial triangulations) of a marked surface. The goal of this paper is to address this question when the surface is a disk with marked points on its boundary. We obtain analogous isomorphisms between \textbf{oriented flip graphs} of such polygonal subdivisions and posets of torsion pairs and 2-term simple-minded collections, using lattice theory and the combinatorics of string modules. As there is not a known cluster structure on polygonal subdivisions, we do not know whether all of our results have cluster theoretic interpretations. 






\subsection{Overview} {The purpose of this work is to understand the combinatorics and representation theory associated with lattices of polygonal subdivisions of a convex polygon. Given such a polygonal subdivision, one naturally associates to it a finite dimensional algebra $\Lambda$, which we will refer to as a \textbf{tiling algebra} \cite{simoes2016endomorphism}. Our aim in this paper is to: \begin{itemize}\item provide a combinatorial model for the torsion pairs in the module category of $\Lambda$ and \item classify 2-term simple-minded collections in the bounded derived category of $\Lambda$. \end{itemize}}

Tiling algebras are a class of representation-finite gentle algebras that were very recently introduced in \cite{simoes2016endomorphism}. These algebras also form a subclass of the {algebras of partial triangulations} introduced in \cite{demonet2016algebras}. The class of tiling algebras contains nice families of algebras including {Jacobian algebras} \cite{dwz1} of type $A$ and {$m$-cluster-tilted algebras} \cite{murphy2010derived} of type $A$, both of which naturally arise in the study of cluster algebras \cite{fomin2002cluster} and in the additive categorification of cluster algebras \cite{bmrrt, ccs1}.

We refer to the lattices of polygonal subdivisions we study as \textbf{oriented flip graphs} (see Definition~\ref{orflipgraph}). Special cases of these posets include the Tamari order, type $A$ Cambrian lattices \cite{ReadingCamb}, oriented exchange graphs of type $A$ cluster algebras \cite{bdp}, and the Stokes poset of quadrangulations defined by Chapoton \cite{chapoton:stokes}. 

Rather than directly studying polygonal subdivisions, it turns out to be more convenient to formulate our theory in terms of trees that are dual to polygonal subdivisions of a polygon. That is, our work begins with the initial data of a tree $T$ embedded in a disk so that its leaves lie on the boundary and its other vertices lie in the interior of the disk. This data gives rise to a simplicial complex of \textbf{noncrossing} sets of \textbf{arcs} on this tree that we call the \textbf{reduced noncrossing complex} (see Section~\ref{Sec:noncrossingcomplex} for the precise definitions of these notions). The combinatorics of the facets of this pure, thin simplcial complex (see Corollary~\ref{Cor_pure_thin}) allow us to define our oriented flip graphs, which we denote by $\overrightarrow{FG}(T)$.

Our first main combinatorial result (Theorem~\ref{thm_eta_phi_main}), which sets the stage for the rest of the paper, is that these oriented flip graphs are \textbf{congruence-uniform lattices}. The Tamari order is a standard example of a congruence-uniform lattice \cite{geyer:intervallverdopplung}; see also \cite{caspard.poly-barbut:tamari}, \cite{ReadingCamb}. Nathan Reading gave a proof of congruence-uniformity of the Tamari order by proving that the weak order on permutations is congruence-uniform and applying the lattice quotient map from the weak order to the Tamari order defined by Bj\"orner and Wachs in \cite{bjorner.wachs:shellable_2}. To prove our congruence-uniformity result, we take a similar approach. We define a congruence-uniform lattice of \textbf{biclosed sets} of $T$, denoted $\text{Bic}(T)$, and identify the oriented flip graph $\overrightarrow{FG}(T)$ with a lattice quotient of $\text{Bic}(T)$. This method was applied to some other Tamari-like lattices in \cite{garver2015lattice},\cite{mcconville:grassmann}. The technique of studying a lattice by realizing it as a quotient lattice is not new, see for example \cite{pilaud:brick_lattice}, \cite{pons:lattice}.

Congruence-uniform lattices admit an alternate poset structure called the \textbf{shard intersection order} \cite{ReadingPAB}. For example, the shard intersection order of the Tamari lattice is the lattice of noncrossing set partitions \cite{reading:noncrossing}. We introduce a new family of objects called \textbf{noncrossing tree partitions} of $T$, and identify the shard intersection order of $\overrightarrow{FG}(T)$ with the lattice of noncrossing tree partitions of $T$, denoted NCP$(T)$ (Theorem~\ref{Thm:rhocircPsi}).

\subsection{Organization and main results} In Section~\ref{Sec:oreg}, we recall the definition of \textbf{oriented exchange graphs}, which are defined by the initial data of a quiver. When the quiver is in the \textbf{mutation-class} of a type $A$ Dynkin quiver, its oriented exchange graph is isomorphic to an oriented flip graph (see Theorem~\ref{orexandorflip}). In Sections~\ref{subsec_lattices} and \ref{subsec_cu_lattices}, we review the lattice theory that we will use to obtain many of our results.

Our main combinatorial and lattice-theoretic results appear Sections~\ref{Sec:noncrossingcomplex}, \ref{Sec:Sub_and_Quot}, \ref{Sec:NCP}, and \ref{Sec:Poly_Sub}. In Section~\ref{Sec:noncrossingcomplex}, we introduce the noncrossing complex and reduced noncrossing complex of arcs on a tree. We then develop the combinatorics of these complexes, which is an important part of the definition of oriented flip graphs. In Section~\ref{Sec:Sub_and_Quot}, we introduce the lattice of biclosed sets of $T$ and we show how the oriented flip graph $\overrightarrow{FG}(T)$ is both a sublattice and quotient lattice of $\text{Bic}(T)$. 

In Section~\ref{Sec:NCP}, we introduce {noncrossing tree partitions} of $T$, which generalize the classical noncrossing set partitions. We show that, as in the classical case, noncrossing tree partitions form a lattice NCP$(T)$ under refinement. Furthermore, we show that NCP$(T)$ is isomorphic to the shard intersection order of the oriented flip graph of $T$ (Theorem~\ref{Thm:rhocircPsi}). In Section~\ref{Sec:Poly_Sub}, we show that the top element of $\overrightarrow{FG}(T)$ is obtained by \textit{rotating} arcs in the bottom element of $\overrightarrow{FG}(T)$ (see Theorem~\ref{univtagrotation}). This result recovers one of Br\"ustle and Qiu (see \cite{brustle2015tagged}) in the case where the surface is a disk without punctures. 

In Sections~\ref{Sec:FinDimAlg} and \ref{Sec:SMCs}, we interpret the combinatorics of oriented flip graphs and noncrossing tree partitions in terms of the representation theory of the {tiling algebra} $\Lambda_T$ defined by $T$. In Section~\ref{subsec:torsf}, we show that the lattice of torsion-free classes (resp. torsion classes) of $\Lambda_T$ ordered by inclusion (resp. reverse inclusion) is isomorphic to $\overrightarrow{FG}(T)$ (see Theorem~\ref{Thm:TorsfOrFlipIso}). To obtain this result, we make use of the lattice quotient description of $\overrightarrow{FG}(T)$ from Section~\ref{Sec:Sub_and_Quot} and the classification of extensions between indecomposable $\Lambda_T$-modules found in Section~\ref{Sec:stringalgtree}. In Section~\ref{Subsec:NCP_and_wide}, we show that the poset of noncrossing tree partitions of $T$ is isomorphic to the poset of \textbf{wide subcategories} of $\Lambda_T$-mod. Wide subcategories have already been used in \cite{ingalls2009noncrossing} to model the lattice of noncrossing partitions associated with a Dynkin quiver.

In Section~\ref{Sec:SMCs}, we show that the data of a noncrossing tree partition and its \textbf{Kreweras complement} is equivalent to a \textbf{2-term simple-minded collection} of objects in the bounded derived category of $\Lambda_T$ (see Theorem~\ref{Thm:ncpsmcbijection}). This theorem relies on the description of extensions between indecomposable $\Lambda_T$-modules found in Section~\ref{Sec:stringalgtree} and on a combinatorial description of the operation of \textbf{left-} and \textbf{right-mutation} on simple-minded collections found in Section~\ref{Sec:MutSMC} (see Lemma~\ref{Cone_Lemma}).

We conclude the paper with a classification of \textbf{c}-matrices of quivers defined by triangulations of polygons (see Theorem~\ref{Thm:cmatclassif}). This classification is similar to the classification obtained in \cite{st13} for acyclic quivers and to the classification found in \cite{garver2015combinatorics} for type $A$ Dynkin quivers.

{\bf Acknowledgements.~}
Alexander Garver thanks Kiyoshi Igusa, Gregg Musiker, Ralf Schiffler, Sibylle Schroll, and Hugh Thomas for helpful conversations. The authors thank Emily Barnard for useful discussions. They also thank Hugh Thomas for noticing the connection between noncrossing tree partitions and wide subcategories.

\section{Preliminaries}

\subsection{Oriented exchange graphs}\label{Sec:oreg} The oriented flip graphs that we will introduce in Section~\ref{Sec:noncrossingcomplex} generalize a certain subclass of oriented exchange graphs of quivers, which are important objects in representation theory of finite dimensional algebras. We present the definition of oriented exchange graphs to motivate the introduction of the former.

A \textbf{quiver} $Q$ is a directed graph. In other words, $Q$ is a 4-tuple $(Q_0,Q_1,s,t)$, where $Q_0 = [m] := \{1,2, \ldots, m\}$ is a set of \textbf{vertices}, $Q_1$ is a set of \textbf{arrows}, and two functions $s, t:Q_1 \to Q_0$ defined so that for every $\alpha \in Q_1$, we have $s(\alpha) \xrightarrow{\alpha} t(\alpha)$. An \textbf{ice quiver} is a pair $(Q,F)$ with $Q$ a quiver and $F \subset Q_0$ a set of \textbf{frozen vertices} with the restriction that any $i,j \in F$ have no arrows of $Q$ connecting them. By convention, we assume $Q_0\backslash F = [n]$ and $F = [n+1,m] := \{n+1, n+2, \ldots, m\}.$ Any quiver $Q$ is regarded as an ice quiver by setting $Q = (Q, \emptyset)$.

If a given ice quiver $(Q,F)$ has no loops or 2-cycles, we can define a local transformation of $(Q,F)$ called \textbf{mutation}. The {\bf mutation} of an ice quiver $(Q,F)$ at a nonfrozen vertex $k$, denoted $\mu_k$, produces a new ice quiver $(\mu_kQ,F)$ by the three step process:

(1) For every $2$-path $i \to k \to j$ in $Q$, adjoin a new arrow $i \to j$.

(2) Delete any $2$-cycles created during the first steps.

(3) Reverse the direction of all arrows incident to $k$ in $Q$.

\noindent We show an example of mutation below with the nonfrozen (resp. frozen) vertices in black (resp. blue).

\vspace{-.1in}
\[
\begin{array}{c c c c c c c c c}
\raisebox{-.4in}{$(Q,F)$} & \raisebox{-.4in}{=} & {\begin{xy} 0;<1pt,0pt>:<0pt,-1pt>:: 
(0,30) *+{1} ="0",
(40,0) *+{2} ="1",
(80,30) *+{3} ="2",
(40,60) *+{\textcolor{blue}{4}} ="3",
"0", {\ar@<-.5ex>"1"},
"0", {\ar@<.5ex>"1"},
"1", {\ar"2"},
"3", {\ar"2"},
\end{xy}} & \raisebox{-.4in}{$\stackrel{\mu_2}{\longmapsto}$} & {\begin{xy} 0;<1pt,0pt>:<0pt,-1pt>:: 
(0,30) *+{1} ="0",
(40,0) *+{2} ="1",
(80,30) *+{3} ="2",
(40,60) *+{\textcolor{blue}{4}} ="3",
"2", {\ar"1"},
"0", {\ar@<-.5ex>"2"},
"0", {\ar@<.5ex>"2"},
"1", {\ar@<-.5ex>"0"},
"1", {\ar@<.5ex>"0"},
"3", {\ar"2"},
\end{xy}} & \raisebox{-.4in}{=} & \raisebox{-.4in}{$(\mu_2Q,F)$}
\end{array}
\]

The information of an ice quiver can be equivalently described by its (skew-symmetric) \textbf{exchange matrix}. Given $(Q,F),$ we define $B = B_{(Q,F)} = (b_{ij}) \in \mathbb{Z}^{n\times m} := \{n \times m \text{ integer matrices}\}$ by $b_{ij} := \#\{i \stackrel{\alpha}{\to} j \in Q_1\} - \#\{j \stackrel{\alpha}{\to} i \in Q_1\}.$ Furthermore, ice quiver mutation can equivalently be defined  as \textbf{matrix mutation} of the corresponding exchange matrix. Given an exchange matrix $B \in \mathbb{Z}^{n\times m}$, the \textbf{mutation} of $B$ at $k \in [n]$, also denoted $\mu_k$, produces a new exchange matrix $\mu_k(B) = (b^\prime_{ij})$ with entries
\[
b^\prime_{ij} := \left\{\begin{array}{ccl}
-b_{ij} & : & \text{if $i=k$ or $j=k$} \\
b_{ij} + \frac{|b_{ik}|b_{kj}+ b_{ik}|b_{kj}|}{2} & : & \text{otherwise.}
\end{array}\right.
\]
For example, the mutation of the ice quiver above (here $m=4$ and $n=3$) translates into the following matrix mutation. Note that mutation of matrices and of ice quivers is an involution (i.e. $\mu_k\mu_k(B) = B$). 
\[
\begin{array}{c c c c c c c c c c}
B_{(Q,F)} & = & \left[\begin{array}{c c c | r}
0 & 2 & 0 & 0 \\
-2 & 0 & 1 & 0\\
0 & -1 & 0 & -1\\
\end{array}\right]
& \stackrel{\mu_2}{\longmapsto} &
\left[\begin{array}{c c c | r}
0 & -2 & 2 & 0 \\
2 & 0 & -1 & 0\\
-2 & 1 & 0 & -1\\
\end{array}\right] 
& = & B_{(\mu_2Q,F)}.
\end{array}
\]

Let Mut($(Q,F)$) denote the collection of ice quivers obtainable from $(Q,F)$ by finitely many mutations where such ice quivers are considered up to an isomorphism of quivers that fixes the frozen vertices. We refer to Mut($(Q,F)$) as the \textbf{mutation-class} of $Q$. Such an isomorphism is equivalent to a simultaneous permutation of the rows and first $n$ columns of the corresponding exchange matrices.

{Given a quiver $Q$, we define its \textbf{framed} quiver to be the ice quiver $\widehat{Q}$  where $\widehat{Q}_0 := Q_0 \sqcup [n+1, 2n]$, $F = [n+1, 2n]$, and $\widehat{Q}_1 := Q_1 \sqcup \{i \to n+i: i \in [n]\}$.}  We define  the \textbf{exchange graph} of $\widehat{Q}$, denoted $EG(\widehat{Q})$, to be the (a priori, infinite) graph whose vertices are elements of Mut$(\widehat{Q})$  and two vertices are connected by an edge if the corresponding quivers differ by a single mutation.

The exchange graph of $\widehat{Q}$ has natural acyclic orientation using the notion of $\textbf{c}$-vectors. We refer to this directed graph as the \textbf{oriented exchange graph} of $Q$, denoted $\overrightarrow{EG}(\widehat{Q})$. Given $\widehat{Q}$, we say that $C = C_R$ is a \textbf{c}-\textbf{matrix} of $Q$ if there exists $R \in EG(\widehat{Q})$ such that $C$ is the $n\times n$ submatrix of $B_R = (b_{ij})_{i \in [n], j \in [2n]}$ containing its last $n$ columns. That is, $C = (b_{ij})_{i \in [n], j \in [n+1,2n]}$. We let \textbf{c}-mat($Q$) $:= \{C_R: R \in EG(\widehat{Q})\}$. A row vector of a \textbf{c}-matrix, $\textbf{c}_i$, is known as a \textbf{c}-\textbf{vector}. Since a \textbf{c}-matrix $C$ is only defined up to a permutations of its rows, $C$ can be regarded simply as a \textit{set} of \textbf{c}-vectors.

The celebrated theorem of Derksen, Weyman, and Zelevinsky \cite[Theorem 1.7]{dwz10}, known as the {sign-coherence} of $\textbf{c}$-vectors, states that for any $R \in EG(\widehat{Q})$ and $i \in [n]$ the \textbf{c}-vector $\textbf{c}_i$ is a nonzero element of $\mathbb{Z}_{\ge 0}^n$ or $\mathbb{Z}_{\le0}^n$. If $\textbf{c}_i \in \mathbb{Z}_{\ge 0}^n$ (resp. $\textbf{c}_i \in \mathbb{Z}_{\le 0}^n$) we say it is \textbf{positive} (resp. \textbf{negative}). It turns out that for any quiver $Q$ one has $\textbf{c}\text{-vec(}Q\text{)} := \{\text{\textbf{c}-vectors of $Q$}\} = \textbf{c}\text{-vec(}Q\text{)}^+ \sqcup -\textbf{c}\text{-vec(}Q\text{)}^+$ where $\textbf{c}\text{-vec(}Q\text{)}^+ := \{\text{positive } \textbf{c}\text{-vectors of } Q\}$.

\begin{figure}[t]
$$\begin{array}{cc}
 \raisebox{.05in}{\includegraphics[scale=1]{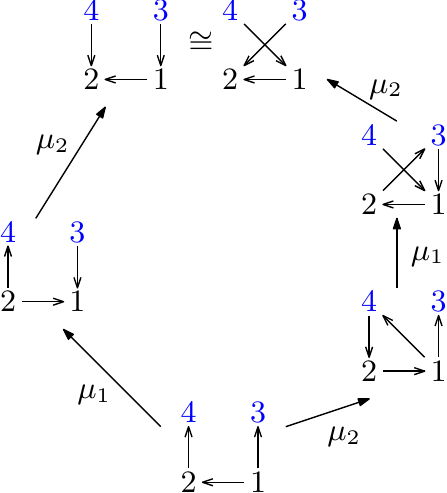}} & \raisebox{.1in}{\includegraphics[scale=.9]{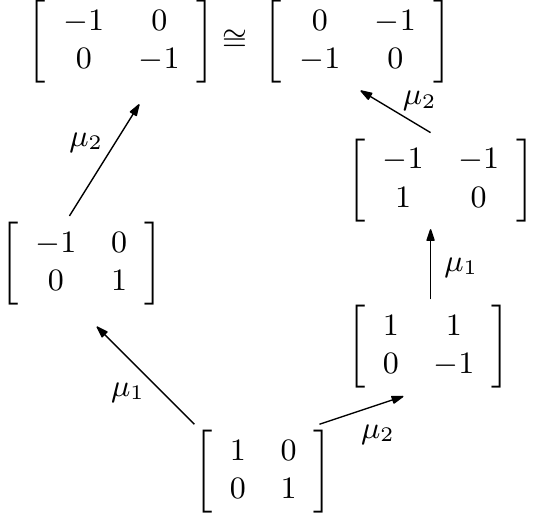}}
\end{array}$$ 
\vspace{-.3in}
\caption{The oriented exchange graph of of $Q=2\leftarrow 1$ and the corresponding $\textbf{c}$-matrices.}
\label{oregA2_1}
\end{figure}

\begin{definition}\cite{bdp}
The \textbf{oriented exchange graph} of a quiver $Q$, denoted $\overrightarrow{EG}(\widehat{Q})$, is the directed graph whose underlying unoriented graph is $EG(\widehat{Q})$ with directed edges $(R^1,F) \longrightarrow (\mu_kR^1,F)$ if and only if $\textbf{c}_k$ is positive in $C_{R^1}$. In Figure~\ref{oregA2_1}, we show $\overrightarrow{EG}(\widehat{Q})$ and we also show all of the \textbf{c}-matrices in $\textbf{c}\text{-mat(}Q\text{)}$ where $Q = 2 \leftarrow 1$.
\end{definition}

\subsection{Lattices}\label{subsec_lattices}

We will see that many properties of oriented flip graphs can be deduced from their lattice structure. In this section, we review some background on lattice theory, following \cite{ReadingPAB}. Unless stated otherwise, we assume that all lattices considered are finite.

Given a poset $(P,\leq)$, the \textbf{dual} poset $(P^*,\leq^*)$ has the same underlying set with $x\leq^* y$ if and only if $y\leq x$. A \textbf{chain} in $P$ is a totally ordered subposet of $P$. A chain $x_0<\cdots<x_N$ is \textbf{saturated} if there does not exist $y\in P$ such that $x_{i-1}<y<x_i$ for some $i$. A saturated chain is \textbf{maximal} if $x_0$ is a minimal element of $P$ and $x_N$ is a maximal element of $P$.

A \textbf{lattice} is a poset for which any two elements $x,y$ have a least upper bound $x\vee y$ called the \textbf{join} and a greatest lower bound $x\wedge y$ called the \textbf{meet}. Any finite lattice has a lower and an upper bound, denoted $\hat{0}$ and $\hat{1}$, respectively. Unless stated otherwise, we will assume that our lattices are finite. An element $j$ is \textbf{join-irreducible} if $j\neq\hat{0}$ and whenever $j=x\vee y$ either $j=x$ or $j=y$ holds. \textbf{Meet-irreducible} elements are defined dually. Let $\JI(L)$ and $\MI(L)$ be the sets of join-irreducibles and meet-irreducibles of $L$, respectively.

For $A\subseteq L$, the expression $\bigvee A$ is \textbf{irredundant} if there does not exist a proper subset $A^{\pr}\subsetneq A$ such that $\bigvee A^{\pr}=\bigvee A$. Given $A,B\subseteq JI(L)$ such that $\bigvee A$ and $\bigvee B$ are irredundant and $\bigvee A=\bigvee B$, we set $A\preceq B$ if for $a\in A$ there exists $b\in B$ with $a\leq b$. If $x\in L$ and $A\subseteq\JI(L)$ such that $x=\bigvee A$ is irredundant, we say $x=\bigvee A$ is a \textbf{canonical join-representation} for $x$ if $A\preceq B$ for any other irrendundant join-representation $x=\bigvee B,\ B\subseteq\JI(L)$. Dually, one may define \textbf{canonical meet-representations}.

\begin{figure}
\centering
\includegraphics{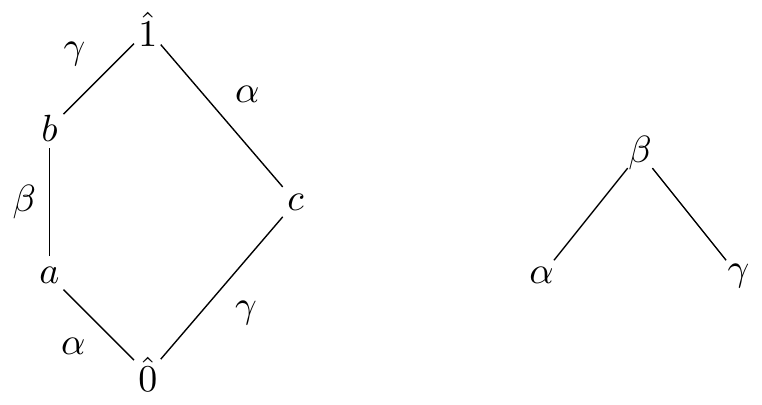}
\caption{\label{fig_culabel}(left) A lattice with a CU-labeling (right) a poset of labels}
\end{figure}

In Figure \ref{fig_culabel}, we define a lattice with 5 elements. The set of join-irreducibles is $\{a,b,c\}$. The top element $\hat{1}$ has two irredundant expressions as a join of join-irreducibles, namely $a\vee c=\hat{1}$ and $b\vee c=\hat{1}$. Since $\{a,c\}\leq\{b,c\}$, the expression $a\vee c=\hat{1}$ is the canonical join-representation for $\hat{1}$.

A lattice $L$ is \textbf{meet-semidistributive} if for any three elements $x,y,z\in L$, $x\wedge z=y\wedge z$ implies $(x\vee y)\wedge z=x\wedge z$. A lattice $L$ is \textbf{semidistributive} if both $L$ and $L^*$ are meet-semidistributive. It is known that a lattice is semidistributive if and only if it has canonical join-representations and canonical meet-representations for each of its elements.

A \textbf{lattice congruence} $\Theta$ is an equivalence relation such that if $x\equiv y\mod\Theta$ then $x\wedge z\equiv y\wedge y\mod\Theta$ and $x\vee z\equiv y\vee z\mod\Theta$ for all $x,y,z\in L$. If $\Theta$ is a lattice congruence of $L$, the set of equivalence classes $L/\Theta$ inherits a lattice structure from $L$. Namely, $[x]\vee[y]=[x\vee y]$ and $[x]\wedge[y]=[x\wedge y]$ for $x,y\in L$. The lattice $L/\Theta$ is called a \textbf{lattice quotient} of $L$, and the natural map $L\ra L/\Theta$ is a \textbf{lattice quotient map}. Although lattice quotients are easiest to describe in algebraic terms, it is often more useful to give the following order-theoretic definition.

\begin{lemma}\label{lem_lattice_congruence_0}
An equivalence relation $\Theta$ on a finite lattice $L$ is a lattice congruence if
\begin{enumerate}
\item every equivalence class of $\Theta$ is a closed interval of $L$, and
\item the maps $x\mapsto\min[x]_{\Theta}$ and $x\mapsto\max[x]_{\Theta}$ are order-preserving.
\end{enumerate}
\end{lemma}

Lemma \ref{lem_lattice_congruence_0} has been proven several times in the literature. For our purposes, it is more convenient to use the following modification; see \cite[Lemma 3.1]{garver2015lattice} or \cite[Lemma 4.2]{dermenjain.hohlweg.pilaud:facial_weak}.

\begin{lemma}\label{lem_lattice_congruence}
Let $L$ be a lattice with idempotent, order-preserving maps $\pi_{\downarrow}:L\ra L,\ \pi^{\uparrow}:L\ra L$. If for $x\in L$
\begin{enumerate}
\item $\pi_{\downarrow}(x)\leq x\leq \pi^{\uparrow}(x)$,
\item $\pi_{\downarrow}(\pi^{\uparrow}(x))=\pi_{\downarrow}(x)$,
\item $\pi^{\uparrow}(\pi_{\downarrow}(x))=\pi^{\uparrow}(x)$,
\end{enumerate}
then the equivalence relation $x\equiv y\mod\Theta$ if $\pi_{\downarrow}(x)=\pi_{\downarrow}(y)$ is a lattice congruence.
\end{lemma}

Given $x,y$ in a poset $P$, we say $y$ \textbf{covers} $x$, denoted $x\lessdot y$, if $x<y$ and there does not exist $z\in P$ such that $x<z<y$. We let $\Cov(P)$ denote the set of all covering relations of $P$. If $P$ is finite, then the partial order on $P$ is the transitive closure of its covering relations. In a finite lattice $L$, if $j\in\JI(L)$, then $j$ covers a unique element $j_*$. Dually, if $m\in\MI(L)$, then $m$ is covered by a unique element $m^*$. It should be clear from context whether $m^*$ is an element of the dual lattice $L^*$ or is the unique element covering a meet-irreducible $m$. We describe the behavior of covering relations under lattice quotients in Lemma~\ref{lem_cong_elementary}. A proof of this lemma may be found in Section 1-5 of \cite{ReadingPAB}.

\begin{lemma}\label{lem_cong_elementary}
Let $L$ be a lattice with a lattice congruence $\Theta$.
\begin{enumerate}
\item The interval $[[x]_{\Theta},[y]_{\Theta}]$ in $L/\Theta$ is isomorphic to the quotient interval $[x,y]/\Theta$.
\item If $(x,y)\in\Cov(L)$, then either $[x]_{\Theta}=[y]_{\Theta}$ or $([x]_{\Theta},[y]_{\Theta})\in\Cov(L/\Theta)$.
\item If $x=\max[x]_{\Theta}$, then for each $[y]_{\Theta}$ with $([x]_{\Theta},[y]_{\Theta})\in\Cov(L/\Theta)$ there exists a unique $y^{\pr}\in[y]_{\Theta}$ with $(x,y^{\pr})\in\Cov(L)$.
\item If $y=\min[y]_{\Theta}$, then for each $[x]_{\Theta}$ with $([x]_{\Theta},[y]_{\Theta})\in\Cov(L/\Theta)$ there exists a unique $x^{\pr}\in[x]_{\Theta}$ with $(x^{\pr},y)\in\Cov(L)$.
\end{enumerate}
\end{lemma}

The set of lattice congruences $\Con(L)$ of a lattice $L$ is partially ordered by refinement. The top element of $\Con(L)$ is the congruence that identifies all of the elements of $L$, whereas the bottom element does not identify any elements of $L$. It is known that $\Con(L)$ is a distributive lattice. By Birkhoff's representation theorem for distributive lattices, $\Con(L)$ is isomorphic to the poset of order-ideals of $\JI(\Con(L))$, where the set of join-irreducibles is viewed as a subposet of $\Con(L)$.

Given $x\lessdot y$ in $L$, let $\con(x,y)$ denote the most refined lattice congruence for which $x\equiv y$. These congruences are join-irreducible, and if $L$ is finite, then every join-irreducible lattice congruence is of the form $\con(j_*,j)$ for some $j\in\JI(L)$ \cite[Theorem 2.30]{freese1995free}. Consequently, there is a natural surjective map of sets $\JI(L)\ra\JI(\Con(L))$ given by $j\mapsto\con(j_*,j)$. Dually, there is a natural surjection $\MI(L)\ra\MI(\Con(L))$ given by $m\mapsto\con(m,m^*)$. If both maps are bijections, then we say $L$ is \textbf{congruence-uniform} (or \textbf{bounded}). Congruence-uniform lattices are the topic of the next section.

\subsection{Congruence-uniform lattices}\label{subsec_cu_lattices}
Given a subset $I$ of a poset $P$, let $P_{\leq I}=\{x\in P:\ (\exists y\in I)\ x\leq y\}$.  If $I$ is a closed interval of a poset $P$, the \textbf{doubling} $P[I]$ of $P$ at $I$ is the induced subposet of $P\times 2$ consisting of the elements in $P_{\leq I}\times\{0\}\sqcup((P\setm P_{\leq I})\cup I)\times\{1\}$. Some doublings are shown in Figure \ref{fig_doubling}. Day proved that a lattice is congruence-uniform if and only if it may be constructed from a 1-element lattice by a sequence of interval doublings \cite{day:congruence}.

\begin{figure}
\centering
\includegraphics{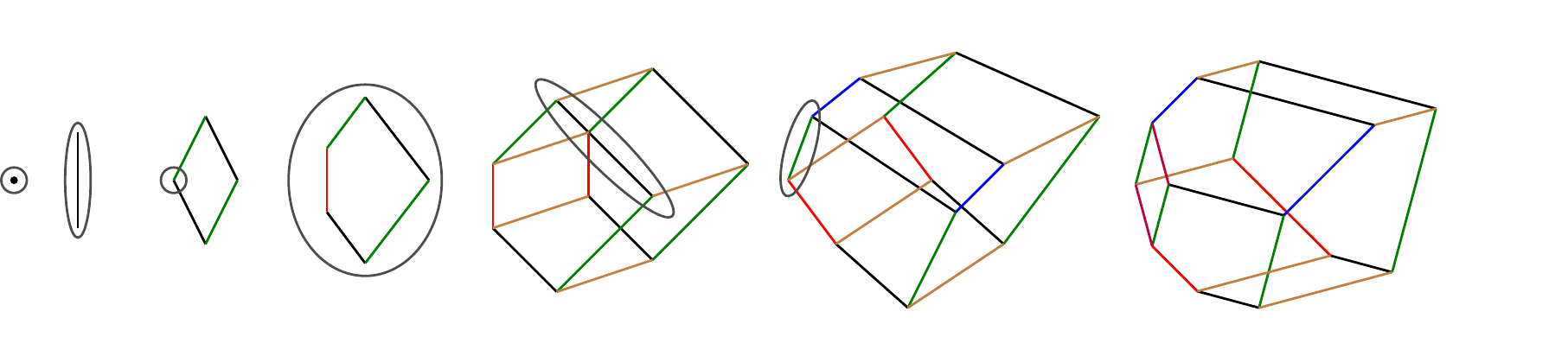}
\caption{\label{fig_doubling}A sequence of interval doublings}
\end{figure}

Let $L$ be a lattice and $P$ a poset. A function $\lambda:\Cov(L)\ra P$ is a \textbf{CN-labeling} of $L$ if $L$ and its dual $L^*$ satisfy the following condition (see \cite{reading:lattice}): For elements $x,y,z\in L$ with $z\lessdot x,\ z\lessdot y$, and maximal chains $C_1,C_2$ in $[z,x\vee y]$ with $x\in C_1$ and $y\in C_2$,
\begin{list}{}{}
\item[(CN1)] the elements $x^{\pr}\in C_1,\ y^{\pr}\in C_2$ such that $x^{\pr}\lessdot x\vee y$ and $y^{\pr}\lessdot x\vee y$ satisfy
$$\lambda(x^{\pr},x\vee y)=\lambda(z,y),\ \lambda(y^{\pr},x\vee y)=\lambda(z,x);$$
\item[(CN2)] if $(u,v)\in\Cov(C_1)$ with $z<u,\ v<x\vee y$ then $\lambda(z,x)<\lambda(u,v)$ and $\lambda(z,y)<\lambda(u,v)$;
\item[(CN3)] the labels on $\Cov(C_1)$ are all distinct.
\end{list}

A lattice is \textbf{congruence-normal} if it has a CN-labeling. Alternatively, a lattice is congruence-normal if it may be constructed from a 1-element lattice by a doubling sequence of \textbf{order-convex sets}; see \cite{reading:lattice}.

\begin{lemma}\label{lem_cn_cong}
Let $L$ be a congruence-normal lattice with CN-labeling $\lambda:\Cov(L)\ra P$.
\begin{enumerate}
\item Let $\Theta$ be a lattice congruence of $L$. Define an edge-labeling $\tilde{\lambda}:\Cov(L/\Theta)\ra P$ by $\tilde{\lambda}([x]_{\Theta},[y]_{\Theta})=\lambda(x,y)$ whenever $(x,y)\in\Cov(L)$ and $x\nequiv y\mod\Theta$. This labeling is well-defined and is a CN-labeling of $L/\Theta$.
\item The restriction of a CN-labeling to an interval $[x,y]$ is a CN-labeling of $[x,y]$.
\end{enumerate}
\end{lemma}

We say $\lambda:\Cov(L)\ra P$ is a \textbf{CU-labeling} if it is a CN-labeling, and
\begin{list}{}{}
\item[(CU1)] $\lambda(j_*,j)\neq\lambda(j_*^{\pr},j^{\pr})$ for $j,j^{\pr}\in\JI(L),\ j\neq j^{\pr}$, and
\item[(CU2)] $\lambda(m,m^*)\neq\lambda(m^{\pr},m^{\pr *})$ for $m,m^{\pr}\in\MI(L),\ m\neq m^{\pr}$.
\end{list}

For example, the colors on the edges of Figure \ref{fig_doubling} form a CU-labeling, where the color set is ordered $s\leq t$ if color $s$ appears before $t$ in the sequence of doublings.

In \cite{reading:lattice}, Reading characterized congruence-normal lattices as those lattices that admit a CN-labeling. From his proof, it is straight-forward to show that a lattice is congruence-uniform if and only if it admits a CU-labeling.

\begin{proposition}
A lattice is congruence-uniform if and only if it admits a CU-labeling.
\end{proposition}

If $x\lessdot y$ and $w\lessdot z$, then covers $(x,y)$ and $(w,z)$ are \textbf{associates} if either $y\wedge w=x$ and $y\vee w=z$ or $x\wedge z=w$ and $x\vee z=y$. Such a notion is useful for lattice congruences. Namely, if $(x,y)$ and $(w,z)$ are associates and $\Theta$ is a lattice congruence, then $x\equiv y\mod\Theta$ if and only if $w\equiv z\mod\Theta$.

For an element $x$, let $\lambda_{\downarrow}(x)=\{\lambda(y,x):\ y\in L, y\lessdot x\}$.  Dually, let $\lambda^{\uparrow}(x)=\{\lambda(x,y):\ y\in L, x\lessdot y\}$.

\begin{lemma}\label{lem_associates}
Let $L$ be a congruence-uniform lattice with CU-labeling $\lambda:\Cov(L)\ra P$. For any $s\in P$, if $j$ is a minimal element with the property $s\in\lambda_{\downarrow}(j)$, then $j$ is a join-irreducible. Moreover, if $(x,y)\in\Cov(L)$ such that $\lambda(x,y)=s$, then $(j_*,j)$ and $(x,y)$ are associates. Conversely, if $(j_*,j)$ and $(x,y)$ are associates, then they have the same label. Dually, if $m$ is a maximal element with the property $s\in\lambda^{\uparrow}(m)$, then $m$ is meet-irreducible, and the cover $(m,m^*)$ is associates with every other cover with the label $s$.
\end{lemma}

\begin{proof}
Let $s\in P$ be given, and let $j$ be minimal such that $s\in\lambda_{\downarrow}(j)$, and let $w\in L$ with $\lambda(w,j)=s$. If $j$ is not join-irreducible, then there exists some $z$ covered by $j$ distinct from $w$. By (CN1), there exists an element $w^{\pr}<j$ such that $\lambda(w\wedge z,w^{\pr})=s$, which is a contradiction to the minimality of $j$. Hence, $j$ is join-irreducible.

Let $x,y\in L$ such that $x\lessdot y$ and $\lambda(x,y)=s$. If $y$ is join-irreducible, then $y=j$ by (CU1). Otherwise, by the previous argument, $(x,y)$ is associates with some cover $(x_1,y_1)$ such that $y>y_1$. Applying this several times, we get a sequence $y>y_1>\cdots>y_N$ and covers $(x_i,y_i)$ such that $(x,y)$ is associates with $(x_i,y_i)$ for all $i$. This terminates if $y_N$ is minimal. But that forces $y_N=j$, so $(x,y)$ is associates with $(j_*,j)$.

Now let $j\in\JI(L)$ and $(x,y)\in\Cov(L)$ such that $(j_*,j)$ and $(x,y)$ are associates. If $y$ is a join-irreducible, then it is clear that $y=j$. Otherwise, we may construct a sequence $(x_i,y_i)\in\Cov(L)$ such that any two covers are associates, $\lambda(x_i,y_i)=\lambda(x,y)$ and $y_1>y_2>\cdots>y_N$ with $y_N\in\JI(L)$. Since associates pairs induce the same lattice congruence, we have $\con(j_*,j)=\con(x_N,y_N)$, so $j=y_N$.

The dual statement may be proved in a similar manner.
\end{proof}

Lemma~\ref{lem_associates} shows that a CU-labeling is essentially unique if it exists. Using the proof, one can construct the following labeling.

\begin{corollary}
If $L$ is a congruence-uniform lattice, then the edge-labeling $\lambda:\Cov(L)\ra\JI(\Con(L))$ where $\lambda(x,y)=\con(x,y)$ is a CU-labeling.
\end{corollary}

It is known that congruence-uniformity is preserved under lattice quotients. 

\begin{corollary}
If $L$ is a (finite) congruence-uniform lattice and $\Theta$ is a lattice congruence, then $L/\Theta$ is congruence-uniform.
\end{corollary}

\begin{proposition}\label{prop_cu_CJR}
Let $L$ be a congruence-uniform lattice with CU-labeling $\lambda$. For $x\in L$, the canonical join-representation of $x$ is $\bigvee_D j$, where $D$ is the set of join-irreducibles such that $\lambda(j_*,j)\in\lambda_{\downarrow}(x)$. Dually, for $x\in L$, the canonical meet-representation of $x$ is $\bigwedge_U m$, where $U$ is the set of meet-irreducibles such that $\lambda(m,m^*)\in\lambda_{\uparrow}(x)$.
\end{proposition}

\begin{proof}
We prove that $x=\bigvee_D j$ is a canonical join-representation of $x$. The dual statement may be proved similarly.

We first show that the equality $x=\bigvee_D j$ holds. For $j\in D$, the pair $(j_*,j)$ is associates with some cover $(c,x)$, so $j<x$. Hence, $\bigvee_D j\leq x$. If they are unequal, then there exists an element $c$ covered by $x$ for which $\bigvee_D j\leq c$. But $(c,x)$ is associates with $(j_*,j)$ for some $j\in D$, which implies $j\nleq c$. Hence, $x=\bigvee_D j$.

Now suppose $\bigvee_D j$ is redundant, and let $j_0\in D$ such that $x=\bigvee_{D\setm\{j_0\}}j$. Let $c_0$ be the element covered by $x$ with $\lambda(c_0,x)=\lambda((j_0)_*,j_0)$. Since $c_0<\bigvee_{D\setm\{j_0\}}j$, there exists $j_1\in D\setm\{j_0\}$ where $c_0\vee j_1=x$. Let $c_1$ be covered by $x$ with $\lambda(c_1,x)=\lambda((j_1)_*,j_1)$. By (CN1), there exists $c_1^{\pr}$ with $c_0\wedge c_1\lessdot c_1^{\pr}\leq c_0$ such that $\lambda(c_0\wedge c_1,c_1^{\pr})=\lambda((j_1)_*,j_1)$. But this means $j_1\leq c_1^{\pr}\leq c_0$ holds, which is a contradiction.

Now let $E\subseteq\JI(L)$ such that $x=\bigvee_E j$ is irredundant, and suppose $D\neq E$. Let $j_0\in D\setm E$, and let $c_0$ be the element covered by $x$ such that $\lambda((j_0)_*,j_0)=\lambda(c_0,x)$. Since $c_0<\bigvee_E j$, there exists $j^{\pr}\in E$ such that $j^{\pr}\nleq c_0$. Since $j^{\pr}\neq j_0$, the cover $(j^{\pr}_*,j^{\pr})$ is not associates with $(c_0,x)$. In particular, $c_0\wedge j^{\pr}<j^{\pr}_*$ holds. Let $a_0$ be an element covering $c_0\wedge j^{\pr}$ with $a_0<j^{\pr}$. Then $a_0\vee c_0=x$, so $(c_0\wedge j^{\pr},a_0)$ and $(c_0,x)$ are associates. This means $j_0\leq a_0<j^{\pr}$. Hence, $D\preceq E$, as desired.
\end{proof}

\begin{lemma}\label{lem_kreweras}
Let $L$ be a congruence-uniform lattice with CU-labeling $\lambda$.  For $x\in L$, there exists a unique element $y$ such that $\lambda^{\uparrow}(x)=\lambda_{\downarrow}(y)$.
\end{lemma}

\begin{proof}

We prove the lemma by induction on $|L|$. If $|L|=1$, then the statement is immediate. If not, let $L^{\pr}$ be a congruence-uniform lattice with interval $I$ such that $L^{\pr}[I]\cong L$. Let $\Theta$ be the lattice congruence whose equivalence classes are the fibers of $L\ra L^{\pr}$. Let $s$ be the label in each $\Theta$-equivalence class.

For $x\in L$, if $x=\max[x]_{\Theta}$, then the upper covers of $x$ in $L$ are in correspondence with the upper covers of $[x]_{\Theta}$ in $L/\Theta$. This correspondence preserves labels. Hence, there is a unique element $[y]_{\Theta}$ in $L/\Theta$ with $\lambda_{\downarrow}([y]_{\Theta})=\lambda^{\uparrow}([x]_{\Theta})$. Taking $y$ to be the minimum element in $[y]_{\Theta}$, we have
$$\lambda_{\downarrow}(y)=\lambda_{\downarrow}([y]_{\Theta})=\lambda^{\uparrow}([x]_{\Theta})=\lambda^{\uparrow}(x).$$

By the uniqueness of $[y]_{\Theta}$, if $y$ is not unique in $L$, then there exists an element $y^{\pr}$ such that $y^{\pr}\neq\min[y^{\pr}]_{\Theta}$. But $s\in\lambda_{\downarrow}(y^{\pr})$ and $s\notin\lambda^{\uparrow}(x)$. Hence, the element $y$ is unique in $L$.

Now let $x$ be an element of $L$ such that $x\neq\max[x]_{\Theta}$. Then the upper covers of $x$ are in correspondence with upper covers of $[x]_{\Theta}$ restricted to the interval $I$ and one additional element, $\max[x]_{\Theta}$. Since $s\in\lambda^{\uparrow}(x)$, any element $y$ with $\lambda_{\downarrow}(y)=\lambda^{\uparrow}(x)$ satisfies $[y]_{\Theta}\in I$ and $y=\max[y]_{\Theta}$. Since $I$ inherits a CU-labeling from $L/\Theta$, there exists a unique element $[y]_{\Theta}$ in $I$ whose lower covers in $I$ have the same labels as the upper covers of $[x]_{\Theta}$ (restricted to $I$). Taking $y=\max[y]_{\Theta}$, we deduce that $\lambda_{\downarrow}(y)=\lambda^{\uparrow}(x)$. The uniqueness of $y$ follows from the uniqueness of $[y]_{\Theta}$.
\end{proof}

We define the \textbf{Kreweras map} $\Kr:L\ra L$ where $\Kr(x)=y$ if $x$ and $y$ are defined as in Lemma \ref{lem_kreweras}. A dual statement to Lemma \ref{lem_kreweras} shows that $\Kr$ is a bijection. A special case of this bijection was originally defined by Kreweras on the lattice of noncrossing set partitions \cite{kreweras:partitions}. Using a standard bijection between noncrossing partitions and bracketings of a word, the bijection defined by Kreweras is equivalent to the Kreweras map on the Tamari order.

Lemma \ref{lem_kreweras} may be restated using Proposition \ref{prop_cu_CJR} to define a bijection $L\ra L$ that switches canonical join-representations with canonical meet-representations. In these terms, this bijection can be shown to exist more generally for semidistributive lattices \cite{barnard}.

\begin{lemma}\label{lem_cu_facial_interval}
Let $L$ be a congruence-uniform lattice with CU-labeling $\lambda:\Cov(L)\ra P$. Let $[x,y]$ be an interval of $L$ for which $y=\bigvee_{i=1}^la_i$ for some elements $a_1,\ldots,a_l$ that cover $x$. Then there exist elements $c_1,\ldots,c_l$ covered by $y$ such that $x=\bigwedge_{i=1}^lc_i$ and $\lambda(x,a_i)=\lambda(c_i,y)$ for all $i$.
\end{lemma}

\begin{proof}
Since the restriction of a CU-labeling to an interval $[x,y]$ is a CU-labeling of $[x,y]$, we may assume $x=\hat{0},\ y=\hat{1}$. Let $U$ be the set of meet-irreducibles $m$ such that $\lambda(m,m^*)\in\lambda^{\uparrow}(\hat{0})$. Then $\hat{0}=\bigwedge_Um$ is a canonical meet-representation. Then $\Kr(\hat{0})=\bigvee_U\kappa(m)$ is a canonical join-representation. But $\{\kappa(m):\ m\in U\}$ is the set of atoms of $L$, so
$$\hat{1}=\Kr(\hat{0})=\bigvee_U\kappa(m)=\bigvee_Aj$$
where $A$ is the set of atoms of $L$. As this is the canonical join-representation of $\hat{1}$, we must have $A=\{a_1,\ldots,a_l\}$, and there exist $c_1,\ldots,c_l$ covered by $y$ with $\lambda(\hat{0},a_i)=\lambda(c_i,\hat{1})$ for all $i$. As each $c_i$ is meet-irreducible, we have $\kappa(c_i)=a_i$ for all $i$. Hence, $x=\bigwedge_{i=1}^lc_i$.
\end{proof}

Given a congruence-uniform lattice $L$, the shard intersection order can be defined from the labeling $\lambda:\Cov(L)\ra S$ as follows. For $x\in L$, let $y_1,\ldots,y_k$ be the set of elements in $L$ such that $(y_i,x)\in\Cov(L)$. Define
$$\psi(x)=\{\lambda(w,z):\ \bigwedge_{i=1}^ky_i\leq w\lessdot z\leq x\}.$$
The \textbf{shard intersection order} $\Psi(L)$ is the collection of sets $\psi(x)$ for $x\in L$, ordered by inclusion. The shard intersection order was defined at this level of generality by Nathan Reading following Theorem 1-7.24 in \cite{ReadingPAB}.

The poset $\Psi(L)$ derives its name from a related construction on hyperplane arrangements. If $\Acal$ is a real, central, simplicial hyperplane arrangement, then the poset of regions with respect to any choice of fundamental chamber is a semidistributive lattice. Each hyperplane is divided into several cones, called \textbf{shards}. The \textbf{shard intersection order} is the poset of intersections of shards, ordered by reverse inclusion. When the poset of regions is a congruence-uniform lattice, the resulting poset is isomorphic to $\Psi(L)$. However, while any shard intersection order coming from a congruence-uniform poset of regions is a lattice, this does not hold for arbitrary congruence-uniform lattices.

\section{The noncrossing complex}\label{Sec:noncrossingcomplex}

In this section, we introduce the noncrossing complex of arcs on a tree. This simplicial complex gives rise to a pure, thin simplicial complex that we refer to as the reduced noncrossing complex. We use the facets of the reduced noncrossing complex to define our main object of study, the oriented flip graph of a tree.

A \textbf{tree} is a finite, connected acyclic graph. Any tree may be embedded in a disk $D^2$ in such a way that a vertex is on the boundary if and only if it is a leaf. Unless specified otherwise, we will assume that any tree comes equipped with such an embedding. We will refer to non-leaf vertices of a tree as \textbf{interior vertices}. We assume that any interior vertex of a tree has degree at least 3. We say two trees $T$ and $T^{\pr}$ to be \textbf{equivalent} if there is an isotopy between the spaces $D^2\backslash T$ and $D^2\backslash T^{\pr}$. 

A tree $T$ embedded in $D^2$ determines a collection of 2-dimensional regions in $D^2$ that we will refer to as \textbf{faces}. A \textbf{corner} of a tree is a pair $(v,F)$ consisting of an interior vertex $v$ and a 2-dimensional face $F$ containing $v$. We let $\Cor(T)$ denote the set of corners of $T$. The embedding that accompanies $T$ also endows each interior vertex with a cyclic ordering. Given two corners $(u,F), (u,G) \in \text{Cor}(T)$, we say that $(u,G)$ is \textbf{immediately clockwise} (resp. \textbf{immediately counterclockwise}) from $(u,F)$ if $F\cap G \neq \emptyset$ and $G$ is clockwise (resp. counterclockwise) from $F$ according to the cyclic ordering at $u$.

An \textbf{acyclic path} (or \textbf{chordless path}) supported by a tree $T$ is a sequence $(v_0,\ldots,v_t)$ of vertices of $T$ such that $v_i$ and $v_j$ are adjacent if and only if $|i-j|=1$. We typically identify acyclic paths with their underlying vertex sets; that is, we do not distinguish between acyclic paths of the form $(v_0,\ldots,v_t)$ and $(v_t,\ldots,v_0)$. We will refer to $v_0$ and $v_t$ as the \textbf{endpoints} of the acyclic path $(v_0,\ldots, v_t)$. Note that an acyclic path is determined by its endpoints, and thus we can write $[v_0,v_t] = (v_0, \ldots, v_t)$. As an acyclic path $(v_0,\ldots,v_t)$ defines a subgraph of $T$ (namely, the induced subgraph on the vertices $v_0, \ldots, v_t$), it makes sense to refer to an \textbf{edge} of $(v_0,\ldots,v_t)$. Additionally, if $(v_0, \ldots, v_t)$ and $(v_t, \ldots, v_{t+s})$ are acyclic paths that agree only at $v_t$ and where $[v_0, v_{t+s}]$ is an acyclic path, we define their \textbf{composition} as $[v_0,v_t] \circ [v_t, v_{t+s}] := [v_0,v_{t+s}]$.

An \textbf{arc} $p = (v_0, \ldots, v_t)$ is an acyclic path whose endpoints are distinct leaves and any two edges $(v_{i-1},v_i)$ and $(v_i,v_{i+1})$ are incident to a common face. We say $p$ \textbf{traverses a corner} or \textbf{contains a corner} $(v,F)$ if $v = v_i$ for for some $i = 0,1, \ldots, t$ and $F$ is the face that is incident to both $(v_{i-1},v_i)$ and $(v_i, v_{i+1})$. Since an arc $p$ divides $D^2$ into two components, it determines two disjoint subsets of the set of faces of $T$ that we will call \textbf{regions}. We let $\text{Reg}(p,F)$ denote the region defined by $p$ that contains the face $F$. 

A \textbf{segment} is an acyclic path consisting of at least two vertices and with the same incidence condition that is required of arcs, but whose endpoints are \emph{not} leaves. Observe that interior vertices of $T$ are not considered to be segments. Since trees have unique geodesics between any two vertices, if the endpoints of a segment or arc are $v,w$, we may denote it by $[v,w]$.

\begin{example}
Let $T$ denote the tree shown in Figure~\ref{arcEx2} and let $p = (7,10,11,12,5)$ be the arc of $T$ shown in blue. The arc $p$ contains the corners $(10, F_2)$, $(11, F_5)$, and $(12, F_5).$ The two regions defined by $p$ are $\text{Reg}(p,F_1) = \{F_1, F_2, F_3, F_6, F_7, F_8\}$ and $\text{Reg}(p,F_4) = \{F_4, F_5\}$. 
\end{example}

\begin{definition}
We say that two arcs $p = (v_0,\ldots, v_{t_1}), q = (w_0, \ldots, w_{t_2})$ are \textbf{crossing} along a segment $s = (u_0, \ldots, u_r)$ if 

$\begin{array}{rl}
i) & \text{each vertex of $s$ appears in $p$ and in $q$ and}\\
ii) & \text{if $\text{R}_p$ and $\text{R}_q$ are regions defined by $p$ and $q$, respectively, then $\text{R}_p \not\subset \text{R}_q$ and $\text{R}_q \not \subset \text{R}_p.$}\\
\end{array}$

\noindent We say they are \textbf{noncrossing} otherwise. The \textbf{noncrossing complex} $\Delta^{NC}(T)$ is defined to be the abstract simplicial complex whose simplices are pairwise noncrossing collections of arcs supported by the tree $T$.
\end{definition}

\begin{example}
Let $T$ denote the tree shown in Figure~\ref{arcEx2}. Let $p = (7,10,11,12,5)$ and $q = (6, 10, 11, 9, 1)$ denote the arcs of $T$ shown in blue and red, respectively. The arcs $p$ and $q$ cross along the segment $s = (10,11)$ shown in purple. 

\end{example}
\begin{figure}[h]
\includegraphics[scale=1]{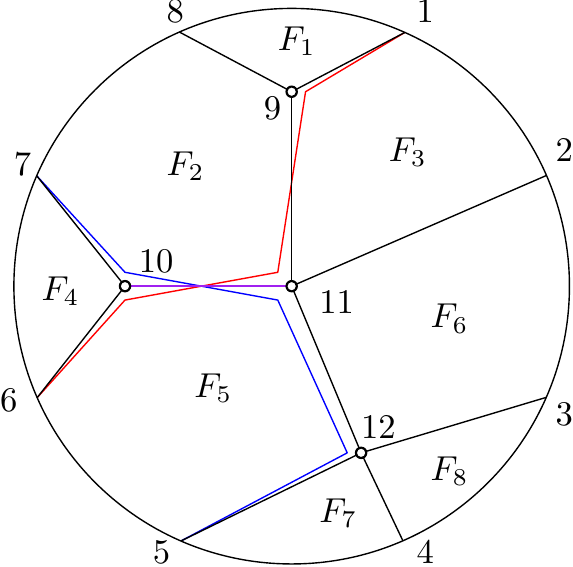}
\caption{}
\label{arcEx2}
\end{figure}

\begin{example}\label{Example_assoc}
If every internal vertex of $T$ has degree $3$, then $\wtil{\Delta}^{NC}(T)$ is isomorphic to the dual associahedron. By this identification, our notion of performing a flip on a facet of the reduced noncrossing complex of $T$ translates into the well-known operation of performing a \textbf{diagonal flip} on the corresponding triangulation (see Figure~\ref{A2associahedron}).
\end{example}

\begin{figure}[h]
\includegraphics[scale=1]{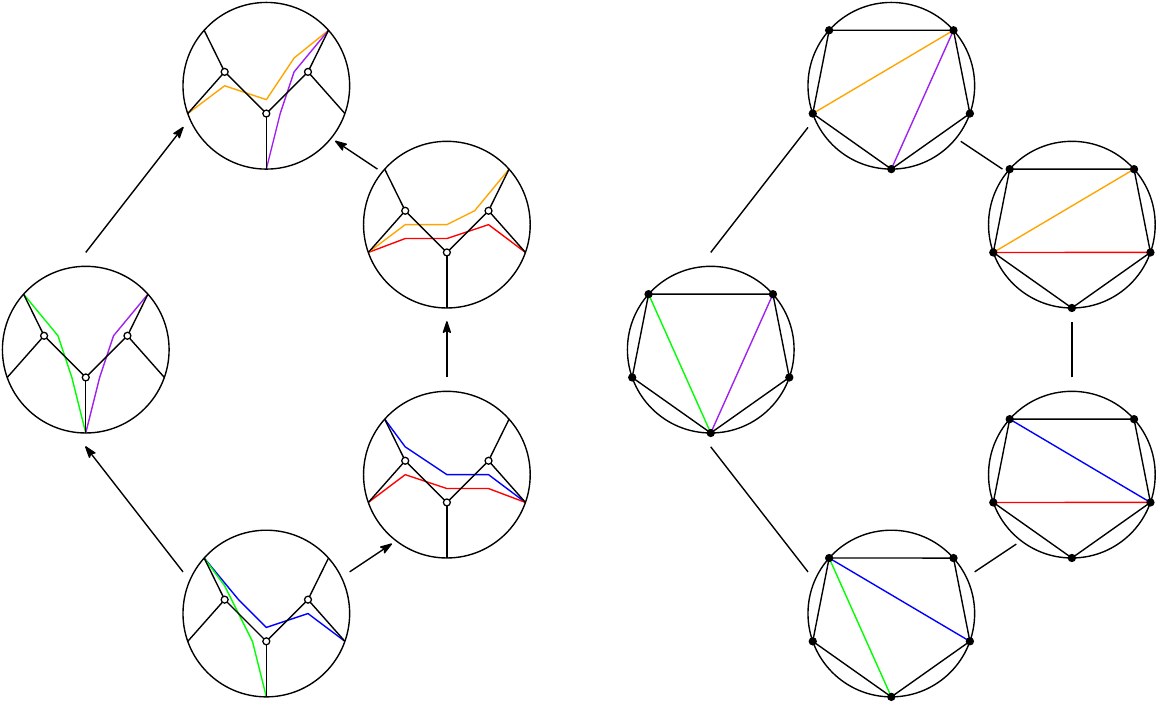}
\caption{The oriented flip graph and the triangulations corresponding to each facet of the reduced noncrossing complex.}
\label{A2associahedron}
\end{figure}

If $p$ is an arc whose vertices all lie on a common face, then $p$ is noncrossing with every arc supported by $T$. We call such an arc a \textbf{boundary arc}. Observe that boundary arcs are exactly those arcs that define a region consisting of a single face. This implies that the faces of $T$ are in bijection with boundary arcs of $T$. Using this fact, at times we will refer to the boundary arc corresponding to a given face. The \textbf{reduced noncrossing complex} $\wtil{\Delta}^{NC}(T)$ is the abstract simplicial complex consisting of the faces of $\Delta^{NC}(T)$ containing no boundary arcs.

We now introduce a partial ordering on arcs that contain a particular corner of $T$. This partial ordering enables us to understand the the combinatorial structure of the noncrossing complex and the reduced noncrossing complex of $T$. Let $\Fcal$ be a face of $\Delta^{NC}(T)$ and let $(v,F)$ be a corner that is contained in at least one arc of $\mathcal{F}$. The arcs of $\mathcal{F}$ that contain $(v,F)$ are partially ordered in the following way: $p \le_{(v,F)} q$ if and only if $\text{Reg}(p,F) \subset \text{Reg}(q,F)$.

\begin{lemma}\label{arc_order_lemma}
If $\mathcal{F}$ is a face of $\Delta^{NC}(T)$ and $(v,F)$ is a corner contained in at least one arc of $\mathcal{F}$, then the partially ordered set $(\{p \in \mathcal{F}: \text{ $p$ contains $(v,F)$}\}, \le_{(v,F)})$ is a linearly ordered set. 
\end{lemma}
\begin{proof}
Since any pair $p_1, p_2 \in \{p \in \mathcal{F}: \text{ $p$ contains $(v,F)$}\}$ are noncrossing and since each $p_i$ defines a region that contains $F$, one has that $p_1 \le_{(v,F)} p_2$ or $p_2 \le_{(v,F)} p_1.$ Thus $(\{p \in \mathcal{F}: \text{ $p$ contains $(v,F)$}\}, \le_{(v,F)})$ is a linearly ordered set.
\end{proof}

It follows from Lemma~\ref{arc_order_lemma} that the partially order set $(\{p \in \mathcal{F}: \text{ $p$ contains $(v,F)$}\}, \le_{(v,F)})$ has a unique maximal element, which we will denote by $p(v,F).$ We say that an arc $p$ of $\mathcal{F}$ is \textbf{marked} at $(v,F)$ if $p = p(v,F).$

The following proposition enables us to show that the simplicial complex $\widetilde{\Delta}^{NC}(T)$ is a \textbf{pure} (i.e. any two facets have the same cardinality) and \textbf{thin} (i.e. every codimension 1 simplex is a face of exactly two facets) in Corollary~\ref{Cor_pure_thin}. 

\begin{proposition}\label{nccomplexprop}
Let $\Fcal$ be a face of $\Delta^{NC}(T)$, let $p\in\Fcal$, and let $\text{Reg}_1, \text{Reg}_2$ denote the regions defined by $p$.
\begin{enumerate}
\item The arc $p$ is marked at some corner of $T$.
\item In $p$ is not a boundary arc, then $p$ is marked at a corner in $\text{Reg}_1$ and at a corner in $\text{Reg}_2$.
\item Assume that $p$ is marked at two distinct corners $(v,F), (w,G) \in \text{Cor}(T)$ and that $F$ and $G$ belong to the same region defined by $p$. Then there exists an arc $p^\prime \not \in \mathcal{F}$ that contains $(v,F)$ and $(w,G^\prime)$ where $G^\prime \neq G$ and where $\mathcal{F}\cup \{p^\prime\} \in \Delta^{NC}(T).$ 
\item If $\Fcal$ is a facet and $p\in\Fcal$ is not a boundary arc, then there exists a unique arc $q\notin\Fcal$ such that $(\Fcal\backslash\{p\})\cup\{q\}$ is a facet. Moreover, if $p$ is marked at two distinct corners $(v,F), (u,G) \in \text{Cor}(T)$, then $[v,u]$ is the unique longest segment along which $p$ and $q$ cross.
\end{enumerate}
\end{proposition}

\begin{proof}
(1) Let $(v,F) \in \text{Cor}(T)$ be a corner contained in $p$. If $p=p(v,F)$, then we are done. Otherwise, let $q \in \mathcal{F}$ be the arc containing $(v,F)$  such that $p \lessdot_{(v,F)} q.$ Let $w$ be an interior vertex at which $p$ and $q$ separate, let $(w,G)$ be the corner traversed by $p$ at $w$, and let $p^\prime = p(w,G) \in \mathcal{F}$. Since $p^{\pr}$ and $q$ are noncrossing and $p\leq_{(w,G)}p^{\pr}$, $p^{\pr}$ must contain the corner $(v,F)$ and $G \in \text{Reg}(p, F)$. Now this implies $p \le_{(v,F)} p^\prime$ so $p \le_{(v,F)} p^\prime <_{(v,F)} q$. Thus $p = p^\prime$. 

(2) In the proof of (1), we showed that if $p$ contains a corner $(w_i,G_i)$ with $G_i \in \text{Reg}_i$, then there exists a corner $(v_i,F_i)$ with $F_i \in \text{Reg}_i$ such that $p = p(v_i, F_i)$. If $p$ is not a boundary arc, then it contains such a corner $(w_i,G_i)$ with $G_i \in \text{Reg}_i$ for $i = 1, 2.$ 

(3) Assume that $p$ contains two distinct corners $(v,F), (w,G) \in \text{Cor}(T)$ where $p = p(v,F)$ and $p = p(w,G)$ and where $F$ and $G$ belong to the same region defined by $p$. Let $G^{\prime}$ be the face containing $w$ such that $G\cap G^{\pr}$ is an edge of the segment $[v,w]$. We can assume that at least one arc of $\mathcal{F}$ contains $(w,G^\prime) \in \text{Cor}(T)$, otherwise define $p^\prime$ to be the boundary arc corresponding to $G^\prime$ and we obtain that $\mathcal{F}\cup \{p^\prime\} \in \Delta^{NC}(T).$ 

Let $q := p(w,G^{\pr}) \in \mathcal{F}.$ The arc $p$ is expressible as the composition $p=[v_0,v]\circ[v,w]\circ[w,w_0]$.  Similarly, $q$ is the composition $q=[v_1,w]\circ[w,w_1]$ where $[w,w_1]$ and $p$ do not agree along any edges. Let $p^{\prime}$ be the arc $p^{\prime} := [v_0,w]\circ[w,w_1]$. Clearly, $p^\prime$ and $p$ do not cross. 

Next, we show that $\mathcal{F} \cup \{p^\prime\} \in \Delta^{NC}(T).$ Let $q^\prime \in \mathcal{F}$ and suppose that $q^\prime$ and $p^\prime$ cross along a segment $s.$ It is enough to assume that $s$ is contained in either $[v_0,w]$ or $[w,w_1]$. If $s$ is contained in $[v_0,w]$, then since $p$ and $p^\prime$ agree along $[v_0,w]$ we have that $q^\prime$ and $p$ cross along $s$, a contradiction. Similarly, $q^\prime$ and $p^\prime$ cannot cross along a segment $s$ contained in $[w,w_1]$. We conclude that $\mathcal{F} \cup \{p^\prime\} \in \Delta^{NC}(T).$

(4) By $(2)$, there exist distinct corners $(v_1, F_1), (v_2, F_2) \in \text{Cor}(T)$ contained in $p$ where $F_i \in \text{Reg}_i$ and such that $p = p(v_i, F_i)$ for $i = 1, 2$. Let $p_1$ and $p_2$ be arcs of $\mathcal{F}\backslash\{p\} \in \Delta^{NC}(T)$ where $p_1$ and $p_2$ are marked at $p(v_1,F_1)$ and $p(v_2,F_2)$, respectively, with respect to the other arcs of $\mathcal{F}\backslash\{p\}$. Since $\mathcal{F}$ is a facet, it contains each boundary arc. As $p$ is not a boundary arc, there does exist the desired arcs $p_1$ and $p_2$ in $\mathcal{F}\backslash\{p\}$. 

\begin{figure}[h]
$$\begin{array}{cccccccccccccc}
\includegraphics[scale=2]{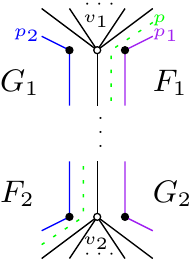} & & \includegraphics[scale=2]{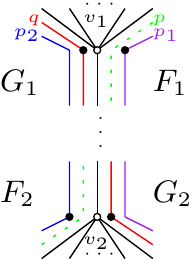}\\
(a) & & (b)
\end{array}$$
\caption{In (a), we show part of the face $\mathcal{F}\backslash\{p\}$ and we indicate corners at which $p_1$ and $p_2$ are marked by black dots. In (b), we show the part of the face $\mathcal{F}\backslash\{p\}\cup \{q\}$ and we indicate corners at which $p_1, p_2,$ and $q$ are marked by black dots. In (a) and (b), we indicate where the arc $p$ appeared before it was removed.}
\label{flipconfig}
\end{figure}

\begin{lemma}\label{fliplemma} In the face $\mathcal{F}\backslash\{p\}$, $p_1 = p(v_2,G_2)$ and $p_2 = p(v_1,G_1)$ where $G_i$ is the unique face of the tree $T$ such that $(v_i, G_i)$ is immediately clockwise from $(v_i, F_i)$ (see Figure~\ref{flipconfig}). 
\end{lemma}

\begin{proof}[Proof of Lemma~\ref{fliplemma}]
We show that $p_1 = p(v_2,G_2)$ and the proof that $p_2 = p(v_1,G_1)$ is similar. Write $p_1 = s_1\circ [v_1,w_1]$ and $p(v_2, G_2) = s_2 \circ [v_2,w_2]$ where $s_1, [v_1,w_1], s_2,$ and $[v_2,w_2]$ are acyclic paths of $T$, $w_1$ and $w_2$ are leaf vertices of $T$, and where we require that $[v_1,w_1]$ and $[v_2,w_2]$ each contain part of the segment $[v_1,v_2]$.

Now consider the arc $p^\prime := s_1\circ [v_1,v_2] \circ s_2$. Since $p_1$ (resp. $p(v_2,G_2)$) does not cross any arcs of $\mathcal{F}$ along $s_1$ (resp. $s_2$), the same is true for $p^\prime$. Similarly, $p$ does not cross any arcs of $\mathcal{F}$ along $[v_1, v_2]$ so the same is true for $p^\prime$. As $\mathcal{F}$ is a facet of $\Delta^{NC}(T)$, we have that $p^\prime \in \mathcal{F}$. Now it is clear that $p^\prime = p(v_1, F_1)$ and $p^\prime = p(v_2, G_2)$, and the result follows. \end{proof}

Next, let $p_1 = s_1 \circ [v_2,w_1]$ and let $p_2 = [w_2,v_1] \circ s_2$ for some acyclic paths $s_1$ and $s_2$ and some leaf vertices $w_1$ and $w_2$ of $T$. Define $q := [w_2,v_1] \circ [v_1,v_2] \circ [v_2, w_1] \neq p.$ By Lemma~\ref{fliplemma} and the proof of Proposition~\ref{nccomplexprop} $(3)$, we have that $(\mathcal{F}\backslash\{p\})\cup \{q\} \in \Delta^{NC}(T).$ Furthermore, it is clear that $q = p(v_1, G_1) = p(v_2, G_2)$ in $(\mathcal{F}\backslash \{p\}) \cup \{q\}$ and that $[v_1,v_2]$ is the unique longest segment along which $p$ and $q$ cross.

Next, we show that $\mathcal{F}$ and $(\mathcal{F}\backslash\{p\})\cup \{q\}$ are the unique faces of $\Delta^{NC}(T)$ that contain $\mathcal{F}\backslash\{p\}$. Note that from this it also follows that $(\mathcal{F}\backslash\{p\})\cup \{q\}$ is a facet of $\Delta^{NC}(T)$. Suppose there exists an arc $p^\prime \not \in \mathcal{F}\backslash \{p\}$ such that $(\mathcal{F}\backslash \{p\}) \cup \{p^\prime\}$ is a facet. Then $p^\prime = p(v_2, F_2) = p(v_1, F_1)$ or $p^\prime = p(v_2, G_2) = p(v_1, G_1)$, otherwise by combining Proposition~\ref{nccomplexprop} (3) and Lemma~\ref{fliplemma} we have that $(\mathcal{F}\backslash \{p\}) \cup \{p^\prime\}$ is not a facet. In particular, we obtain that $p^\prime$ contains the segment $[v_1,v_2]$. The following lemma shows that if $p^\prime = p(v_2, F_2) = p(v_1, F_1)$ (resp. $p^\prime = p(v_2, G_2) = p(v_1, G_1)$), then $p^\prime = p$ (resp. $p^\prime = q$). This establishes the uniqueness of $p$ and $q$.

\begin{lemma}\label{agreefromboundary}
Let $p = [u_2,v_1] \circ [v_1,v_2] \circ [v_2, u_1]$. 
\begin{itemize}
\item[i)] If $p^\prime$ contains the corner $(v_2, F_2)$, then $p^\prime$ and $p$ agree along $[u_1,v_2]\circ [v_2,v_1].$
\item[ii)] If $p^\prime$ contains the corner $(v_1, F_1)$, then $p^\prime$ and $p$ agree along $[u_2,v_1]\circ [v_1,v_2].$
\item[iii)] If $p^\prime$ contains the corner $(v_2, G_2)$, then $p^\prime$ and $q$ agree along $[w_1,v_2]\circ [v_2,v_1].$
\item[iv)] If $p^\prime$ contains the corner $(v_1, G_1)$, then $p^\prime$ and $q$ agree along $[w_2,v_1]\circ [v_1,v_2].$
\end{itemize}
\end{lemma}
\begin{proof}[Proof of Lemma~\ref{agreefromboundary}]
We prove part $i)$, and the proofs of the other parts are analogous. Suppose there exists an interior vertex $x \in [u_1,v_2]$ where $p$ and $p^\prime$ separate. Let $(x,H) \in \text{Cor}(T)$ be the corner contained in $p$. Since $p = p(v_1, F_1) = p(v_2, F_2)$ in $\mathcal{F}$ and since $\mathcal{F}$ is a facet, there exists an arc $a \in \mathcal{F}\backslash\{p\}$ where $a = p(x,H)$ in $\mathcal{F}$. There are two cases: $H \in \text{Reg}(p, F_2)$ or $H \in \text{Reg}(p, F_1)$ (see Figure~\ref{separateconfig}).

Without loss of generality, we assume $H \in \text{Reg}(p, F_2)$. If $a$ contains $(v_2,F_2)$, then $\text{Reg}(p,F_2) = \text{Reg}(p,H) \subsetneq \text{Reg}(a, H) = \text{Reg}(a, F_2)$, contradicting that $p = p(v_2, F_2)$ in $\mathcal{F}$. Thus $a$ does not contain $(v_2,F_2)$. This implies that there exists $y \in [x,v_2]$ such that $p$ and $a$ separate at $y$. Since $a$ and $p$ are noncrossing and since $p <_{(x,H)} a$, any edge of $a$ that is not an edge of $p$ is only incident to faces in $\text{Reg}(p, F_1)$. We conclude that $p^\prime$ and $a$ cross along $[x,y]$, a contradiction.
\end{proof}\end{proof}

\begin{figure}[h]
$$\begin{array}{cccccccccc}
\includegraphics[scale=2]{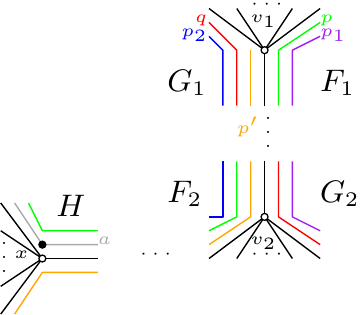} & & \includegraphics[scale=2]{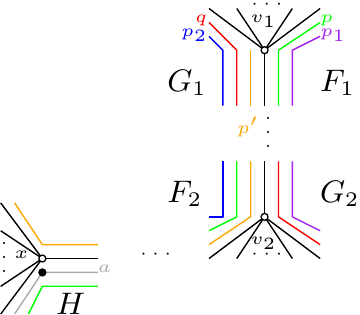}\\
(a) & & (b)
\end{array}$$
\caption{The configuration of arcs in the setting of Lemma 3.8 $i)$. Note that in this situation we do not know if $p^\prime$ contains $(v_1,F_2)$ or $(v_1,F_1)$, which is why it appears to terminate at $v_1$ in $(a)$ and $(b)$. The arc $a = p(x,H)$ has the property that $H \in \text{Reg}(p,F_2)$ or $H \in \text{Reg}(p,F_1)$. We indicate that $a$ is marked at corner $(x,H)$ by marking it with a black dot in $(a)$ and $(b)$.}
\label{separateconfig}
\end{figure}

In the proof of Proposition~\ref{nccomplexprop} (4), we explained how for a given facet $\mathcal{F} \in {\Delta}^{NC}(T)$ and a given arc $p \in \mathcal{F}$ that is not a boundary arc there is a unique way to produce another facet of ${\Delta}^{NC}(T)$. To summarize our construction, suppose that $p = [u_1,u]\circ [u,v] \circ [v,v_1]$ in a facet $\mathcal{F}$ is a nonboundary arc of $T$ where $p = p(u,F)$ and $p = p(v,G)$ are the unique corners of $T$ where $p$ is maximal. Then there is a unique nonboundary arc $q$ such that $(\mathcal{F}\backslash\{p\}) \cup \{q\}$ is a facet of $\Delta^{NC}(T)$. The arc $q = [u_2,u] \circ [u,v] \circ [v,v_2]$ for some leaf vertices $u_2$ and $v_2$ so that $q = p(u,F^\prime)$ and $q = p(v,G^\prime)$ where the vertices of $F\cap F^\prime$ and $G\cap G^\prime$ are contained in both $p$ and $q$. 

\begin{example}
Figure~\ref{A1orflipgraph} shows an example of the construction in Proposition~\ref{nccomplexprop} (4) for the tree $T$ (the tree depicted in black). A black dot appears in an arc if it is the largest arc containing the corresponding corner in that facet. The boundary arcs of $T$ are $(1,5,2), (1,5,4), (2,6,3), \text{ and } (3,4,6)$. These appear in gold. Flipping the green arc produces the red arc.
\end{example}

\begin{figure}[h]
\includegraphics[scale=1.2]{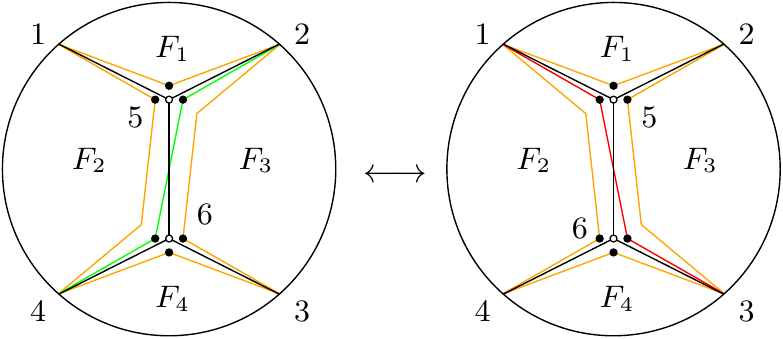}
\caption{The two facets of $\Delta^{NC}(T)$.}
\label{A1orflipgraph}
\end{figure}

\begin{corollary}\label{Cor_pure_thin}
The simplicial complex $\wtil{\Delta}^{NC}(T)$ is pure and thin. 
\end{corollary}

\begin{proof}
Any facet $\mathcal{F}$ of $\Delta^{NC}(T)$ has $\#\mathcal{F} = \#\{\text{nonboundary arcs of $\mathcal{F}$}\} + \#\{\text{boundary arcs of $\mathcal{F}$}\}.$ Note that $\#\{\text{boundary arcs of $\mathcal{F}$}\} = \#\{\text{faces of $T$}\}$. Thus to show $\wtil{\Delta}^{NC}(T)$ is pure, it is enough to prove that $\Delta^{NC}(T)$ is pure. 

Assume $\mathcal{F} \in \Delta^{NC}(T)$ is a facet. Each corner of $T$ is contained in a boundary arc of $\mathcal{F}$, and thus each corner of $T$ has a unique maximal arc containing it. Since $\mathcal{F}$ is a facet, by Proposition~\ref{nccomplexprop} (1), each boundary arc is maximal at exactly one corner of $T$. Similarly, since $\mathcal{F}$ is a facet, by Proposition~\ref{nccomplexprop} (2), each nonboundary arc of $T$ is maximal at exactly two corners of $T$. This implies that $$\begin{array}{rcl}
\#\text{Cor}(T) & = & \#\{\text{boundary arcs in $\mathcal{F}$}\} + 2\#\{\text{nonboundary arcs in $\mathcal{F}$}\}\\
& = & \#\{\text{faces of $T$}\} + 2\#\{\text{nonboundary arcs in $\mathcal{F}$}\}.
\end{array}$$

\noindent Thus $\#\{\text{nonboundary arcs in $\mathcal{F}$}\} = \frac{1}{2}\left(\#\text{Cor}(T) - \#\{\text{faces of $T$}\}\right)$. As the latter number is independent of $\mathcal{F}$, we have that $\Delta^{NC}(T)$ is pure and thus so is $\wtil{\Delta}^{NC}(T).$ 

The simplicial complex $\wtil{\Delta}^{NC}(T)$ is thin because the move between facets of $\Delta^{NC}(T)$ described in Proposition~\ref{nccomplexprop} (4) only involves nonboundary arcs.  \end{proof}


We refer to the operation $\mathcal{F} \longmapsto (\mathcal{F}\backslash\{p\}) \cup \{q\}$ sending facet $\mathcal{F}$ of $\widetilde{\Delta}^{NC}(T)$ to a new facet of $\widetilde{\Delta}^{NC}(T)$ as a \textbf{flip} of $\mathcal{F}$ at $p$ (see Figure~\ref{A1orflipgraph}) and denote it by $\mu_p$. We define the \textbf{flip graph} of $T$, denoted $FG(T)$, to be the graph whose vertices are facets of $\widetilde{\Delta}^{NC}(T)$ and such that two vertices are connected by an edge if and only if the corresponding facets can be obtained from each other by a single flip. 

We now define the following object, which is fundamental to our work in this paper.

\begin{definition}\label{orflipgraph}
Let $\mathcal{F}_1, \mathcal{F}_2 \in {FG}(T)$ and assume that $\mathcal{F}_1$ and $\mathcal{F}_2$ are connected by an edge in $FG(T)$. Let $\mathcal{F}_2 = \mu_p\mathcal{F}_1$ where $\mathcal{F}_2 = \mathcal{F}_1\backslash\{p\} \cup \{q\}$. If $p = p(u,F) = p(v,G)$ and $q = p(u,F^\prime) = p(v,G^\prime)$, we orient the edge connecting $\mathcal{F}_1$ and $\mathcal{F}_2$ so that $\mathcal{F}_1 \longrightarrow \mathcal{F}_2$ if the corner $(u,F^\prime)$ (resp. $(v,G^\prime)$) is immediately clockwise from the corner $(u,F)$ (resp. $(v,G)$) about vertex $u$ (resp. $v$). Otherwise, we orient the edge so that $\mathcal{F}_2 \longrightarrow \mathcal{F}_1$. We refer to the resulting directed graph as the \textbf{oriented flip graph} of $T$ and denote it by $\overrightarrow{FG}(T).$ 

Additionally, any edge of $\overrightarrow{FG}(T)$ connecting $\mathcal{F}$ and $\mu_p\mathcal{F}$ is naturally labeled by the segment determined by the marked corners of $p$ in $\mathcal{F}$ (or in $\mu_p\mathcal{F}$).
\end{definition}

\begin{example}\label{orflipgraphex1}
In Figure~\ref{3dimorflipgraph}, we show the oriented flip graph (without edge labels) of the tree $T$ from Figure~\ref{arcEx2}. 
\end{example}

\begin{figure}[h]
\includegraphics[scale=1]{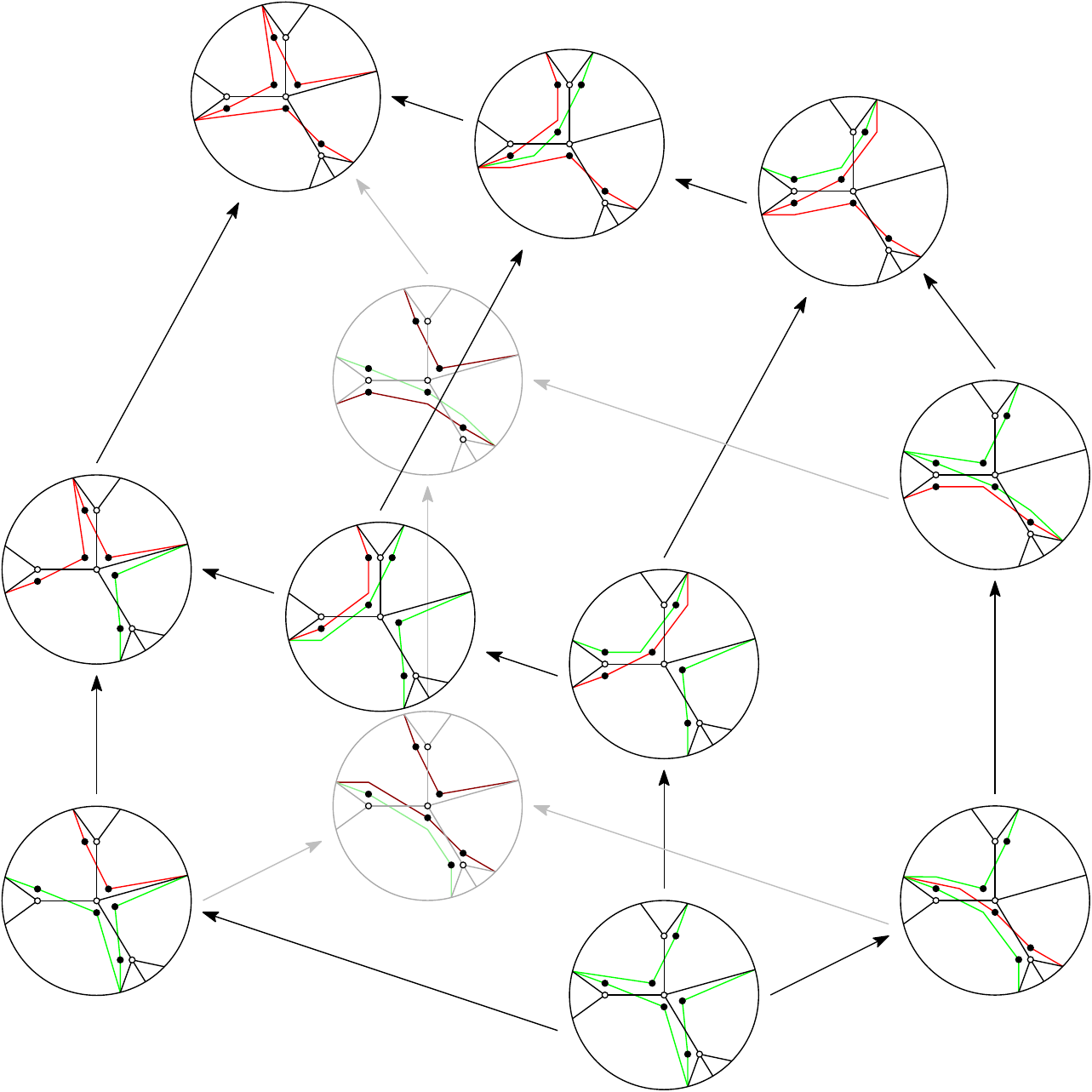}
\caption{An example of an oriented flip graph.}
\label{3dimorflipgraph}
\end{figure}

\section{Sublattice and quotient lattice description of the oriented flip graph}\label{Sec:Sub_and_Quot}

In this section, we identify the oriented flip graph $\ora{FG}(T)$ as both a sublattice and quotient lattice of another lattice. In Section~\ref{subsec_biclosed} we define a closure operator on segments, and introduce a poset of biclosed sets of segments, denoted $\Bic(T)$. It was shown in \cite{garver2015lattice} that $\Bic(T)$ is a congruence-uniform lattice. We define a distinguished lattice congruence $\Theta$ on $\Bic(T)$.

In Section~\ref{subsec_map}, we define maps $\eta:\Bic(T)\ra\ora{FG}(T)$ and $\phi:\ora{FG}(T)\ra\Bic(T)$. The map $\eta$ is a surjective lattice map such that $\eta(X)=\eta(Y)$ exactly when $X\equiv Y\mod\Theta$. The map $\phi$ is a lattice map such that $\eta\circ\phi$ is the identity on $\ora{FG}(T)$. Since congruence-uniformity and polygonality are preserved by lattice quotient maps, we deduce that $\ora{FG}(T)$ is a congruence-uniform and polygonal lattice.

\subsection{Biclosed collections of segments}\label{subsec_biclosed}

Let $\Seg(T)$ be the set of segments supported by a tree $T$. For $X\subseteq\Seg(T)$, we say $X$ is \textbf{closed} if for segments $s,t\in\Seg(T)$, if $s,t\in X$ and $s\circ t\in\Seg(T)$ then $s\circ t\in X$. If $X$ is any subset of $\Seg(T)$, its \textbf{closure} $\ov{X}$ is the smallest closed set containing $X$. Say $X$ is \textbf{biclosed} if $X$ and $\Seg(T)\backslash X$ are both closed. For example, the collection of red segments in the left part of Figure \ref{fig_eta} is biclosed. We let $\Bic(T)$ denote the poset of biclosed subsets of $\Seg(T)$, ordered by inclusion.

Let $Q$ be the graph whose vertices are the edges between interior vertices of $T$, where $e$ and $e^{\pr}$ are adjacent in $Q$ if they meet at a corner $(v,F)$. Later, we will give $Q$ an orientation and view it as a quiver. An \textbf{acyclic path} (or \textbf{chordless path}) of $Q$ is a sequence of vertices $(v_0,\ldots,v_t)$ such that $v_i$ and $v_j$ are adjacent if and only if $|i-j|=1$. We view acyclic paths as undirected, so they are determined by the set of vertices they visit.


A segment of $T$ is naturally regarded as an acyclic path of $Q$. The set of segments of $T$ thus forms some of the acyclic paths of $Q$.  In Theorem 5.4 of \cite{garver2015lattice}, we proved that the set of biclosed subsets of acyclic paths of $Q$ under inclusion forms a congruence-uniform, semidistributive, and polygonal lattice. By a minor modification of the proof, this can be shown to hold for biclosed subsets of any order ideal of acyclic paths, where paths are ordered by inclusion. As $\Seg(T)$ is naturally regarded as an order ideal of acyclic paths of $Q$, we deduce the following result.

\begin{theorem}\label{thm_biclosed_main}
The poset $\Bic(T)$ is a semidistributive, congruence-uniform, and polygonal lattice. Furthermore:
\begin{enumerate}
\item\label{thm_biclosed_main_ss} For $X,Y\in\Bic(T)$, if $X\subsetneq Y$ then there exists $y\in Y\setm X$ with $X\cup\{y\}\in\Bic(T)$.
\item\label{thm_biclosed_main_join} For $W,X,Y\in\Bic(T)$ with $W\subseteq X\cap Y$, the set $W\cup\ov{(X\cup Y)\setm W}$ is biclosed.
\item\label{thm_biclosed_main_cn} The edge-labeling $\lambda:\Cov(\Bic(T))\ra\Seg(T)$ where $\lambda(X,Y)=s$ if $Y\setm X=\{s\}$ is a CN-labeling.
\end{enumerate}
\end{theorem}

A lattice of biclosed sets of segments is given in Figure \ref{fig_bicfig} (see also the upper lattice in Figure 7 in \cite{garver2015lattice}). To simplify the figure, we only show the edges of the tree connecting two interior vertices in Figure \ref{fig_bicfig}.  The Hasse diagram of this lattice is the skeleton of a zonotope with 26 vertices. Although one can find examples where $\text{Bic}(T)$ is isomorphic to the weak order on permutations, Figure \ref{fig_bicfig} shows this is not true for all trees $T$.

\begin{figure}
\centering
\includegraphics{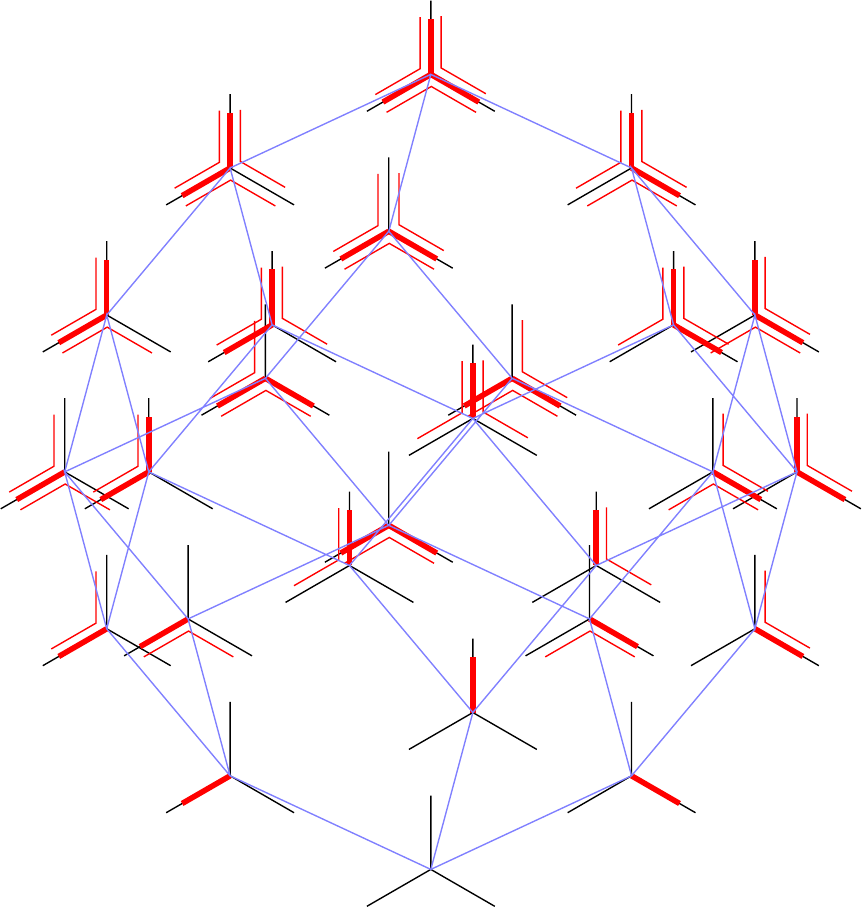}
\caption{\label{fig_bicfig}A lattice of biclosed sets of segments.}
\end{figure}

Any subset $S^{\pr}$ of a closure space $S$ inherits a closure operator $X\mapsto(\ov{X}\cap S^{\pr})$. In general, biclosed subsets of $S^{\pr}$ may not be biclosed as subsets of $S$. For spaces of segments, some intervals of $\Bic(T)$ are isomorphic to $\Bic(S^{\pr})$ for some subset $S^{\pr}$ of segments. We state this precisely as the following proposition.

\begin{proposition}\label{prop_biclosed_facial_interval}
Let $W \subset \text{Seg}(T)$ be a biclosed set of segments, and let $s_1,\ldots,s_k\in\Seg(T)\setm W$ such that $W\cup\{s_i\}$ is biclosed for all $i$. Let $(B_1,\ldots,B_l)$ be the finest partition on $\{s_1,\ldots,s_k\}$ such that if $s_i\circ s_j$ is a segment then $s_i$ and $s_j$ lie in the same block. Then the interval $[W,W\cup\ov{\{s_1,\ldots,s_k\}}]$ is isomorphic to $\Bic(\ov{B_1})\times\cdots\times\Bic(\ov{B_l})$.
\end{proposition}

\begin{proof}
We first prove that the sets $W,\ov{B_1},\ldots,\ov{B_l}$ are all disjoint. Suppose $W\cap\ov{B_i}$ is nonempty for some $i$, and let $t\in W\cap\ov{B_i}$ be of minimum length. Since $s_j\notin W$ for all $j$, $t$ must be a concatenation $t_1\circ t_2$ of elements of $\ov{B_i}$. By minimality, $t_1$ and $t_2$ are not in $W$. But $W$ is co-closed, a contradiction.

Now suppose there are two blocks, say $B_1,B_2$, such that $\ov{B_1}\cap\ov{B_2}$ contains an element $t$. Then $t$ is the concatenation of some elements of $B_1$ and of some elements of $B_2$. Relabeling if necessary, let $s_i\in B_i,\ t_i\in\ov{B_i}$ for $i=1,2$ such that $s_1\circ t_1=t=s_2\circ t_2$. Then either $s_1$ is a subsegment of $s_2$ or vice versa. Without loss of generality, we assume $s_1\subsetneq s_2$. Let $s^{\pr}$ be the segment such that $s_1\circ s^{\pr}=s_2$. Since $W\cup\{s_1\}$ is closed, $s^{\pr}$ must not be in $W$. But $s^{\pr}$ is in $W$ since $W\cup\{s_2\}$ is co-closed. Hence, we have shown that the closures of the blocks are disjoint.

Since the biclosed property is preserved under restriction, the map $X\mapsto(X\cap\ov{B_1},\ldots,X\cap\ov{B_l})$ from $[W,W\cup\ov{\{s_1,\ldots,s_k\}}$ to $\Bic(\ov{B_1})\times\cdots\times\Bic(\ov{B_l})$ is well-defined. It remains to show that the inverse is also well-defined. Namely, given $(X_1,\ldots,X_l)\in\Bic(\ov{B_1})\times\cdots\times\Bic(\ov{B_l})$, we prove that $W\cup\bigcup_{i=1}^lX_i$ is biclosed in $\Seg(T)$. Suppose this does not always hold, and choose $(X_1,\ldots,X_l)$ minimal such that $W\cup\bigcup_{i=1}^lX_i$ is not biclosed. Let $X=\bigcup_{i=1}^lX_i$. Since $X\neq\emptyset$, there is some nonempty $X_j$. As $\Bic(\ov{B_j})$ is ordered by single-step inclusion, there is some $s\in X_j$ such that $X_j\setm\{s\}$ is biclosed. By the minimality assumption, $(W\cup X)\setm\{s\}$ is biclosed.

Assume $W\cup X$ is not co-closed. Then there exist segments $t,t^{\pr}$ not in $W\cup X$ such that $s=t\circ t^{\pr}$. As $X_j$ is co-closed in $\ov{B_j}$, the segment $t$ is not in $\ov{\{s_1,\ldots,s_k\}}$. Since $W\cup\{s_i\}$ is co-closed for any $i$, the segment $s$ can be factored as $s_i\circ s^{\pr}$ for some $s_i\in X_j$ and $s^{\pr}\in\ov{\{s_1,\ldots,s_k\}}$. There are two cases to consider: either $t$ is contained in $s_i$ or $s_i$ is contained in $t$.

If $t\subsetneq s_i$, then there exists a segment $t^{\pr\pr}$ with $t\circ t^{\pr\pr}=s_i$. Since $W\cup\{s_i\}$ is co-closed, $t^{\pr\pr}$ is in $W$. However, $t^{\pr\pr}\circ s^{\pr}=t^{\pr},\ s^{\pr}\in W\cup\ov{B_j}$ and $t^{\pr}\notin W\cup\ov{B_j}$. This contradicts the fact that $W\cup\ov{B_j}$ is closed.

If $s_i\subsetneq t$, then there exists a segment $t^{\pr\pr}$ with $s_i\circ t^{\pr\pr}=t$. Since $W\cup\{s_i\}$ is closed, $t^{\pr\pr}$ is not in $W$. However, $t^{\pr\pr}\circ t^{\pr}=s^{\pr},\ s^{\pr}\in W\cup\ov{B_j}$ and $t^{\pr\pr},t^{\pr}\notin W\cup\ov{B_j}$. This contradicts the fact that $W\cup\ov{B_j}$ is co-closed.

Now assume $W\cup X$ is not closed. Then there exist segments $s^{\pr}\in W\cup X,\ t\notin W\cup X$ such that $s\circ s^{\pr}=t$. Since $X_j$ is closed and segments in blocks $B_i$ with $i\neq j$ cannot be concatenated with $s$, the segment $s^{\pr}$ is in $W$. After relabeling, we may assume $s=s_1\circ\cdots\circ s_m$ for some $m\leq k$. Since $W\cup\{s_m\}$ is closed, the segment $s_m\circ s^{\pr}$ is in $W$. Similarly, $s_i\circ\cdots\circ s_m\circ s^{\pr}$ is in $W$ for any $i$. This contradicts the assumption that $t\notin W$.
\end{proof}

We may refer to intervals of $\Bic(T)$ as in Proposition~\ref{prop_biclosed_facial_interval} as \textbf{facial intervals}.

\subsection{A lattice congruence on biclosed sets}\label{subsec_bic_congruence}

In this section, we define a lattice congruence $\Theta$ on $\Bic(T)$. The quotient lattice $\Bic(T)/\Theta$ will be shown to be isomorphic to $\ora{FG}(T)$ in Section~\ref{subsec_map}.

Let $s=(v_0,\ldots,v_l)$ be a segment, and orient the segment from $v_0$ to $v_l$. Let $C_s$ be the set of segments $(v_i,\ldots,v_j)$ such that
\begin{itemize}
\item if $i>0$ then $s$ turns right at $v_i$, and
\item if $j<l$ then $s$ turns left at $v_j$.
\end{itemize}

We note that $s$ is always in $C_s$ since the above conditions are vacuously true. Furthermore if $t\in C_s$, then $C_t\subseteq C_s$. Let $K_s$ be the set of segments $(v_i,\ldots,v_j)$ such that
\begin{itemize}
\item if $i>0$ then $s$ turns left at $v_i$, and
\item if $j<l$ then $s$ turns right at $v_j$.
\end{itemize}

The following simple statement is used frequently in later proofs, so we state it explicitly.

\begin{lemma}\label{lem_subsegment}
Let $s,t\in\Seg(T)$ such that $t\in C_s$. If $t=t_1\circ t_2$, then either $t_1\in C_s$ or $t_2\in C_s$. The same statement holds replacing $C_s$ with $K_s$.
\end{lemma}

\begin{proof}
If $t=t_1\circ t_2$, then either $t_1\in C_t$ or $t_2\in C_t$. Since $C_t\subseteq C_s$, either $t_1\in C_s$ or $t_2\in C_s$. The dual statement about $K_s$ follows from the same reasoning.
\end{proof}

Given a tree $T$ embedded in a disk, we let $T^{\vee}$ be a \textbf{reflection} of $T$ (i.e. $T^{\vee}$ is the image of $T$ under a Euclidean reflection performed on $D^2$). The choice of reflection is immaterial since the noncrossing complex and oriented flip graph are invariant under rotations of $T$. The tree $T^{\vee}$ has the same set of segments and defines the same noncrossing complex as $T$. Since reflection switches left and right, $\ora{FG}(T^{\vee})$ has the opposite orientation of $\ora{FG}(T)$, and for any segment $s$, $C_{s^\vee}=K_s^\vee$. Let $\pi_{\downarrow},\pi^{\uparrow}$ be functions on $\Bic(T)$ such that for $X\in\Bic(T)$,
$$\pi_{\downarrow}(X)=\{s\in X:\ C_s\subseteq X\}$$
$$\pi^{\uparrow}(X)=\{s\in S:\ K_s\cap X\neq\emptyset\}$$

These maps are closely related to the maps labeled $\pi_{\downarrow}$ and $\pi^{\uparrow}$ in \cite{garver2015lattice}. For completeness, we prove their main properties here.

\begin{lemma}\label{lem_pi_defined}
For $X\in\Bic(T)$, both $\pi_{\downarrow}(X)$ and $\pi^{\uparrow}(X)$ are biclosed.
\end{lemma}

\begin{proof}
Let $s\in\pi_{\downarrow}(X)$. Then $C_s\subseteq X$. Since $C_t\subseteq C_s$ for $t\in C_s$, it follows that $C_s\subseteq\pi_{\downarrow}(X)$. If $s=t\circ u$, then either $t\in C_s$ or $u\in C_s$, so either $t\in\pi_{\downarrow}(X)$ or $u\in\pi_{\downarrow}(X)$. Hence, $\pi_{\downarrow}(X)$ is co-closed.

Let $s,t\in\pi_{\downarrow}(X)$ such that $s\circ t$ is a segment. For $u\in C_{s\circ t}$ if $u$ is a subsegment of $s$ or $t$, then $u\in C_s$ or $u\in C_t$, respectively. Otherwise, $u=u^{\pr}\circ u^{\pr\pr}$ where $u^{\pr}$ is a subsegment of $s$ and $u^{\pr\pr}$ is a subsegment of $t$. In this case $u^{\pr}\in C_s$ and $u^{\pr\pr}\in C_t$. In either case, $u\in X$ holds. Consequently $s\circ t\in\pi_{\downarrow}(X)$. Therefore, $\pi_{\downarrow}(X)$ is biclosed.

The fact that $\pi^{\uparrow}(X)$ is biclosed may be proved by a similar argument. Alternatively, it follows from the fact that $\pi_{\downarrow}(X)$ is biclosed and Lemma \ref{lem_map_properties}(\ref{lem_map_properties_0}).
\end{proof}

\begin{lemma}\label{lem_map_properties}
For $X,Y\in\Bic(T)$:
\begin{enumerate}
\item\label{lem_map_properties_0} $\pi_{\downarrow}(\Seg(T^\vee)\setm X^\vee)=\Seg(T^\vee)\setm \pi^{\uparrow}(X)^\vee$,
\item\label{lem_map_properties_1} $\pi_{\downarrow}(\pi^{\uparrow}(X))=\pi_{\downarrow}(X)$,
\item\label{lem_map_properties_2} $\pi^{\uparrow}(\pi_{\downarrow}(X))=\pi^{\uparrow}(X)$,
\item\label{lem_map_properties_3} $\pi_{\downarrow}(X)\subseteq X\subseteq\pi^{\uparrow}(X)$,
\item\label{lem_map_properties_4} $\pi_{\downarrow}(\pi_{\downarrow}(X))=\pi_{\downarrow}(X)$,
\item\label{lem_map_properties_5} $\pi^{\uparrow}(\pi^{\uparrow}(X))=\pi^{\uparrow}(X)$,
\item\label{lem_map_properties_6} if $X\subseteq Y$, then $\pi_{\downarrow}(X)\subseteq\pi_{\downarrow}(Y)$ and $\pi^{\uparrow}(X)\subseteq\pi^{\uparrow}(Y)$.
\end{enumerate}
\end{lemma}

\begin{proof}
Both (\ref{lem_map_properties_3}) and (\ref{lem_map_properties_6}) are clear from the definitions. (\ref{lem_map_properties_2}) and (\ref{lem_map_properties_5}) follow from (\ref{lem_map_properties_1}) and (\ref{lem_map_properties_4}) by taking the complement of the reflection of $X$ and applying (\ref{lem_map_properties_0}). It remains to prove (\ref{lem_map_properties_0}), (\ref{lem_map_properties_1}), and (\ref{lem_map_properties_4}).

For (\ref{lem_map_properties_0}), we have the following set of equalities:
\begin{align*}
\pi_{\downarrow}(\Seg(T^\vee)\setm X^\vee) &= \{s^\vee\in\Seg(T^\vee)|\ C_{s^\vee}\subseteq\Seg(T^\vee)\setm X^\vee\}\\
&=\{s^\vee\in\Seg(T^\vee)|\ K_s^\vee\subseteq\Seg(T^\vee)\setm X^\vee\}\\
&=\{s^\vee\in\Seg(T^\vee)|\ K_s\subseteq\Seg(T)\setm X\}\\
&=\Seg(T^\vee)\setm \{s^\vee\in\Seg(T^\vee)|\ K_s\cap X\neq\emptyset\}\\
&=\Seg(T^\vee)\setm \pi^{\uparrow}(X)^\vee.
\end{align*}

For (\ref{lem_map_properties_1}), the reverse inclusion is clear. Suppose $\pi_{\downarrow}(\pi^{\uparrow}(X))\neq\pi_{\downarrow}(X)$ and let $s\in\pi_{\downarrow}(\pi^{\uparrow}(X))\setm \pi_{\downarrow}(X)$ be of minimum length. Since $C_t\subseteq C_s$ for $t\in C_s$, this implies $s\in\pi^{\uparrow}(X)\setm X$. Let $u\in K_s\cap X$. Then either $s=t\circ u, s=u\circ t^{\pr}$, or $s=t\circ u\circ t^{\pr}$ holds for some segments $t,t^{\pr}\in C_s$. But this implies $s\in X$, a contradiction.

For (\ref{lem_map_properties_4}), the inclusion $\pi_{\downarrow}(\pi_{\downarrow}(X))\subseteq\pi_{\downarrow}(X)$ is clear. Let $s\in\pi_{\downarrow}(X)$. Then $C_s\subseteq X$ holds. If $t\in C_s$, then $C_t\subseteq C_s$ and $t\in\pi_{\downarrow}(X)$. Consequently, $C_s\subseteq\pi_{\downarrow}(X)$, so $s\in\pi_{\downarrow}(\pi_{\downarrow}(X))$.
\end{proof}

For the remainder of the paper, we let $\Theta$ be the equivalence relation on $\Bic(T)$ such that $X\equiv Y\mod{\Theta}$ if $\pi_{\downarrow}(X)=\pi_{\downarrow}(Y)$. Using Lemmas \ref{lem_lattice_congruence} and \ref{lem_map_properties}, we deduce the following proposition.

\begin{proposition}\label{prop_Theta}
The equivalence relation $\Theta$ is a lattice congruence on $\Bic(T)$.
\end{proposition}

\subsection{Map from biclosed sets to the oriented flip graph}\label{subsec_map}

In this section, we define a surjective map $\eta:\Bic(T)\ra\ora{FG}(T)$ and prove that it is a lattice quotient map.

Let $X\in\Bic(T)$. Given a corner $(v,F)$, let $p_{(v,F)}$ be the (unique) arc supported by $T$ such that for any interior vertex $u$ of $p_{(v,F)}$ distinct from $v$, the following condition holds:
\begin{itemize}
\item Orienting $p_{(v,F)}$ from $v$ to $u$, the arc $p_{(v,F)}$ turns left at $u$ if and only if $[v,u]$ is in $X$.
\end{itemize}

In Lemmas~\ref{lem_eta_CK} and \ref{lem_eta_facet}, we prove that the this collection of arcs is a facet of the noncrossing complex. Before proving this, we set up some notation.

\begin{figure}
\centering
\includegraphics{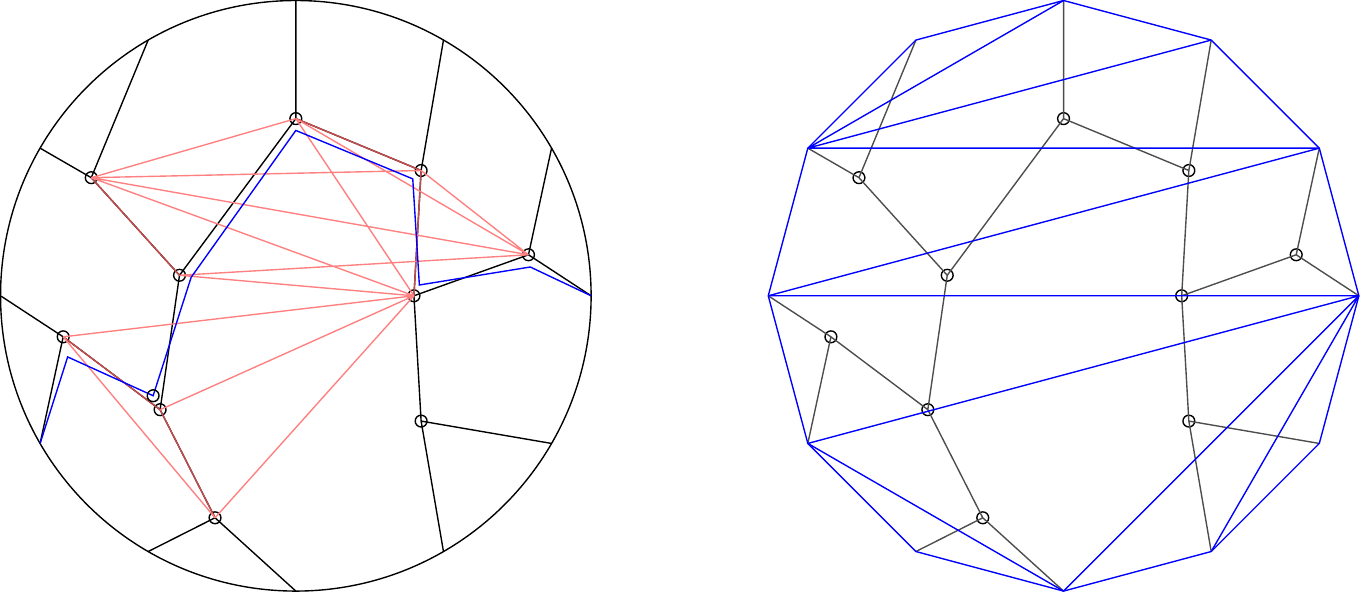}
\caption{\label{fig_eta}(left) A blue arc defined by $\eta$ at the circled corner with respect to the red biclosed set of segments; (right) The triangulation defined by $\eta$}
\end{figure}

For an arc $p=(v_0,\ldots,v_l)$ oriented from $v_0$ to $v_l$, let $C_p$ be the set of segments $(v_i,\ldots,v_j),\ 0<i<j<l$ such that
\begin{itemize}
\item $p$ turns right at $v_i$, and
\item $p$ turns left at $v_j$.
\end{itemize}

Define $K_p$ in the same way, switching the roles of left and right.


\begin{lemma}\label{lem_eta_CK}
Let $X$ and $\{p_{(v,F)}\}_{(v,F)}$ be defined as above. For $p\in\{p_{(v,F)}\}_{(v,F)}$, $C_p\subseteq X$ and $K_p\cap X=\emptyset$.
\end{lemma}

\begin{proof}
Let $p=p_{(v,F)}$ for some corner $(v,F)$ of $T$. Let $s\in C_p$, and set $s=[u,w]$. We show that $s\in X$ by considering several cases on the location of $v$ relative to $s$.

If $v$ is an endpoint of $s$, then $s\in X$ by the defining rule of $p_{(v,F)}$.

If $v$ is in the interior of $s$, then $s=[u,v]\circ[v,w]$. Since $p$ extends left through both endpoints of $s$, both $[u,v]$ and $[v,w]$ are in $X$. Since $X$ is closed, this implies $s\in X$.

If $v$ is not in $s$, then there exists a segment $[v,u]$ such that $[v,u]\circ[u,w]$ is a segment of $p$. Since $p$ extends left through both endpoints of $s$, $[v,w]\in X$ but $[v,u]\notin X$. Since $X$ is co-closed, this implies $s\in X$.

The fact that $K_p\cap X=\emptyset$ follows from a dual argument.
\end{proof}

\begin{lemma}\label{lem_eta_facet}
The set $\{p_{(v,F)}\}_{(v,F)}$ is a facet of $\Delta^{NC}(T)$. Moreover, $p_{(v,F)}$ is the arc marked at the corner $(v,F)$.
\end{lemma}

\begin{proof}
Let $(v,F),\ (v^{\pr},F^{\pr})$ be two corners of $T$ and let $p_1=p_{(v,F)}$ and $p_2=p_{(v^{\pr},F^{\pr})}$. Suppose $p_1$ and $p_2$ cross along a segment $s$. We may assume that $p_1$ leaves each of the endpoints of $s$ to the right while $p_2$ leaves $s$ to the left. Then $s\in K_{p_1}$ and $s\in C_{p_2}$. By Lemma \ref{lem_eta_CK}, $K_{p_1}\cap X=\emptyset$ and $C_{p_2}\subseteq X$, a contradiction.

Let $\Fcal=\{p_{(v,F)}\}_{(v,F)}$, and let $(v,F)$ be a corner of $T$. Let $q\in\Fcal$ be the arc marked at $(v,F)$. If $q\neq p_{(v,F)}$, then they agree on some segment $[v,w]$ and diverge at $w$. Orient both arcs from $v$ to $w$. Then $p_{(v,F)}$ turns in the same direction at both $v$ and $w$ whereas $q$ turns in different directions.

If $q$ turns left at $v$ and right at $w$, then $[v,w]\notin X$ since $K_q\cap X=\emptyset$. As $p_{(v,F)}$ turns left at $w$, this contradicts the rule defining $p_{(v,F)}$.

If $q$ turns right at $v$ and left at $w$, then $[v,w]\in X$ since $C_q\subseteq X$. As $p_{(v,F)}$ turns right at $w$, this again contradicts the rule defining $p_{(v,F)}$.

In either case, we obtain a contradiction. Hence, $p_{(v,F)}$ is the arc marked at $(v,F)$.

It remains to show that $\Fcal$ is maximal. If not, then there exist two corners $(v,F),(v^{\pr},F^{\pr})$ such that $p_{(v,F)}=p_{(v^{\pr},F^{\pr})}$ and $R_{(v,F)}(p_{(v,F)})=R_{(v^{\pr},F^{\pr})}(p_{(v^{\pr},F^{\pr})})$. Let $p=p_{(v,F)}$. Orient $p$ from $v$ to $v^{\pr}$. Since $R_{(v,F)}(p)=R_{(v^{\pr},F^{\pr})}(p)$, $p$ turns in the same direction at $v$ and $v^{\pr}$.

If $p$ turns right at both $v$ and $v^{\pr}$, then $[v,v^{\pr}]\in X$ by the definition of $p_{(v^{\pr},F^{\pr})}$ but $[v,v^{\pr}]\notin X$ by definition of $p_{(v,F)}$. If $p$ turns left at both $v$ and $v^{\pr}$, then $[v,v^{\pr}]\in X$ by definition of $p_{(v,F)}$ but $[v,v^{\pr}]\notin X$ by definition of $p_{(v^{\pr},F^{\pr})}$. In either case, we obtain a contradiction. Hence, $\Fcal$ is a facet of $\Delta^{NC}(T)$.
\end{proof}

We let $\eta:\Bic(T)\ra\ora{FG}(T)$ be the map $\eta(X)=\{p_{(v,F)}\}_{(v,F)}$ where $X\in\Bic(T)$ and arcs $p_{(v,F)}$ are defined as above. An example of this map is given in Figure \ref{fig_eta}.

For $\Fcal\in\ora{FG}(T)$, let $\phi(\Fcal)=\ov{\bigcup_{p\in\Fcal}C_p}$.

\begin{lemma}\label{lem_eta_phi}
For $X\in\Bic(T)$ and $\Fcal\in\ora{FG}(T)$,
\begin{enumerate}
\item\label{lem_eta_phi_1} $\phi(\eta(X))=\pi_{\downarrow}(X)$, and
\item\label{lem_eta_phi_2} $\eta(\phi(\Fcal))=\Fcal$.
\end{enumerate}
\end{lemma}

\begin{proof}
(\ref{lem_eta_phi_1}): Let $X\in\Bic(T)$ be given. Set $\Fcal=\eta(X)$. If $s\in\phi(\Fcal)$, then there exist $s_1,\ldots,s_l$ such that $s=s_1\circ\cdots\circ s_l$ and $s_i\in C_p$ for some $p\in\Fcal$. For each $i$, $C_{s_i}\subseteq C_p\subseteq X$, so $s_i\in\pi_{\downarrow}(X)$. As $\pi_{\downarrow}(X)$ is closed, this implies $s\in\pi_{\downarrow}(X)$. Hence $\phi(\eta(X))\subseteq\pi_{\downarrow}(X)$.

We prove the reverse inclusion $\pi_{\downarrow}(X)\subseteq\phi(\eta(X))$ by induction on the length. Let $s\in\pi_{\downarrow}(X)$ and assume that $t\in C_s,\ t\neq s$ implies $t\in\phi(\eta(X))$. Let $v$ be an endpoint of $s$. Orienting $s$ away from $v$, let $F$ be the face to the right of $s$. Let $p=p_{(v,F)}$ and orient $p$ in the same direction as $s$. Let $v^{\pr}$ be the last vertex along $s$ at which $s$ and $p$ meet. Let $t=[v,v^{\pr}]$. If $v^{\pr}$ is an endpoint of $s$, then $p$ must turn left at $v^{\pr}$ by definition, and $s\in C_p$. If $v^{\pr}$ is not an endpoint of $s$, we consider two cases:

(i) If $s$ turns left at $v^{\pr}$, then $t\in C_s$. By the inductive hypothesis, $t\in\phi(\eta(X))$ holds, which contradicts the definition of $p_{(v,F)}$.

(ii) If $s$ turns right at $v^{\pr}$, then $s=t\circ t^{\pr}$ and $t\in C_p$. Since $t^{\pr}\in C_s$, $t^{\pr}\in\phi(\eta(X))$. Hence $s\in\phi(\eta(X))$ holds.

(\ref{lem_eta_phi_2}): Let $\Fcal\in\ora{FG}(T)$ and set $X=\phi(\Fcal)$. Let $(v,F)$ be a corner of $T$. Let $p$ be the arc in $\eta(\phi(\Fcal))$ marked at $(v,F)$ and let $q$ be the arc in $\Fcal$ marked at $(v,F)$. We prove that $p=q$ and conclude that $\eta(\phi(\Fcal))=\Fcal$.

Suppose $p$ and $q$ diverge at some vertex $v^{\pr}$. Orient both paths from $v$ to $v^{\pr}$. Let $s=[v,v^{\pr}]$.

Assume $p$ turns left at $v^{\pr}$ and $q$ turns right at $v^{\pr}$. Then $s\in\phi(\Fcal)$, so there exist $s_1,\ldots,s_l$ such that $s=s_1\circ\cdots\circ s_l$ and $s_i\in C_{q_i}$ for some arcs $q_i\in\Fcal$. Orient each $q_i$ in the same direction as $q$. We may assume $v\in s_1$ and $v^{\pr}\in s_l$. Let $v_i$ be the first vertex of $s_i$ for each $i$. Since $q_1$ and $q$ do not cross and $q$ is marked at $(v,F)$, we conclude that both $q_1$ and $q$ turn left at $v_2$. By similar reasoning, $q_2$ and $q$ both turn left at $v_3$. By induction, $q$ turns left at $v^{\pr}$, a contradiction.

Now assume $p$ turns right at $v^{\pr}$ and $q$ turns left. Then $s\notin\phi(\Fcal)$. Since $s\notin C_q$, $q$ must turn left at $v$. Let $F^{\pr}$ be the face to the right of $q$ containing $v$ and the first edge of $s$. Let $q^{\pr}$ be the arc of $\Fcal$ marked at $(v,F^{\pr})$. Then $q^{\pr}$ and $q$ agree after $v$. Hence $s\in C_{q^{\pr}}$, a contradiction.
\end{proof}

By Lemma~\ref{lem_eta_phi}, the equivalence relation on $\Bic(T)$ induced by $\eta$ is equal to $\Theta$. That is, $X\equiv Y\mod\Theta$ holds if and only if $\eta(X)=\eta(Y)$. By Proposition~\ref{prop_Theta}, we may identify the facets of the noncrossing complex with the elements of the quotient lattice $\Bic(T)/\Theta$. It remains to show that this ordering is isomorphic to $\ora{FG}(T)$. To this end, it is enough to check that the Hasse diagram of $\Bic(T)/\Theta$ is $\ora{FG}(T)$, as in the following lemma. Recall the edge-labeling of the oriented flip graph from Definition~\ref{orflipgraph}.


\begin{lemma}\label{lem_eta_phi_order_preserving}
The Hasse diagram of $\ora{FG}(T)$ is isomorphic to that of $\Bic(T)/\Theta$. More precisely, we have the following.
\begin{enumerate}
\item\label{lem_eta_phi_order_preserving_1} Let $X\in\Bic(T)$ such that $X=\pi_{\downarrow}(X)$. If $s$ is a segment in $X$ such that $X\setm \{s\}$ is biclosed, then $\eta(X\setm \{s\})\stackrel{s}{\ra}\eta(X)$. 
\item\label{lem_eta_phi_order_preserving_2} Let $\Fcal\in\ora{FG}(T)$. If $\Fcal\setm \{p\}\cup\{\tilde{p}\}\stackrel{s}{\ra}\Fcal$ for some arcs $p,\tilde{p}$ and segment $s$, then $\phi(\Fcal)\setm \{s\}$ is biclosed and $\eta(\phi(\Fcal)\setm \{s\})=\Fcal\setm \{p\}\cup\{\tilde{p}\}$. 
\end{enumerate}
\end{lemma}

\begin{proof}
(\ref{lem_eta_phi_order_preserving_1}): Let $X\in\Bic(T)$ such that $X=\pi_{\downarrow}(X)$. Let $s$ be a segment in $X$ such that $X\setm \{s\}$ is biclosed. Then $C_s\subseteq X$ and $K_s\cap X=\emptyset$. Let $v,v^{\pr}$ be the endpoints of $s$. Orient $s$ from $v$ to $v^{\pr}$. Let $F$ be the face to the right of $s$ incident to $v$ and the first edge of $s$. Let $p$ be the arc of $\eta(X)$ marked at $(v,F)$. Since $C_s\subseteq X$ and $K_s\cap X=\emptyset$, $p$ contains $s$ and turns left at $v^{\pr}$. Let $F^{\pr}$ be the face left of $s$ incident to $v^{\pr}$ and the last edge of $s$. Let $p^{\pr}$ be the arc of $\eta(X)$ marked at $(v^{\pr},F^{\pr})$. Reversing the orientation on $s$, the previous argument implies that $p^{\pr}$ contains $s$.

We claim that $p=p^{\pr}$. If not, then $p$ and $p^{\pr}$ must diverge at a vertex $v^{\pr\pr}$. Let $t=[v^{\pr},v^{\pr\pr}]$ and $u=[v,v^{\pr\pr}]$. Without loss of generality, we may assume that $s\circ t=u$. Since $p$ and $p^{\pr}$ do not cross, $p$ turns left at $v^{\pr\pr}$ and $p^{\pr}$ turns right at $v^{\pr\pr}$. Hence $u\in X$ and $t\notin X$. But then $X\setm \{s\}$ is not co-closed, a contradiction.

Let $\tilde{p}$ be the arc obtained by flipping $p$ in $\eta(X)$. Then $p$ and $\tilde{p}$ meet along $s$. We show that $\eta(X\setm \{s\})=\eta(X)\setm \{p\}\cup\{\tilde{p}\}$.

Let $G$ be the face left of $s$ containing $v$ and the first edge of $s$. Similarly, let $G^{\pr}$ be the face right of $s$ containing $v^{\pr}$ and the last edge of $s$. Let $q$ be the arc marked at $(v,G)$ in $\eta(X)$, and let $q^{\pr}$ be the arc marked at $(v^{\pr},G^{\pr})$ in $\eta(X)$.

By the definition of $\eta$, the only arcs that can be different between $\eta(X)$ and $\eta(X\setm \{s\})$ are those arcs marked at $(v,F),(v,G),(v^{\pr},F^{\pr}),$ or $(v^{\pr},G^{\pr})$. Just as we proved that $p$ is the arc in $\eta(X)$ marked at $(v,F)$ and $(v^{\pr},F^{\pr})$, a similar argument shows that $\tilde{p}$ is the arc in $\eta(X\setm \{s\})$ marked at $(v,G)$ and $(v^{\pr},G^{\pr})$.

We show that $q$ is in $\eta(X\setm \{s\})$ and is marked at $(v^{\pr},F^{\pr})$. Similarly we claim that $q^{\pr}$ is in $\eta(X\setm \{s\})$ and is marked at $(v,F)$. As these two proofs are nearly identical, we only write the first.

Let $\tilde{q}$ be the arc in $\eta(X\setm \{s\})$ marked at $(v^{\pr},F^{\pr})$, and assume $q\neq\tilde{q}$. Let $v^{\pr\pr}$ be a vertex at which $q$ and $\tilde{q}$ diverge. Orient $\tilde{q}$ from $v^{\pr}$ to $v^{\pr\pr}$. Then $\tilde{q}$ turns right at $v^{\pr\pr}$, so $v\neq v^{\pr\pr}$. Now orient $q$ from $v$ to $v^{\pr\pr}$.

If $v^{\pr\pr}$ is strictly between $v$ and $v^{\pr}$, then $q$ and $\tilde{q}$ turn in the same direction at $v^{\pr\pr}$. This is impossible since $C_s\subseteq X$ and $K_s\cap X=\emptyset$, and either $[v,v^{\pr\pr}]\in C_s$ and $[v^{\pr\pr},v^{\pr}]\in K_s$ or $[v,v^{\pr\pr}]\in K_s$ and $[v^{\pr\pr},v^{\pr}]\in C_s$.

If $v^{\pr}$ is strictly between $v$ and $v^{\pr\pr}$, then either $[v^{\pr},v^{\pr\pr}]\in X$ and $[v,v^{\pr\pr}]\notin X$ or $[v^{\pr},v^{\pr\pr}]\in X$ but $[v,v^{\pr\pr}]\notin X$. The first case implies $X\setm \{s\}$ is not co-closed, and the second case implies $X$ is not closed.

If $v$ is strictly between $v^{\pr}$ and $v^{\pr\pr}$, then we again deduce a contradiction in a similar way as the previous case. This completes the proof.

(\ref{lem_eta_phi_order_preserving_2}): Let $\Fcal\in\ora{FG}(T)$. Assume $\Fcal\setm \{p\}\cup\{\tilde{p}\}\stackrel{s}{\ra}\Fcal$ for some arcs $p,\tilde{p}$ and segment $s$. Then $s\in C_p$, so $s\in\phi(F)$.

Suppose $\phi(\Fcal)\setm \{s\}$ is not closed. Then there exist segments $t,u\in\phi(\Fcal)$ such that $t\circ u=s$. We may assume $t\in K_s$ and $u\in C_s$. Let $v,v^{\pr}$ be the endpoints of $t$. Assume $t$ and $u$ meet at $v^{\pr}$, and orient the arcs containing $t$ from $v$ to $v^{\pr}$. Let $t_1,\ldots,t_l$ be segments such that $t=t_1\circ\cdots\circ t_l$ and $t_i\in C_{p_i}$ for some arcs $p_i\in\Fcal$. We assume $v$ is in $t_1$ and $v^{\pr}$ is in $t_l$. For each $i$, let $v_i$ be the first vertex in $t_i$ with the orientation induced by $t$. Since $p_1$ and $\tilde{p}$ do not cross along $t_1$, $\tilde{p}$ must turn left at $v_2$. Similarly, $\tilde{p}$ turns left at $v_3,\ldots,v_l$. But since $t\in K_s$, $\tilde{p}$ turns right at $v^{\pr}$, so it crosses $p_l$, a contradiction. We deduce that $\phi(\Fcal)\setm \{s\}$ is closed.

Suppose $\phi(\Fcal)\setm \{s\}$ is not co-closed. Then there exist segments $t,u$ such that $t\notin\phi(\Fcal),\ u\in\phi(\Fcal)$, and $s\circ t=u$. Since $\pi_{\downarrow}(\phi(\Fcal))=\phi(\Fcal)$, we deduce that $s\in C_u$ and $t\in K_u$. Let $u_1,\ldots,u_l$ be segments with arcs $p_1,\ldots,p_l$ in $\Fcal$ such that $u_i\in C_{p_i}$ and $u=u_1\circ\cdots\circ u_l$. Orient $u$ from $u_1$ to $u_l$. By similar reasoning as before, since $\tilde{p}$ and $p_i$ do not cross along $u_i$ for each $i$, if $u_i$ is a subsegment of $s$, then $\tilde{p}$ turns left at the end of $u_i$. As $t\notin\phi(\Fcal)$, there exists some segment $u_j$ that is neither a subsegment of $s$ or $t$. Let $v^{\pr}$ be the common endpoint of $s$ and $t$, and let $v$ be the endpoint of $u_j$ contained in $s$. Since $s\in C_u$, $u_j$ turns left at $v$. Hence, $p_j$ turns right at $v$ and left at $v^{\pr}$, whereas $\tilde{p}$ turns left at $v$ and right at $v^{\pr}$. But this means $\tilde{p}$ and $p_j$ cross along $[v,v^{\pr}]$, an impossibility.

Therefore, $\phi(\Fcal)\setm \{s\}$ is biclosed. From (\ref{lem_eta_phi_order_preserving_1}), the equality $\eta(\phi(\Fcal)\setm \{s\})=\Fcal\setm \{p\}\cup\{\tilde{p}\}$ holds.
\end{proof}

\begin{theorem}\label{thm_eta_phi_main}
The maps $\eta$ and $\phi$ identify $\ora{FG}(T)$ as a quotient lattice and a sublattice of $\Bic(T)$ as follows.
\begin{enumerate}
\item The map $\eta$ is a surjective lattice map such that $\eta(X)=\eta(Y)$ if and only $X\equiv Y\mod\Theta$.
\item The map $\phi$ is an injective lattice map whose image is $\pi_{\downarrow}(\Bic(T))$.
\end{enumerate}
\end{theorem}

\begin{proof}
We have already established that $\eta$ is lattice quotient map. It remains to show that $\phi$ preserves the lattice operations.

Let $\Fcal,\Fcal^{\pr}\in\ora{FG}(T)$. Since $\eta$ is a lattice map,
\begin{align*}
\Fcal\vee\Fcal^{\pr} &=\eta(\phi(\Fcal))\vee\eta(\phi(\Fcal^{\pr}))\\
&=\eta(\phi(\Fcal)\vee\phi(\Fcal^{\pr}))\\
&\leq\eta(\phi(\Fcal\vee\Fcal^{\pr}))\\
&=\Fcal\vee\Fcal^{\pr}.
\end{align*}
Hence, $\eta(\phi(\Fcal)\vee\phi(\Fcal^{\pr}))=\eta(\phi(\Fcal\vee\Fcal^{\pr}))$. Since $\phi(\Fcal\vee\Fcal^{\pr})$ is minimal in its $\Theta$-equivalence class, $\phi(\Fcal\vee\Fcal^{\pr})\leq\phi(\Fcal)\vee\phi(\Fcal^{\pr})$. Since $\phi$ is order-preserving, the reverse inequality also holds. Thus, $\phi$ preserves joins.

Since $\phi$ is order-preserving, $\phi(\Fcal\wedge\Fcal^{\pr})\leq\phi(\Fcal)\wedge\phi(\Fcal^{\pr})$ holds. Let $X=\phi(\Fcal)\wedge\phi(\Fcal^{\pr})$. Since
$$\eta(\phi(\Fcal\wedge\Fcal^{\pr}))=\Fcal\wedge\Fcal^{\pr}=\eta\phi(\Fcal)\wedge\eta\phi(\Fcal^{\pr}),$$
it suffices to show that $\pi_{\downarrow}(X)=X$.

Let $s\in X$ and $t\in C_s$. Since $\phi(\Fcal)=\pi_{\downarrow}(\phi(\Fcal))$, $C_s\subseteq\phi(\Fcal)\cap\phi(\Fcal^{\pr})$. If $t\notin X$ then there exist $u_1,\ldots,u_l\notin\phi(\Fcal)\cap\phi(\Fcal^{\pr})$ such that $t=u_1\circ\cdots\circ u_l$. But $u_i\in C_t$ for some $i$. Since $C_t\subseteq C_s$, we deduce $u_i\in\phi(\Fcal)\cap\phi(\Fcal^{\pr})$, a contradiction.
\end{proof}

By Lemma~\ref{lem_cn_cong} and Theorem~\ref{thm_biclosed_main}(\ref{thm_biclosed_main_cn}), it follows that the labeling $\Fcal^{\pr}\stackrel{s}{\ra}\Fcal$ of the covering relations of $\ora{FG}(T)$ by segments is a CN-labeling. To see that this is a CU-labeling, we observe that if there is a flip $\Fcal^{\pr}\stackrel{s}{\ra}\Fcal$, then $\eta(C_s)\leq\Fcal$. The following corollary is a consequence of Proposition~\ref{prop_cu_CJR}.

\begin{corollary}\label{cor_CJR}
The canonical join-representation of a element $\Fcal\in\ora{FG}(T)$ is
$$\Fcal=\bigvee_{\substack{s\in S\\\exists\Fcal^{\pr}\stackrel{s}{\ra}\Fcal}}\eta(C_s).$$
\end{corollary}

\section{Noncrossing tree partitions}\label{Sec:NCP}

In this section, we introduce noncrossing tree partitions, which are partitions of the interior vertices of a tree embedded in a disk whose blocks are noncrossing as defined in Section~\ref{subsec_admissible_curves}. In Section~\ref{subsec_kreweras}, we define a bijection on the set of noncrossing tree partitions, which we call Kreweras complementation. The equivalence of this definition of Kreweras complementation with the lattice-theoretic definition in Section~\ref{subsec_lattices} is given in Section~\ref{subsec_redgreen_trees}. Our main result in this section is that the lattice of noncrossing tree partitions is isomorphic to the shard intersection order of $\ora{FG}(T)$, which we prove in Section~\ref{subsec_shard_intersection_order}.

\subsection{Admissible curves}\label{subsec_admissible_curves}

Fix a tree $T=(V,E)$ embedded in a disk $D^2$ with the Euclidean metric. Let $V^o$ denote the set of interior vertices of $T$. We fix a small $\epsilon>0$ such that the $\epsilon$-ball centered at any interior vertex of $T$ is contained in $D^2$, and no two such $\epsilon$-balls intersect. For each corner $(v,F)$, we fix a point $z(v,F)$ in the interior of $F$ of distance $\epsilon$ from $v$. Let
$$T_{\epsilon}=T\cup\bigcup_{v\in V^o}\{x\in D^2:\ |x-v|<\epsilon\}.$$
In words, $T_{\epsilon}$ is the embedded tree $T$ plus the open $\epsilon$-ball around each interior vertex. If $s$ is a segment of $T$, let $s_{\epsilon}$ denote the set of points on an edge of $s$ of distance at least $\epsilon$ from any interior vertex of $T$.

It will be convenient to represent segments as certain curves in the disk as follows. A \textbf{flag} is a triple $(v,e,F)$ of a vertex $v$ incident to an edge $e$, which is incident to a face $F$. Orienting $e$ away from $v$, we say a flag is \textbf{green} if $F$ is left of $e$. Otherwise, the flag is \textbf{red}. Let $(u,e,F),\ (v,e^{\pr},G)$ be two green flags such that $[u,v]$ is a segment containing the edges $e,e^{\pr}$ as in Figure \ref{fig_redgreenarcs}. A \textbf{green admissible curve} $\gamma:[0,1]\ra D^2$ for $[u,v]$ is a simple curve for which $\gamma(0)=z(u,F),\ \gamma(1)=z(v,G)$ and $\gamma([0,1])\subseteq D^2\setm (T_{\epsilon}\setm [u,v]_{\epsilon})$. Similarly, if $(u,e,F^{\pr})$ and $(v,e^{\pr},G^{\pr})$ are red flags, then a \textbf{red admissible curve} is defined the same way, with $\gamma(0)=z(u,F^{\pr}),\ \gamma(1)=z(v,G^{\pr})$. We say a segment is \textbf{green} if it is represented by a green admissible curve. Similarly, a segment is \textbf{red} if it is represented by a red admissible curve. We may also refer to an \textbf{admissible curve} for a segment without specifying a color. Such a curve may be either green or red.

If a colored segment $s$ is represented by a curve with endpoints $z(u,F)$ and $z(v,G)$, we say that $(u,F)$ and $(v,G)$ are the \textbf{endpoints} of $s$, and we write $\Endpt(s)=\{(u,F),(v,G)\}$. If $S$ is a collection of colored segments, we let $\Endpt(S)=\bigcup_{s\in S}\Endpt(s)$. We refer to corners or vertices as the endpoints of a segment at different parts of this paper. The distinction should be clear from context.

\begin{figure}
\includegraphics{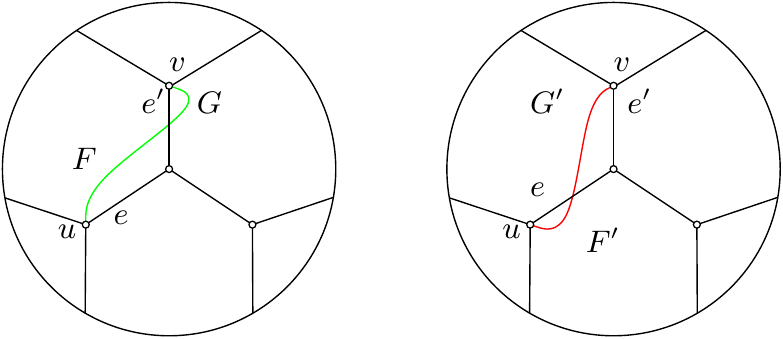}
\caption{\label{fig_redgreenarcs}A green admissible curve and a red admissible curve for the segment $[u,v]$}
\end{figure}

Given an interior vertex $v\in V^o$ incident to faces $F$ and $F^{\pr}$, let $\alpha_v^{F,F^{\pr}}:[0,1]\ra D^2$ be a simple path contained in $F\cup F^{\pr}$ with $\alpha_v^{F,F^{\pr}}(0)=z(v,F),\ \alpha_v^{F,F^{\pr}}(1)=z(v,F^{\pr})$ and $|\alpha_v^{F,F^{\pr}}(t)-v|=\epsilon$ for $t\in[0,1]$. We use the paths $\alpha_v^{F,F^{\pr}}$ to concatenate admissible curves.

Two colored segments are \textbf{noncrossing} if they admit admissible curves that do not intersect. Otherwise, they are \textbf{crossing}. We remark that if two curves share an endpoint $z(u,F)$ then they are considered to be crossing. To determine whether two colored segments $s,t$ cross, one must check whether the endpoints of $t$ lie in different connected components of $(D^2\setm (T_{\epsilon}\setm t_{\epsilon}))\setm \gamma$ for some admissible curve $\gamma$ for $s$. We will find it convenient to distinguish several cases of crossing as in the following lemma. The three cases correspond to the three columns of Figure \ref{fig_admissible_curve_intersect}.

\begin{lemma}\label{Lemma:crossingcurves}
Let $\gamma$ and $\gamma^{\pr}$ be two (left or right) admissible curves corresponding to segments $s$ and $s^{\pr}$ that meet along a common segment $t$. Let $t=[a,b]$ and orient $\gamma$ and $\gamma^{\pr}$ from $a$ to $b$. Assume that $\gamma$ and $\gamma^{\pr}$ do not share a corner. Then $\gamma$ and $\gamma^{\pr}$ are noncrossing if and only if one of the following holds:
\begin{enumerate}
\item $\gamma$ (or $\gamma^{\pr}$) does not share an endpoint with $t$, and $\gamma$ turns left (or right) at both endpoints of $t$;
\item $\gamma$ starts at $a$ and turns left (resp. right) at $b$, and $\gamma^{\pr}$ ends at $b$ and turns right (resp. left) at $a$;
\item $\gamma$ and $\gamma^{\pr}$ both start at $a$ (resp. both end at $b$) where $\gamma$ leaves $a$ (resp. $b$) to the left, and $\gamma$ turns left at $b$ (resp. $a$) or $\gamma^{\pr}$ turns right at $b$ (resp. $a$).
\end{enumerate}
If $\gamma$ and $\gamma^{\pr}$ are both left admissible or both right admissible, then the third case does not occur.
\end{lemma}

\begin{figure}
\centering
\includegraphics{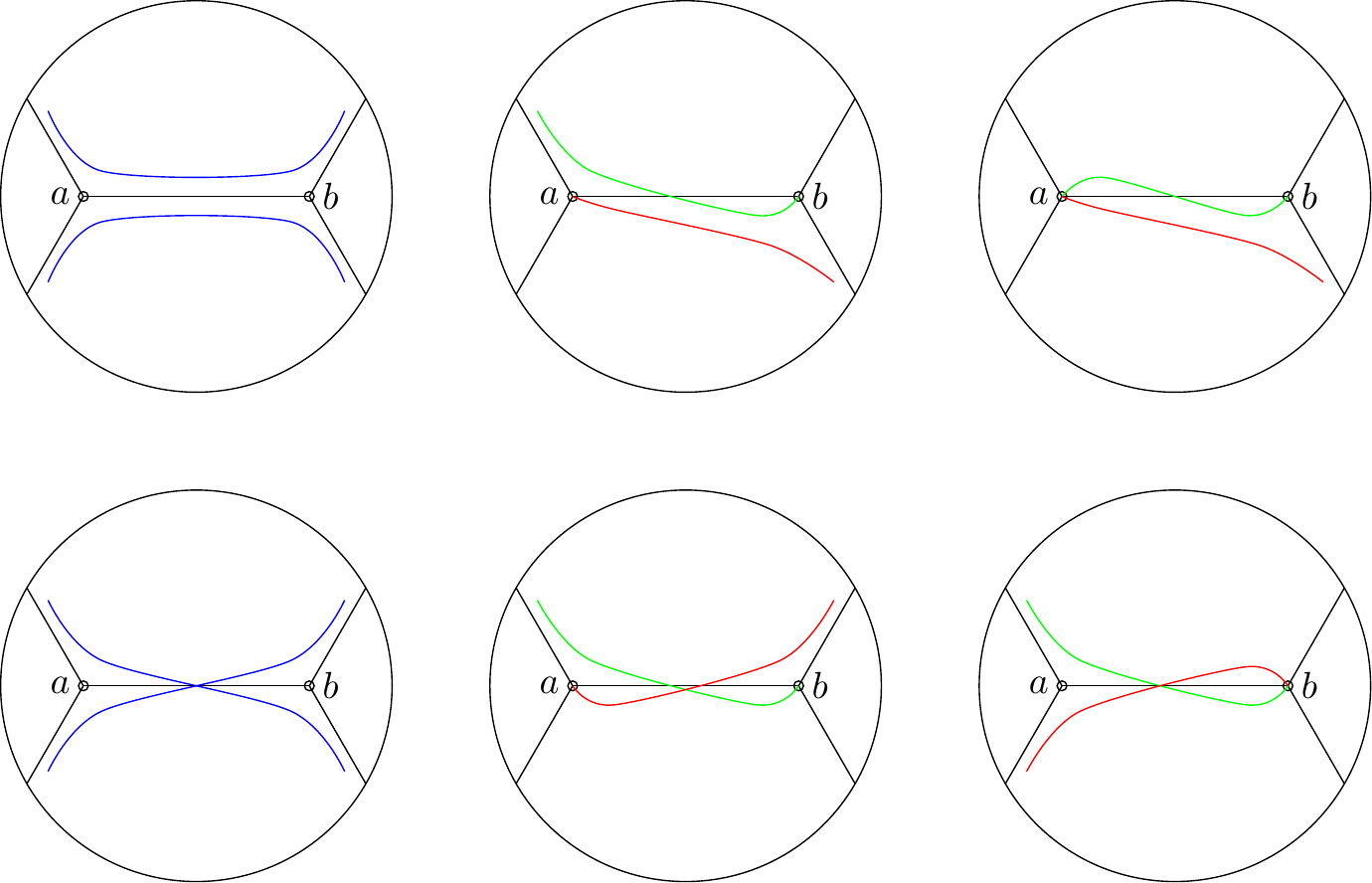}
\caption{\label{fig_admissible_curve_intersect}Several examples of crossing and noncrossing admissible curves representing segments supported by the tree.}
\end{figure}

\begin{lemma}\label{lem_crossing_condition}
If red segments $s$ and $s^{\pr}$ are noncrosssing, then $K_s\cap C_{s^{\pr}}$ is empty.
\end{lemma}

\begin{proof}
Suppose $K_s\cap C_{s^{\pr}}$ contains an element $[u,v]$. Orient $s$ and $s^{\pr}$ from $u$ to $v$. Then $s$ either starts at $u$ or turns left at $u$, and it either ends at $v$ or turns right at $v$. On the other hand, $s^{\pr}$ either starts at $u$ or turns right at $u$, and it either ends at $v$ or turns left at $v$. In each case, the segments $s$ and $s^{\pr}$ are crossing.
\end{proof}

By a similar analysis, the same result holds for segments of different color. We note that Lemma~\ref{lem_crossing_condition_2} is asymmetric in red and green.

\begin{lemma}\label{lem_crossing_condition_2}
If a green segment $s$ and a red segment $s^{\pr}$ are noncrossing, then $K_s\cap C_{s^{\pr}}$ is empty.
\end{lemma}

For $B\subseteq V^o$, let $\Seg_g(B)$ be the set of inclusion-minimal segments whose endpoints lie in $B$. We say $B$ is segment-connected if for any two elements $u,v$ of $B$, there exists a sequence $u=u_0,\ldots,u_N=v$ of elements of $B$ such that $[u_{i-1},u_i]\in\Seg(B)$ for all $i$. If $\Bbf=(B_1,\ldots,B_l)$ is a partition of $V^o$, we let $\Seg(\Bbf)=\bigcup_{i=1}^l\Seg(B_i)$. We let $\Seg_g(\Bbf)$ (resp. $\Seg_r(\Bbf)$) denote the same set of segments, all colored green (resp. red). 

A \textbf{noncrossing tree partition} $\Bbf$ is a set partition of $V^o$ such that any two segments of $\Seg_r(\Bbf)$ are noncrossing and each block of $\Bbf$ is segment-connected. {Note that we intentionally define noncrossing tree partitions using only red segments.} Let $\NCP(T)$ be the poset of noncrossing tree partitions of $T$, ordered by refinement. We give an example of $\NCP(T)$ in Figure \ref{fig_ncp_lattice} where $T$ is the tree whose biclosed sets appear in Figure~\ref{fig_bicfig}. We remark that the lattice of noncrossing tree partitions is not isomorphic to the lattice of noncrossing set partitions in this example. 


\begin{figure}
\centering
\includegraphics{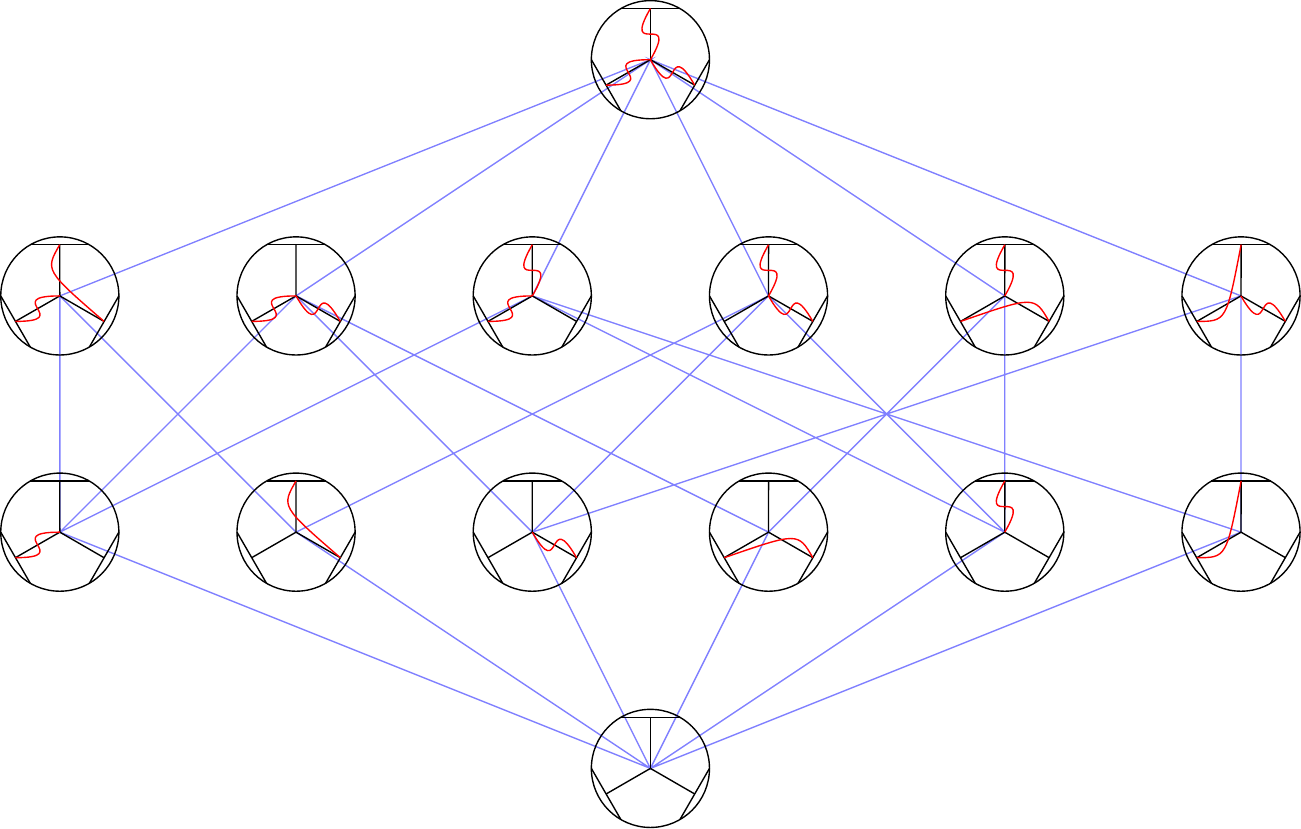}
\caption{\label{fig_ncp_lattice}A lattice of noncrossing tree partitions}
\end{figure}

We record some basic properties of noncrossing tree partitions in the rest of this section.

\begin{lemma}\label{lem_nonsegment}
Let $\Bbf$ be a noncrossing tree partition containing a block $B$. If $u,v\in B$ are distinct vertices such that $[u,v]$ is not a segment, then there exists a vertex $w\in B$ distinct from $u$ and $v$ such that $w\in[u,v]$.
\end{lemma}

\begin{proof}
Let $w\in[u,v]$ such that $[u,w]$ is a segment of maximum length. Since $B$ is segment-connected, there exists a sequence $u=u_0,u_1,\ldots,u_l=v$ of elements of $B$ such that $[u_{i-1},u_i]$ is a segment for all $i$. We further assume that each segment $[u_{i-1},u_i]$ is in $\Seg(B)$ and that $l$ is minimal with this property.

Since $w$ is in $[u,v]$, there exists some segment $[u_{i-1},u_i]$ containing $w$ such that $u_{i-1}\in[u,w]$. Then $w\in[u_i,v]$, so the noncrossing property forces $w\in B$.
\end{proof}

\begin{lemma}\label{lem_segment_decomposition}
Let $\Bbf$ be a noncrossing tree partition. If $B$ is a block of $\Bbf$, then for any distinct vertices $u,v\in B$, there exists a sequence $u=u_0,\ldots,u_l=v$ such that $[u_{i-1},u_i]$ is in $\Seg(B)$ for all $i$ and $[u,v]=[u_0,u_1]\circ\cdots\circ[u_{l-1},u_l]$.
\end{lemma}

\begin{proof}
Let $B$ be a block of $\Bbf$ with at least two elements, and fix distinct vertices $u,v\in B$. We proceed by induction on the length of $[u,v]$. Among the vertices of $[u,v]$, let $u_1$ be the element of $B\setm \{u\}$ minimizing the length of $[u,u_1]$. By Lemma \ref{lem_nonsegment}, $[u,u_1]$ is a segment. By assumption, it is inclusion-minimal, so it is in $\Seg(B)$. If $u_1\neq v$, then by the inductive hypothesis, there exists a sequence $u_1,u_2,\ldots,u_l$ of elements of $B$ such that $u_l=v$ and $[u_{i-1},u_i]\in\Seg(B)$ for all $i$.
\end{proof}

\begin{lemma}\label{lem_crossing_div}
Let $b$ be a vertex in a green segment $[a,c]$, and let $[d,e]$ be some green segment that crosses $[a,c]$. Then either $[a,b]$ or $[b,c]$ crosses $[d,e]$, where $[a,b]$ and $[b,c]$ are both green.
\end{lemma}

\begin{proof}
Suppose neither $[a,b]$ nor $[b,c]$ crosses $[d,e]$. Let $\gamma_1,\gamma_2$ and $\gamma^{\pr}$ be green admissible curves for $[a,b], [b,c],$ and $[d,e]$, respectively. Let $e,e^{\pr}$ be the edges of $[a,b]$ and $[b,c]$ incident to $b$. Orienting $e$ and $e^{\pr}$ away from $b$, let $F$ be the face left of $e$ and $F^{\pr}$ the face left of $e^{\pr}$. Then $\gamma_1$ (resp. $\gamma_2$) has an endpoint at $z(b,F)$ (resp. $z(b,F^{\pr})$). Since $[a,c]$ is a segment, the faces $F$ and $F^{\pr}$ are adjacent. Let $\gamma=\gamma_1\circ\alpha_b^{F,F^{\pr}}\circ\gamma_2$. Then $\gamma$ is a green admissible curve for $[a,c]$ and $\gamma$ does not intersect $\gamma^{\pr}$, a contradiction.
\end{proof}

\subsection{Kreweras complementation}\label{subsec_kreweras}

In this section, we define a bijection on $\NCP(T)$, which we call Kreweras complementation. A representation-theoretic interpretation of this bijection is given in Section~\ref{Sec:SMCs}. 

We define a function $\rho:\ora{FG}(T)\ra\NCP(T)$ as follows. Let $\Fcal\in\ora{FG}(T)$, and let $S$ be the set of segments for which there exists $\Fcal^{\pr}$ with $\Fcal^{\pr}\stackrel{s}{\ra}\Fcal$. Since the arcs in $\Fcal$ are pairwise noncrossing, there exists a realization by simple curves $\{\gamma_p:\ p\in\Fcal\}$ such that the following conditions hold.
\begin{itemize}
\item If $s$ is the largest segment contained in an arc $p\in\Fcal$, then the image of $\gamma_p$ is contained in $D^2\setm (T_{\epsilon}\setm s_{\epsilon})$.
\item For distinct $p,q\in\Fcal$, $\gamma_p$ and $\gamma_q$ are disjoint except possibly at the endpoints.
\item For $p\in\Fcal$, if $\gamma_p$ is marked at $(v,F)$, then $\gamma_p$ contains the point $z(v,F)$.
\end{itemize}

For $s\in S$, let $p$ be the arc marked at the endpoints of $s$. If $p$ is marked at the corners $(v,F),\ (v^{\pr},F^{\pr})$, we let $\gamma_s$ be the subpath of $\gamma_p$ with endpoints $z(v,F)$ and $z(v^{\pr},F^{\pr})$. Since $s$ is a lower label of $\Fcal$, the curve $\gamma_s$ is a red admissible curve for $s$. Since $\gamma_p$ and $\gamma_q$ are disjoint for distinct arcs $p,q$, the collection $\{\gamma_s:\ s\in S\}$ is a noncrossing set of red admissible curves. Hence, $S$ defines a noncrossing tree partition $\Bbf$.

Given $\Fcal$ and $\Bbf$ as above, we set $\rho(\Fcal)=\Bbf$. We prove that $\rho$ is a bijection.

\begin{proposition}\label{prop_rho_bij}
The map $\rho$ is a bijection.
\end{proposition}

\begin{proof}
Given $\Bbf\in\NCP(T)$, let
$$\tilde{\phi}(\Bbf)=\ov{\bigcup_{s\in\Seg(\Bbf)}C_s}.$$
Since $C_s$ is a biclosed set, $\tilde{\phi}(\Bbf)$ is biclosed by Theorem~\ref{thm_biclosed_main}. By Corollary~\ref{cor_CJR}, we have $\phi=\tilde{\phi}\circ\rho$. Since $\eta\circ\phi$ is the identity on $\ora{FG}(T)$, we have $(\eta\circ\tilde{\phi})\circ\rho=\eta\circ\phi=\id_{\ora{FG}(T)}$. Hence, $\rho$ is an injective function.

To show that $\rho$ is surjective, it suffices to prove that $\tilde{\phi}$ is injective and its image is $\pi_{\downarrow}(\Bic(T))$. The latter statement is clear since $C_s\in\pi_{\downarrow}(\Bic(T))$ for any segment $s$, and $\pi_{\downarrow}(\Bic(T))$ is a sublattice of $\Bic(T)$ by Theorem~\ref{thm_eta_phi_main}.

Let $\Bbf\in\NCP(T)$ and set $X=\tilde{\phi}(\Bbf)$. Let $S=\{s \in \text{Seg}(T):\ X\setm\{s\}\in\Bic(T)\}$. To prove that $\tilde{\phi}$ is injective, we show that $\Seg(\Bbf)=S$. 

Suppose $S\setm\Seg(\Bbf)$ is nonempty, and let $t\in S\setm\Seg(\Bbf)$. Since $X\setm\{t\}$ is biclosed, the segment $t$ is not the concatenation of any two segments in $X$. Consequently, $t\in C_s$ for some $s\in\Seg(\Bbf)$ and $s\neq t$. Then $s=t^{(1)}\circ t\circ t^{(2)}$ where $t^{(1)}$ or $t^{(2)}$ is a nonempty segment (or both). Moreover, since $X\setm\{t\}$ is co-closed and $s\in X$, either $t^{(1)}\in X$ or $t^{(2)}\in X$ holds. We may assume without loss of generality that $t^{(1)}$ is a segment in $X$. Since $t\in C_s$, we have $t^{(1)}\in K_s$. By definition, $t^{(1)}=t_1\circ\cdots\circ t_l$ where each $t_i$ is in $C_{s^{\pr}}$ for some $s^{\pr}\in\Seg(\Bbf)$. By repeated application of Lemma~\ref{lem_subsegment}, some $t_i$ is in $K_s$. Hence, $K_s\cap C_{s^{\pr}}$ is nonempty for some $s^{\pr}\in\Seg(\Bbf)$. By Lemma~\ref{lem_crossing_condition}, the red segments $s$ and $s^{\pr}$ are crossing, which is a contradiction.

Now assume $\Seg(\Bbf)\setm S$ is nonempty, and let $s\in\Seg(\Bbf)\setm S$. Then $X\setm\{s\}$ is not biclosed.

Suppose $X\setm\{s\}$ is not closed. Then there exists $t,t^{\pr}\in X$ such that $s=t\circ t^{\pr}$. Without loss of generality, we may assume $t\in C_s$ and $t^{\pr}\in K_s$. Since $t^{\pr}\in X$, there exist a decomposition $t^{\pr}=t_1\circ\cdots\circ t_l$ such that for all $i$, $t_i\in C_{s^{\pr}}$ for some $s^{\pr}\in\Seg(\Bbf)$. By Lemma~\ref{lem_subsegment}, some $t_i$ is in $K_s$. Hence, $K_s\cap C_{s^{\pr}}$ is nonempty for some $s^{\pr}\in\Seg(\Bbf)$, and we again deduce that $s$ and $s^{\pr}$ are crossing.

Suppose $X\setm\{s\}$ is not co-closed. Then there exists $t\in\Seg(T)\setm X$ such that $s\circ t\in X$. We choose the segment $t$ to be minimal with those properties.

Suppose $s\circ t\in C_{s^{\pr}}$ for some $s^{\pr}\in\Seg(\Bbf)$. Since $t\notin X$, $t$ is not in $C_{s^{\pr}}$. Hence, $t\in K_{s\circ t}$ and $s\in C_{s\circ t}$. But this implies $s\in C_{s^{\pr}}$, so $s$ and $s^{\pr}$ are crossing.

Now assume $s\circ t=s_1\circ\cdots\circ s_l,\ l>1$ where for all $i$, $s_i\in C_{s^{\pr}}$ for some $s^{\pr}\in\Seg(\Bbf)$. Since $t\notin X$ and $X=\pi_{\downarrow}(X)$, the segment $t$ is not in $C_{s_1\circ\cdots\circ s_l}$. Consequently, $s\in C_{s_1\circ\cdots\circ s_l}$. We consider two cases: either $s_l$ is a proper subsegment of $t$ or $t$ is a proper subsegment of $s_l$. We note that $t$ is not equal to $s_l$ since $t\notin X$ and $s_l\in X$.

If $s_l$ is a subsegment of $t$, then there exists a segment $t^{\pr}$ such that $s\circ t^{\pr}=s_1\circ\cdots\circ s_{l-1}$. Then $t^{\pr}\in X$ by minimality of $t$. But $t=t^{\pr}\circ s_l$, contrary to $X$ being closed.

If $t$ is a subsegment of $s_l$, then there exists $t^{\pr}$ such that $t^{\pr}\circ t=s_l$. Then $t\in K_{s_l}$ since $t\in K_{s_1\circ\cdots\circ s_l}$, so $t^{\pr}\in C_{s_l}$. Since $s_l\in C_{s^{\pr}}$ for some $s^{\pr}\in\Seg(\Bbf)$, we have $t^{\pr}\in C_{s^{\pr}}$. Since $t^{\pr}$ is a proper subsegment of $s$ that shares an endpoint with $s$, either $t^{\pr}\in K_s$ or $t^{\pr}\in C_s$. If $t^{\pr}\in K_s$, then $s$ and $s^{\pr}$ cross by Lemma~\ref{lem_crossing_condition}. If $t^{\pr}\in C_s$, then $s_1\circ\cdots\circ s_{l-1}\in K_s$. By repeated application of Lemma~\ref{lem_subsegment}, some $s_i$ is in $K_s$. But $s_i\in C_{s^{\pr}}$ for some $s^{\pr}\in\Seg(\Bbf)$, which is again a contradiction.

We have established that $S$ and $\Seg(\Bbf)$ are identical. Hence, the map $\tilde{\phi}$ is injective, and the result follows.
\end{proof}

\begin{proposition}\label{prop_red_green_switch}
For $s,t\in\Seg(T)$, if $s$ and $t$ are noncrossing as green segments, then they are noncrossing as red segments.
\end{proposition}

\begin{proof}
Let $\gamma_s$ and $\gamma_t$ be green admissible curves for $s$ and $t$ that do not intersect. Suppose $s$ has corners $(u,F),\ (v,G)$ as a green segment and $(u,F^{\pr}),\ (v,G^{\pr})$ as a red segment. Define $\gamma_s^{\pr}$ to be the curve $\alpha_v^{G,G^{\pr}}\gamma_s\alpha_u^{F^{\pr},F}$. We apply a slight homotopy to $\gamma_s^{\pr}$ so that $\gamma_s^{\pr}$ is a simple curve and $z(u,F^{\pr})$ and $z(v,G^{\pr})$ are the unique points of distance at most $\epsilon$ from some interior vertex of $T$. Then $\gamma_s^{\pr}$ is a red admissible curve for $s$. If $\gamma_t^{\pr}$ is defined in a similar manner, then it is a red admissible curve for $t$ that does not intersect $\gamma_s^{\pr}$. Hence, $s$ and $t$ are noncrossing as red segments.
\end{proof}

\begin{theorem}\label{thm_kreweras}
Let $\Bbf$ be a noncrossing tree partition, and let $\Fcal=\rho^{-1}(\Bbf)$. Setting
$$S=\{s\in\Seg(T):\ \exists\Fcal\stackrel{s}{\ra}\Fcal^{\pr}\},$$
we have $S=\Seg(\Bbf^{\pr})$ for some noncrossing tree partition $\Bbf^{\pr}$.
\end{theorem}

The noncrossing tree partition $\Bbf^{\pr}$ of Theorem~\ref{thm_kreweras} is called the \textbf{Kreweras complement} of $\Bbf$. Kreweras complementation is a bijection $\Kr:\NCP(T)\ra\NCP(T)$.

\subsection{Red-green trees}\label{subsec_redgreen_trees}

A \textbf{red-green tree} $\Tcal$ is a collection of pairwise noncrossing colored segments such that every pair of vertices in $V^o$ is connected by a sequence of curves in $\Tcal$. The segments in $\Tcal$ are allowed to be red or green. Let $\Tcal_r$ (resp. $\Tcal_g$) be the subset of red (resp. green) segments of $\Tcal$.

That red-green trees are actual trees (i.e. acyclic) will be a consequence of Theorem~\ref{thm_red_green_trees}.

Given $\Fcal\in\ora{FG}(T)$, let $S_r=\{s:\ \exists\Fcal^{\pr}\stackrel{s}{\ra}\Fcal\}$ and $S_g=\{s:\ \exists\Fcal\stackrel{s}{\ra}\Fcal^{\pr}\}$.

\begin{theorem}\label{thm_red_green_trees}
The sets $S_r$ and $S_g$ form the red and green segments of a red-green tree. Conversely, every red-green tree is of this form.
\end{theorem}

\begin{proof}
In the same way that a nonintersecting collection of red admissible curves for segments of $S_r$ was constructed in the definition of $\rho$ in Section~\ref{subsec_kreweras}, one may construct a family of nonintersecting red and green admissible curves for $S_r\cup S_g$. It remains to show that the graph on the interior vertices of $T$ with edge set $S_r\cup S_g$ is connected. This follows from the fact that the graph of facets is connected, and flips preserve connectivity of $S_r\cup S_g$; see Proposition~\ref{graphofSEG(X)} and Figure~\ref{Fig:typesofflip}.

Now let $\Tcal$ be a red-green tree. Then $\Tcal_r$ is the set of minimal segments of a noncrossing tree partition $\Bbf$. Let $X=\tilde{\phi}(\Bbf)$, where $\tilde{\phi}$ is the map to biclosed sets from the proof of Proposition~\ref{prop_rho_bij}. By definition, $X=\ov{\bigcup_{s\in\Tcal_r}C_s}$. We prove that
$$\Seg(T)\setm\pi^{\uparrow}(X)=\ov{\bigcup_{s\in\Tcal_g}K_s}.$$
Since $\bigvee_{s\in\Tcal_g}\eta(K_s)$ is the canonical join-representation of an element in $\ora{FG}(T^{\vee})$, this equality uniquely identifies $\Tcal_g$.

By definition, the set $\Seg(T)\setm\pi^{\uparrow}(X)$ consists of segments $t$ for which $K_t\cap X$ is empty. We first show that $\ov{\bigcup_{s\in\Tcal_g}K_s}$ is a subset of $\Seg(T)\setm\pi^{\uparrow}(X)$. To this end, it suffices to show that $K_s\cap X=\emptyset$ holds whenever $s\in\Tcal_g$. If not, then let $s\in\Tcal_g$ such that $K_s\cap X$ is nonempty, and let $t\in K_s\cap X$. Since $t\in X$, there exist segments $t_1,\ldots,t_l,\ s_1,\ldots, s_l$ such that $t=t_1\circ\cdots\circ t_l$ and $t_i\in C_{s_i}$ for all $i$. Then $t_i\in K_t$ for some $i$. Since $K_t\subseteq K_s$, $t_i$ is in $K_s$. But since $s_i$ and $s$ do not cross, $K_s\cap C_{s_i}$ is empty by Lemma~\ref{lem_crossing_condition_2}, a contradiction.

Now we prove that $\Seg(T)\setm\pi^{\uparrow}(X)$ is a subset of $\ov{\bigcup_{s\in\Tcal_g}K_s}$. Let $t=[u,v]$ be a segment for which $K_t\cap X=\emptyset$. Since $\Tcal$ is a red-green tree, there is a path in $\Tcal$ with edges $s_1,\ldots,s_l$ such that $s_1$ starts at $u$ and $s_l$ ends at $v$. We consider two cases: either $t$ is the concatenation of $s_1,\ldots,s_l$ (i.e. $t=s_1\circ\cdots\circ s_l$), or it is not.

Assume that $t$ is not equal to the concatenation of $s_1,\ldots,s_l$. Then there exists a vertex $w$ incident to an edge $e$ such that two adjacent segments $s_i,s_{i+1}$ both contain $e$ and share an endpoint at $w$. Then $s_i$ and $s_{i+1}$ must have different colors. Up to reversing the order of the segments, we may assume $s_i$ is red and $s_{i+1}$ is green. Let $[w^{\pr},w]$ be the largest common subsegment of $s_i$ and $s_{i+1}$. Since $s_i$ and $s_{i+1}$ are noncrossing, the segment $[w^{\pr},w]$ is in $K_{s_i}\cap C_{s_{i+1}}$. Let $s_i^{\pr},s_{i+1}^{\pr}$ be segments such that $s_i^{\pr}\circ[w^{\pr},w]=s_i$ and $s_{i+1}^{\pr}\circ[w^{\pr},w]=s_{i+1}$. It is possible that $[w^{\pr},w]$ is equal to $s_i$ or $s_{i+1}$ (but not both), in which case $s_i^{\pr}$ or $s_{i+1}^{\pr}$ is a lazy path and thus not a segment.

If $s_i^{\pr}$ is a segment and $i>1$, we claim that it does not cross $s_{i-1}$. Indeed, if $s_i^{\pr}$ and $s_{i-1}$ cross, then $s_{i-1}$ must be green and $C_{s_i^{\pr}}\cap K_{s_{i-1}}$ is nonempty. But this implies $C_{s_i}\cap K_{s_{i-1}}$ is nonempty, a contradiction.

If $s_i^{\pr}$ is not a segment and $i>1$, we claim that $s_{i+1}^{\pr}$ does not cross $s_{i-1}$. If $s_{i+1}^{\pr}$ and $s_{i-1}$ do cross, then $s_{i-1}$ must be red and $K_{s_{i+1}^{\pr}}\cap C_{s_{i-1}}$ is nonempty. But this implies $K_{s_{i+1}}\cap C_{s_{i-1}}$ is nonempty, a contradiction.

Hence, $s_1,\ldots,s_i^{\pr},s_{i+1}^{\pr},\ldots,s_l$ is a sequence of red and green segments connecting the endpoints of $t$ such that the red segments are in $\bigcup_{s\in\Tcal_r}C_s$ and the green segments are in $\bigcup_{s\in\Tcal_g}K_s$. Moreover, adjacent segments are noncrossing. Proceeding inductively, we may assume that $t$ is the concatenation of noncrossing colored segments $t_1,\ldots,t_l$ where each $t_i$ is either a red segment in $\bigcup_{s\in\Tcal_r}C_s$ or a green segment in $\bigcup_{s\in\Tcal_g}K_s$. If $t_1,\ldots,t_l$ are all green segments, then $t\in\ov{\bigcup_{s\in\Tcal_g}K_s}$, as desired.

Assume at least one segment is red, and let $t_i,\ldots,t_j$ be a maximal subsequence of red segments. We prove that $t_i\circ\cdots\circ t_j$ is in $K_t$. If $i>1$, then $t_{i-1}$ is a green segment not crossing $t_i$ such that the concatenation $t_{i-1}\circ t_i$ is a segment. This implies $t_i\in K_{t_{i-1}\circ t_i}$. Similarly, if $j<l$, then $t_{j+1}$ is a green segment not crossing $t_j$, and $t_j$ is in $K_{t_j\circ t_{j+1}}$. Hence, $t_i\circ\cdots\circ t_j$ is in $K_t$. But this implies $t_m$ is in $K_t$ for some $i\leq m\leq j$. As $t_m\in X$, this contradicts the assumption that $K_t\cap X$ is empty.
\end{proof}

Since $\rho$ is a bijection that only depends on the red segments of a facet, Theorem~\ref{thm_red_green_trees} gives a bijection between noncrossing tree partitions and red-green trees. This correspondence encodes Kreweras complementation in a nice way.

\begin{corollary}\label{cor_kreweras_rgtree}
Let $\Bbf$ be a noncrossing tree partition. There exists a unique red-green tree $\Tcal$ whose set of red segments is $\Seg(\Bbf)$. Moreover, the set of green segments of $\Tcal$ is $\Seg(\Kr(\Bbf))$.
\end{corollary}

\subsection{Lattice property}\label{subsec_noncrossing_lattice}

Let $\Pi(V^o)$ be the lattice of all set partitions of $V^o$, ordered by refinement. Recall that the meet of any two set partitions is their common refinement. We prove that $\NCP(T)$ is a meet-subsemilattice of $\Pi(V^o)$ in Theorem \ref{prop_ncp_lattice}. Since $\NCP(T)$ has a top and bottom element, this implies that it is a lattice.

\begin{theorem}\label{prop_ncp_lattice}
The poset $\NCP(T)$ is a lattice. 
\end{theorem}

\begin{proof}
Let $\Bbf,\Bbf^{\pr}$ be two noncrossing tree partitions, and let $\Bbf^{\pr\pr}$ be the common refinement of $\Bbf$ and $\Bbf^{\pr}$. We claim that $\Bbf^{\pr\pr}$ is a noncrossing tree partition and deduce that $\text{NCP}(T)$ is a meet-subsemilattice of $\Pi(V^o)$. We first prove that every block of $\Bbf^{\pr\pr}$ is segment-connected.

Let $B^{\pr\pr}$ be a block of $\Bbf^{\pr\pr}$, and let $u,v\in B^{\pr\pr}$. There exist blocks $B\in\Bbf,\ B^{\pr}\in\Bbf^{\pr}$, each containing $u$ and $v$. We prove by induction that there exists a sequence $u=u_0,\ldots,u_l=v$ of elements of $B^{\pr\pr}$ such that $[u_{i-1},u_i]$ is a segment for all $i$. Among vertices of $[u,v]$, choose $u_1$ such that $[u,u_1]$ is a segment of maximum length. If $u_1=v$, we are done. Since $B$ is segment-connected, there exists a sequence $u=w_0,\ldots,w_m=v$ of elements of $B$ such that $[w_{i-1},w_i]$ is a segment for all $i$. Moreover, these segments may be chosen so that $[u,v]$ is the concatenation of the segments $[w_{i-1},w_i]$. Then $u_1$ is a vertex in $[w_{i-1},w_i]$ for some $i$. As $[w_{i-1},w_i]$ is a segment, this forces $u_1=w_{i-1}$ or $u_1=w_i$. Hence, $u_1\in B$. By a similar argument $u_1\in B^{\pr}$ so $u_1$ is an element of $B^{\pr\pr}$. By induction, we conclude that $B^{\pr\pr}$ is segment-connected.

Let $S=\bigcup_{B^{\pr\pr}\in\Bbf^{\pr\pr}}\Seg(B^{\pr\pr})$ and suppose $[a,b],[c,d]\in S$ such that $[a,b]$ and $[c,d]$ are crossing.

Assume that these segments share a common endpoint, say $b=c$, then they intersect in a common segment $[b,e]$. As $\Bbf$ and $\Bbf^{\pr}$ are noncrossing tree partitions, there exist blocks $B\in\Bbf,\ B^{\pr}\in\Bbf^{\pr}$ such that $a,b,d,e\in B$ and $a,b,d,e\in B^{\pr}$. Hence, $e\in B^{\pr\pr}$. But $[b,e]$ is a subsegment of $[a,b]$ and $[b,d]$, contradicting the minimality of segments in $\Seg(B^{\pr\pr})$.

Now assume that the endpoints are all distinct. Let $B_1^{\pr\pr},B_2^{\pr\pr}$ be blocks in $\Bbf^{\pr\pr}$ such that $a,b\in B_1^{\pr\pr}$ and $c,d\in B_2^{\pr\pr}$. Since $\Bbf^{\pr\pr}$ is the common refinement of $\Bbf$ and $\Bbf^{\pr}$, we may assume without loss of generality that $\Bbf$ contains distinct blocks $B_1$ and $B_2$ such that $a,b\in B_1$ and $c,d\in B_2$. Since $\Bbf$ is noncrossing, either $[a,b]\notin\Seg(B_1)$ or $[c,d]\notin\Seg(B_2)$. Suppose $[a,b]\notin\Seg(B_1)$. Then there exists $a_1\in[a,b]$ such that $[a,a_1]\in\Seg(B_1)$. Then either $[a,a_1]$ or $[a_1,b]$ cross $[c,d]$ by Lemma \ref{lem_crossing_div}. By induction, there exists segments $[a^{\pr},b^{\pr}]\in\Seg(B_1),\ [c^{\pr},d^{\pr}]\in\Seg(B_2)$ such that $[a^{\pr},b^{\pr}]$ and $[c^{\pr},d^{\pr}]$ cross, a contradiction.
\end{proof}

\subsection{Shard intersection order}\label{subsec_shard_intersection_order}

In this section, we prove that the shard intersection order of $\ora{FG}(T)$ is naturally isomorphic to $\NCP(T)$.

Let $B$ be a segment-connected subset of $T^o$, and let $S=\Seg_r(B)$. We define the \textbf{contracted tree} $T_B$ such that
\begin{itemize}
\item $B$ is the set of interior vertices of $T_B$,
\item $S$ is the set of interior edges of $T_B$, and
\item for edges $e$ with one endpoint $u$ in $B$ and the other endpoint not between two vertices of $B$, there is an edge from $u$ to the boundary in the direction of $e$.
\end{itemize}

As in Proposition~\ref{prop_biclosed_facial_interval}, we may compute the facial intervals of $\ora{FG}(T)$ as follows.

\begin{proposition}\label{prop_quotient_facial_interval}
Let $\Fcal\in\ora{FG}(T)$, and let $s_1,\ldots,s_k$ be a set of segments for which there exists flips $\Fcal\stackrel{s_i}{\ra}\Fcal^{\pr}$ for each $i$. Let $\Bbf=(B_1,\ldots,B_l)$ be the noncrossing tree partition with segments $\Seg(\Bbf)=\{s_1,\ldots,s_k\}$. Let $T_i$ denote the contracted tree $T_{B_i}$. Then
$$[\Fcal,\bigvee\Fcal^{\pr}]\cong\ora{FG}(T_1)\times\cdots\times\ora{FG}(T_l),$$
where the join is taken over $\Fcal^{\pr}$ for which $\Fcal\stackrel{s_i}{\ra}\Fcal^{\pr}$ for some $s_i$ (see Figure~\ref{Fig:Facial_Int}).
\end{proposition}

\begin{figure}
$$\includegraphics[scale=1]{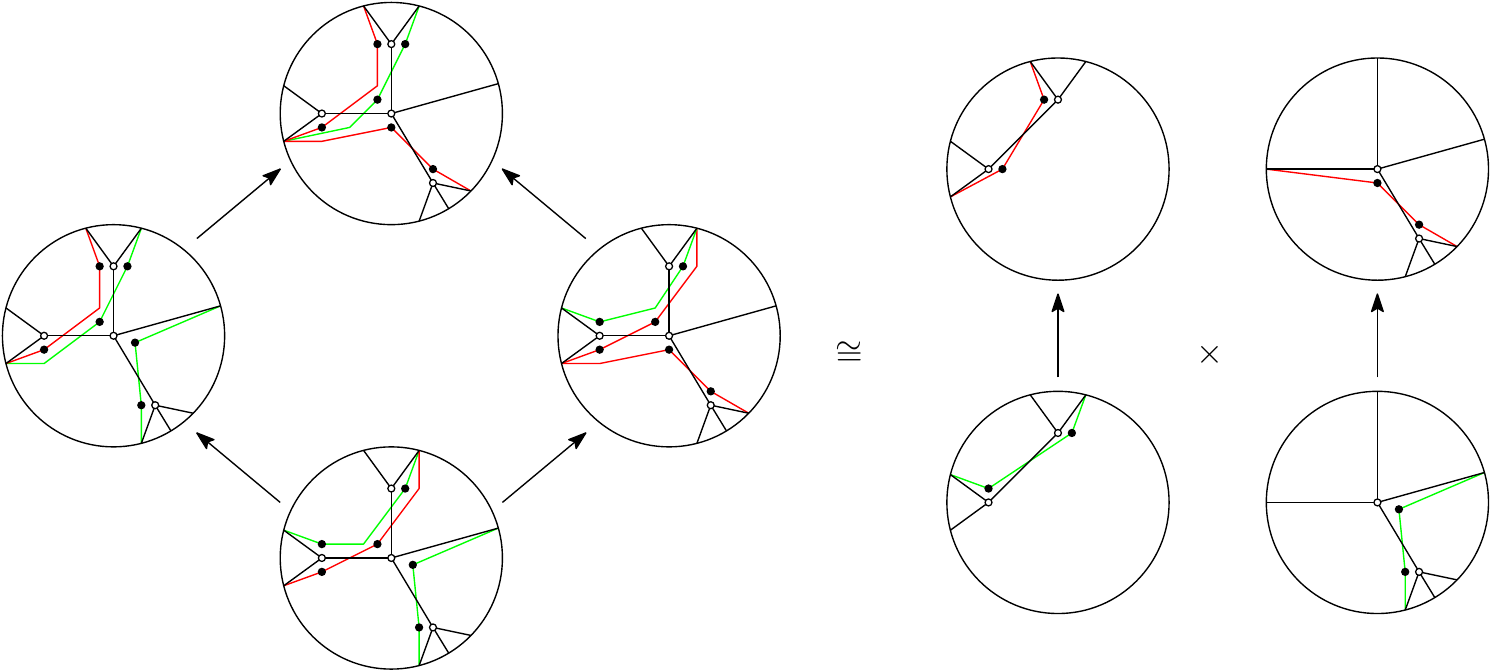}$$
\caption{An example of the isomorphism appearing in Proposition~\ref{prop_quotient_facial_interval} where $\mathcal{F}$ is a facet from the oriented flip graph in Figure~\ref{3dimorflipgraph}.}
\label{Fig:Facial_Int}
\end{figure}

\begin{proof}
Let $X$ be the biclosed set $\pi^{\uparrow}(\phi(\Fcal))$. Then $\{s_1,\ldots,s_k\}$ is the set of segments for which $X\cup\{s_i\}$ is biclosed, and $\eta(X\cup\{s_i\})=\Fcal^{\pr}$ where $\Fcal^{\pr}$ is the facet obtained by flipping $\Fcal$ at $s_i$. Set $Y=X\cup\ov{\{s_1,\ldots,s_k\}}$.

Let $\Bbf=\{B_1,\ldots,B_l\}$ be the noncrossing tree partition with $\Seg(\Bbf)=\{s_1,\ldots,s_k\}$, and let $T_i$ be the contracted tree $T_{B_i}$. By Proposition~\ref{prop_biclosed_facial_interval}, the interval $[X,Y]$ is isomorphic to $\Bic(T_1)\times\cdots\times\Bic(T_l)$.

As usual, we let $\Theta$ denote the lattice congruence that identifies $\ora{FG}(T)$ with $\Bic(T)/\Theta$. We let $\Theta_i$ denote the corresponding lattice congruence on $\Bic(T_i)$. By Lemma~\ref{lem_cong_elementary}, the quotient interval $[X,Y]/\Theta$ is isomorphic to $[\Fcal,\bigvee\Fcal^{\pr}]$. Hence, we prove
$$[X,Y]/\Theta\cong\Bic(T_1)/\Theta_1\times\cdots\times\Bic(T_l)/\Theta_l.$$

Given a segment $s$ supported by $T_i$, we let $C_s^{i}$ (resp. $K_s^{i}$) denote the intersection $C_s\cap\Seg(T_i)$ (resp. $K_s\cap\Seg(T_i)$), and we define maps $\pi_{\downarrow}^{i}$ and $\pi^{\uparrow}_{i}$ by the congruence $\Theta_i$. Explicitly, we have
\begin{align*}
\pi_{\downarrow}^{i}(Z) &= \{s\in\Seg(T_i):\ C_s^{i}\subseteq Z\}\ \mbox{and}\\
\pi^{\uparrow}_{i}(Z) &= \{s\in\Seg(T_i):\ K_s^{i}\cap Z\neq\emptyset\}.
\end{align*}

Let $Z,Z^{\pr}\in[X,Y]$. Then $Z=X\cup\bigcup_{i=1}^lZ_i$ and $Z^{\pr}=X\cup\bigcup_{i=1}^lZ_i^{\pr}$ for some (unique) $Z_i,Z_i^{\pr}\in\Bic(T_i)$. We prove that $Z\equiv Z^{\pr}\mod\Theta$ if and only if $Z_i\equiv Z_i^{\pr}\mod\Theta_i$ for all $i$.

Suppose $Z\equiv Z^{\pr}\mod\Theta$, and fix $i\in\{1,\ldots,l\}$. To prove that $Z_i\equiv Z_i^{\pr}\mod\Theta_i$, it suffices to show that $\pi^{\uparrow}_{i}(Z_i)=\pi^{\uparrow}(Z)\cap\Seg(T_i)$. If $s\in\pi^{\uparrow}_i(Z_i)$, then $K_s^i\cap Z_i$ is nonempty. But this implies $K_s\cap Z$ is nonempty, so $s\in\pi^{\uparrow}(Z)\cap\Seg(T_i)$. Conversely, if $s\in\pi^{\uparrow}(Z)\cap\Seg(T_i)$, then $K_s\cap Z$ is nonempty. Since $\pi^{\uparrow}(X)=X$, we deduce that $K_s\cap\bigcup_{j=1}^lZ_j$ is nonempty. But $K_s\cap Z_j=\emptyset$ whenever $j\neq i$ since blocks $B_i$ and $B_j$ are noncrossing. Hence, $s\in\pi^{\uparrow}_i(Z_i)$, as desired.

Now assume $Z_i\equiv Z_i^{\pr}\mod\Theta_i$ for all $i$. Since $\pi^{\uparrow}_i(Z_i)\subseteq\pi^{\uparrow}(Z)$ and $\pi^{\uparrow}$ is idempotent, we have
\begin{align*}
\pi^{\uparrow}(Z) &= \pi^{\uparrow}(X\cup\bigcup_{i=1}^l\pi^{\uparrow}_i(Z_i))\\
&= \pi^{\uparrow}(X\cup\bigcup_{i=1}^l\pi^{\uparrow}_i(Z_i^{\pr}))\\
&= \pi^{\uparrow}(Z^{\pr}).
\end{align*}

Therefore, $Z\equiv Z^{\pr}\mod\Theta$.

\end{proof}

\begin{theorem}\label{Thm:rhocircPsi}
The map $\rho\circ\psi^{-1}:\Psi(\ora{FG}(T))\ra\NCP(T)$ is a Kreweras-equivariant isomorphism of posets.
\end{theorem}

\begin{proof}
Let $\Fcal$ be an element of $\ora{FG}(T)$, and let
$$\Fcal^{\pr}=\bigwedge\{\Fcal^{\pr\pr}:\ \Fcal^{\pr\pr}\ra\Fcal\}.$$
Let $S=\Seg(\rho(\Fcal))$. By Lemma~\ref{lem_cu_facial_interval}, $\Fcal$ is equal to
$$\bigvee\{\Fcal^{\pr\pr}:\ \Fcal^{\pr}\stackrel{s}{\ra}\Fcal^{\pr\pr},\ s\in S\}.$$
Let $\rho(\Fcal)=(B_1,\ldots,B_l)$, and let $T_i$ be the contracted tree $T_{B_i}$. By Proposition~\ref{prop_quotient_facial_interval}, the interval $[\Fcal^{\pr},\Fcal]$ is isomorphic to $\ora{FG}(T_1)\times\cdots\times\ora{FG}(T_l)$.

The set $\psi(\Fcal)$ is defined to be the set of labels $s$ such that there exists a covering relation $\Fcal^{(1)}\stackrel{s}{\ra}\Fcal^{(2)}$ where $\Fcal^{\pr}\leq\Fcal^{(1)}\lessdot\Fcal^{(2)}\leq\Fcal$. Hence,
$$\psi(\Fcal)=\bigsqcup_{i=1}^l\Seg(T_i)=\ov{\Seg(\rho(\Fcal))}.$$

From this description, it is clear that $\psi$ is a bijection. Hence, the inverse $\psi^{-1}$ exists, and the composite map $\rho\circ\psi^{-1}$ is a bijection. Since the Kreweras complement is defined for both $\Psi(\ora{FG}(T))$ and $\NCP(T)$ via the bijections $\rho$ and $\psi$, the Kreweras-equivariance is immediate. If $\Fcal_1,\Fcal_2\in\ora{FG}(T)$ satisfy $\psi(\Fcal_1)\subseteq\psi(\Fcal_2)$, then the corresponding noncrossing tree partitions are ordered by refinement. Conversely, it is clear that if $\rho(\Fcal_1)\leq\rho(\Fcal_2)$, then any segment in $\psi(\Fcal_1)$ is contained in $\psi(\Fcal_2)$. Hence, the bijection $\rho\circ\psi^{-1}$ is an isomorphism of posets.
\end{proof}

\section{Trees and their tiling algebras}\label{Sec:FinDimAlg}

Given a tree $T$ embedded in $D^2$, we explain how one associates to it a finite dimensional algebra $\Lambda_T$. The construction we present is useful in that the indecomposable modules of the resulting algebra $\Lambda_T$, as we will show (see Corollary~\ref{stringsegmentbij}), are parameterized by the segments of $T.$ We also classify the extensions between indecomposable $\Lambda_T$-modules (see Propositions~\ref{separatesegext} and \ref{Prop:commonendext} and Theorems~\ref{shareendptext} and \ref{crossingnonsplitext}), which will be useful in our applications. Before presenting the definition of $\Lambda_T$, we review some background on path algebras, quiver representations, and string modules. At the end of this section, we show that the oriented flip graph of $T$ is isomorphic to the lattice of torsion-free classes in $\Lambda_T$-mod, and we show that the lattice of noncrossing tree partitions is isomorphic to the lattice of wide subcategories of $\Lambda_T$-mod.

\subsection{Path algebras and quiver representations} Following \cite{ass06}, let $Q$ be a given quiver. We define a \textbf{path} of \textbf{length} $\ell \ge 1$ to be an expression $\alpha_1\alpha_2\cdots\alpha_\ell$ where $\alpha_i \in Q_1$ for all $i \in [\ell]$ and $s(\alpha_i) = t(\alpha_{i+1})$ for all $i \in [\ell-1]$. We may visualize such a path in the following way $$\begin{xy} 0;<1pt,0pt>:<0pt,-1pt>:: 
(0,0) *+{\cdot} ="0",
(25,0) *+{\cdot} ="1",
(50,0) *+{\cdot} ="2",
(75,0) *+{\cdot} ="3",
(100,0) *+{\cdots} ="4",
(125,0) *+{\cdot} ="5",
(150,0) *+{\cdot} ="6",
(175,0) *+{\cdot} ="7",
"1", {\ar_{\alpha_1}"0"},
"2", {\ar_{\alpha_2}"1"},
"3", {\ar"2"},
"6", {\ar"5"},
"7", {\ar_{\alpha_\ell}"6"},
\end{xy}.$$ Furthermore, the \textbf{source} (resp. \textbf{target}) of the path $\alpha_1\alpha_2\cdots\alpha_\ell$ is $s(\alpha_\ell)$ (resp. $t(\alpha_1)$). Let $Q_\ell$ denote the set of all paths in $Q$ of length $\ell$. We also associate to each vertex $i \in Q_0$ a path of length $\ell = 0$, denoted $\varepsilon_i$, that we will refer to as the \textbf{lazy path} at $i$. 

\begin{definition}
Let $Q$ be a quiver. The \textbf{path algebra} of $Q$ is the $\Bbbk$-algebra generated by all paths of length $\ell \ge 0$. Throughout this paper, we assume that $\Bbbk$ is algebraically closed. The multiplication of two paths $\alpha_1\cdots \alpha_\ell \in Q_\ell$ and $\beta_1\cdots \beta_k \in Q_k$ is given by the following rule $$\alpha_1\cdots \alpha_\ell\cdot\beta_1\cdots \beta_k = \left\{\begin{array}{lrl}\alpha_1 \cdots\alpha_\ell\beta_1\cdots\beta_k \in Q_{\ell + k}& : & s(\alpha_\ell) = t(\beta_1) \\ 0 & : & s(\alpha_\ell) \neq t(\beta_1).\end{array}\right.$$ We will denote the path algebra of $Q$ by $\Bbbk Q$. Note also that as $\Bbbk$-vector spaces we have $$\Bbbk Q = \bigoplus_{\ell = 0}^\infty \Bbbk Q_\ell$$ where $\Bbbk Q_\ell$ is the $\Bbbk$-vector space of all paths of length $\ell$.
\end{definition}

In this paper, we study certain quivers $Q$ which have \textbf{oriented cycles}. We say a path of length $\ell \ge 0$ $\alpha_1 \cdots \alpha_\ell \in Q_\ell$ is an \textbf{oriented cycle} if $t(\alpha_1) = s(\alpha_\ell)$. We denote by $\Bbbk Q_{\ell,\text{cyc}} \subset \Bbbk Q_\ell$ the subspace of all oriented cycles of length $\ell \ge 0$. If a quiver $Q$ possesses any oriented cycles of length $\ell \ge 1$, we see that $\Bbbk Q$ is infinite dimensional. If $Q$ has no oriented cycles, we say that $Q$ is \textbf{acyclic}.

In order to avoid studying infinite dimensional algebras, we will add relations to path algebras whose quivers contain oriented cycles in such a way that we obtain finite dimensional quotients of path algebras. The relations we add are those coming from an \textbf{admissible} ideal $I$ of $\Bbbk Q$ meaning that $$I \subset \bigoplus_{\ell = 2}^\infty \Bbbk Q_\ell.$$ If $I$ is an admissible ideal of $\Bbbk Q$, we say that $(Q,I)$ is a \textbf{bound quiver} and that $\Bbbk Q/I$ is a \textbf{bound quiver algebra}.

In this paper, we study modules over a bound quiver algebra $\Bbbk Q/I$ by studying certain representations of $Q$ that are ``compatible" with the relations coming from $I$. A \textbf{representation} $V = ((V_i)_{i \in Q_0}, (\varphi_\alpha)_{\alpha \in Q_1})$ of a quiver $Q$  is an assignment of a $\Bbbk$-vector space $V_i$ to each vertex $i$ and a $\Bbbk$-linear map $\varphi_\alpha: V_{s(\alpha)} \rightarrow V_{t(\alpha)}$ to each arrow $\alpha \in Q_1$. If $\rho \in \Bbbk Q$, it can be expressed as $${\rho = \sum_{i = 1}^m c_i\alpha^{(i)}_{1}\cdots \alpha^{(i)}_{k_i}}$$ where $c_i \in \Bbbk$ and $\alpha^{(i)}_{1}\cdots \alpha^{(i)}_{k_i}\in Q_{i_k}$ so when considering a representation $V$ of $Q$, we define $$\varphi_\rho := \sum_{i = 1}^m c_i\varphi_{\alpha^{(i)}_{1}}\cdots \varphi_{\alpha^{(i)}_{k_i}}.$$  If we have a bound quiver $(Q,I)$, we define a representation of $Q$ \textbf{bound by} $I$ to be a representation of $Q$ where $\varphi_\rho = 0$ if $\rho \in I.$ We say a representation of $Q$ bound by $I$ is \textbf{finite dimensional} if $\dim_\Bbbk V_i < \infty$ for all $i \in Q_0.$ It turns out that $\Bbbk Q/I$-mod is equivalent to the category of finite dimensional representations of $Q$ bound by $I$. In the sequel, we use this fact without mentioning it further. Additionally, the \textbf{dimension vector} of $V \in \Bbbk Q/I$-mod is the vector $\underline{\dim}(V):=(\dim_\Bbbk V_i)_{i\in Q_0}$ and the \textbf{dimension} of $V$ is defined as $\dim_\Bbbk(V) = \sum_{i \in Q_0} \dim_\Bbbk V_i$. The \textbf{support} of $V \in \Bbbk Q/I$-mod is the set $\text{supp}(V) := \{i\in Q_0 : V_i \neq 0\}$.

In this paper, we will focus on a special type of bound quiver algebras known as \textbf{gentle algebras}. Gentle algebras have a simple combinatorial parameterization of their indecomposable modules in terms of string modules. The string modules and the homological properties of string modules of the tiling algebra $\Lambda_T$ (see Section~\ref{Sec:stringalgtree}) defined by $T$ will be closely related to the combinatorics $T$ that we have developed in the preceding sections. A \textbf{gentle algebra} $\Lambda = \Bbbk Q/I$ is a bound quiver algebra that satisfies the following conditions:

\begin{itemize}
\item[i)] For each vertex of $Q$ is the starting point of at most two arrows and the ending point of at most two arrows.
\item[ii)] For each arrow $\beta \in Q_1$ there is at most one arrow $\alpha \in Q_1$ such that $\beta\alpha \not \in I$, and there is at most one arrow $\gamma \in Q_1$ such that $\gamma\beta \not \in I$.
\item[iii)] For each arrow $\beta \in Q_1$, there is at most one arrow $\delta \in Q_1$ such that $\beta\delta \in I$, and there is at most one arrow $\mu \in Q_1$ such that $\mu\beta \in I$.
\item[iv)] $I$ is generated by paths of length 2.
\end{itemize}



\noindent A \textbf{string} in $\Lambda$ is a sequence $$w = x_1 \stackrel{\alpha_1}{\longleftrightarrow} x_2 \stackrel{\alpha_2}{\longleftrightarrow} \cdots \stackrel{\alpha_{m}}{\longleftrightarrow} x_{m+1} $$
where each $x_i \in Q_0$ and each $\alpha_i \in Q_1$ or $\alpha_i \in Q_1^{-1}:= \{\text{formal inverses of arrows of }$Q$\}$. We require that each $\alpha_i$ \textbf{connects} $x_i$ and $x_{i+1}$ (i.e. either $s(\alpha_i) = x_i$ and $t(\alpha_i) = x_{i+1}$ or $s(\alpha_i) = x_{i+1}$ and $t(\alpha_i) = x_i$ where if $\alpha_i \in Q_1^{-1}$ we define $s(\alpha_i) := t(\alpha_i^{-1})$ and $t(\alpha_i) := s(\alpha_i^{-1})$) and that $w$ contains no \textbf{substrings} of $w$ of the following forms:

$\begin{array}{rll}
i) & x \stackrel{\beta}{\longrightarrow} y \stackrel{\beta^{-1}}{\longleftarrow} x \text{ or } x \stackrel{\beta}{\longleftarrow} y \stackrel{\beta^{-1}}{\longrightarrow} x,\\
ii) & x_{i_1} \stackrel{\beta_1}{\longrightarrow} x_{i_2} \cdots x_{i_{s}} \stackrel{\beta_{s}}{\longrightarrow} x_{i_{s+1}} \text{ or } x_{i_1} \stackrel{\gamma_1}{\longleftarrow} x_{i_2} \cdots x_{i_{s}} \stackrel{\gamma_{s}}{\longleftarrow} x_{i_{s+1}} \text{ where $\beta_s \cdots \beta_1, \gamma_1\cdots \gamma_s \in I$}.
\end{array}$

\noindent In other words, $w$ is an irredundant walk in $Q$ that avoids the relations imposed by $I$. By convention, we consider $w$ to be a different word in the vertices of $Q$ than $w^{-1} := x_{m+1} \stackrel{\alpha_m}{\longleftrightarrow} x_m \stackrel{\alpha_{m-1}}{\longleftrightarrow} \cdots \stackrel{\alpha_{1}}{\longleftrightarrow} x_{1}$. We say the string $w$ is \textbf{cyclic} if $x_1 = x_{m+1}$ and we say a cyclic string is a \textbf{band} if $$w^k:= \underbracket{x_1 \stackrel{\alpha_1}{\longleftrightarrow} x_2 \stackrel{\alpha_2}{\longleftrightarrow} \cdots \stackrel{\alpha_{m}}{\longleftrightarrow} x_1 \stackrel{\alpha_1}{\longleftrightarrow} x_2 \stackrel{\alpha_2}{\longleftrightarrow} \cdots \stackrel{\alpha_{m}}{\longleftrightarrow} x_{1}\cdots x_1 \stackrel{\alpha_1}{\longleftrightarrow} x_2 \stackrel{\alpha_2}{\longleftrightarrow} \cdots \stackrel{\alpha_{m}}{\longleftrightarrow} x_{1}}_{\text{$k$ copies of $w$}}$$ is a string but $w$ is not a proper power of another string $u$ (i.e. there does not exist an integer $s \ge 2$ such that $w = u^s$).

Let $w$ be a string in $\Lambda$. The \textbf{string module} defined by $w$ is the bound quiver representation $ M(w) := ((V_i)_{i \in Q_0}, (\varphi_\alpha)_{\alpha\in Q_1})$ where
$$\begin{array}{cccccccccccc}
V_i & := & \left\{\begin{array}{lcl} \Bbbk^{s_j} &: & i = x_j \text{ for some } j \in [m+1]\\ 0 & : & \text{otherwise} \end{array}\right. 
\end{array}$$
where $s_j := \#\{k \in [m+1] : \ x_k = x_j\}$ and the action of $\varphi_\alpha$ is induced by the relevant identity morphisms if $\alpha$ lies on $w$ and is zero otherwise. One observes that $M(w)\cong M(w^{-1})$. If, in addition, $w$ is a band, it defines a \textbf{band module} $M(w,n,\phi) := ((V_i)_{i \in Q_0}, (\varphi_\alpha)_{\alpha\in Q_1})$ where $$\begin{array}{cccccccccccc}
V_i & := & \left\{\begin{array}{lcl} \Bbbk^{n} &: & i = x_j \text{ for some } j \in [m+1]\\ 0 & : & \text{otherwise} \end{array}\right. 
\end{array}$$ for each choice of $n \in \mathbb{N}$ and $\phi \in \text{Aut}(\Bbbk^n)$. The action of $\varphi_\alpha$ is induced by relevant identity morphisms (resp. by $\phi$) if $\alpha = \alpha_j$ for some $j \in [m-1]$ (resp. $\alpha = \alpha_m$). 


If $\Bbbk Q/I$ is a representation-finite gentle algebra, it follows from \cite{wald1985tame} that set of indecomposable $\Bbbk Q/I$-modules, denoted $\text{ind}(\Bbbk Q/I\text{-mod})$, consists of exactly the string modules $M(w)$ where $w$ is a string in $\Bbbk Q/I$. 

\begin{example}\label{stringsA3}
Let $Q$ denote the quiver shown below. Then $\Bbbk Q/I = \Bbbk Q/\langle \beta\alpha, \gamma\beta, \alpha\gamma \rangle$ is a gentle algebra. $$\begin{array}{ccccccccccccccc}
\raisebox{-.4in}{$Q$} & \raisebox{-.4in}{=} & \raisebox{-.2in}{$\begin{xy} 0;<1pt,0pt>:<0pt,-1pt>:: 
(0,20) *+{1} ="0",
(20,0) *+{2} ="1",
(40,20) *+{3} ="2",
"1", {\ar_{\alpha}"0"},
"2", {\ar_{\beta}"1"},
"0", {\ar_{\gamma}"2"},
\end{xy}$}  \end{array}$$

\noindent The algebra $\Bbbk Q/I$ has the following string modules. $$\begin{array}{rcccrcccrccccccc}
\raisebox{-.4in}{$M(1)$} & \raisebox{-.4in}{=} & \raisebox{-.2in}{$\begin{xy} 0;<1pt,0pt>:<0pt,-1pt>:: 
(0,20) *+{\Bbbk} ="0",
(20,0) *+{0} ="1",
(40,20) *+{0} ="2",
"1", {\ar_{0}"0"},
"2", {\ar_{0}"1"},
"0", {\ar_{0}"2"},
\end{xy}$} & & \raisebox{-.4in}{$M(2)$} & \raisebox{-.4in}{=} & \raisebox{-.2in}{$\begin{xy} 0;<1pt,0pt>:<0pt,-1pt>:: 
(0,20) *+{0} ="0",
(20,0) *+{\Bbbk} ="1",
(40,20) *+{0} ="2",
"1", {\ar_{0}"0"},
"2", {\ar_{0}"1"},
"0", {\ar_{0}"2"},
\end{xy}$} & & \raisebox{-.4in}{$M(3)$} & \raisebox{-.4in}{=} & \raisebox{-.2in}{$\begin{xy} 0;<1pt,0pt>:<0pt,-1pt>:: 
(0,20) *+{0} ="0",
(20,0) *+{0} ="1",
(40,20) *+{\Bbbk} ="2",
"1", {\ar_{0}"0"},
"2", {\ar_{0}"1"},
"0", {\ar_{0}"2"},
\end{xy}$}\\
\raisebox{-.4in}{$M(1 \stackrel{\alpha}{\longleftarrow} 2)$} & \raisebox{-.4in}{=} & \raisebox{-.2in}{$\begin{xy} 0;<1pt,0pt>:<0pt,-1pt>:: 
(0,20) *+{\Bbbk} ="0",
(20,0) *+{\Bbbk} ="1",
(40,20) *+{0} ="2",
"1", {\ar_{1}"0"},
"2", {\ar_{0}"1"},
"0", {\ar_{0}"2"},
\end{xy}$} & & \raisebox{-.4in}{$M(2 \stackrel{\beta}{\longleftarrow} 3)$} & \raisebox{-.4in}{=} & \raisebox{-.2in}{$\begin{xy} 0;<1pt,0pt>:<0pt,-1pt>:: 
(0,20) *+{0} ="0",
(20,0) *+{\Bbbk} ="1",
(40,20) *+{\Bbbk} ="2",
"1", {\ar_{0}"0"},
"2", {\ar_{1}"1"},
"0", {\ar_{0}"2"},
\end{xy}$} & & \raisebox{-.4in}{$M(3 \stackrel{\gamma}{\longleftarrow} 1)$} & \raisebox{-.4in}{=} & \raisebox{-.2in}{$\begin{xy} 0;<1pt,0pt>:<0pt,-1pt>:: 
(0,20) *+{\Bbbk} ="0",
(20,0) *+{0} ="1",
(40,20) *+{\Bbbk} ="2",
"1", {\ar_{0}"0"},
"2", {\ar_{0}"1"},
"0", {\ar_{1}"2"},
\end{xy}$}
\end{array}$$
\end{example}

\subsection{The tiling algebra of a tree}\label{Sec:stringalgtree}

Let $T$ be a tree embedded in $D^2$. Then $T$ defines a bound quiver, denoted $(Q_T,I_T)$, as follows. Let $Q_T$ be quiver whose vertices are in bijection with the edges of $T$ that contain no leaves and whose arrows are exactly those of the form $e_1 \stackrel{\alpha}{\longrightarrow} e_2$ satisfying:

$\begin{array}{rl}
i) & \text{$e_1$ and $e_2$ define a corner of $T$,}\\
ii) & \text{$e_2$ is {counter}clockwise from $e_1$.} 
\end{array}$

\noindent The admissible ideal $I_T$ is, by definition, generated by the relations $\alpha\beta$ where $\alpha: e_2 \longrightarrow e_3$ defines the corner $(v,F)$ and $\beta: e_1 \longrightarrow e_2$ defines the corner $(v,G)$. We define $\Lambda_T := \Bbbk Q_T/I_T$ and refer to this as the \textbf{tiling algebra} of $T$.

\begin{example}\label{LambdaTExample}

In Figure~\ref{LambdaTFig}, we show three trees. In the left tree in Figure~\ref{LambdaTFig}, we illustrate how $T_1$ determines the quiver $Q_{T_1} = 1 \stackrel{\beta}{\longrightarrow} 2 \stackrel{\alpha}{\longrightarrow} 3.$ The algebra defined by $T_1$ is $\Lambda_{T_1} = \Bbbk Q_{T_1}/I_{T_1}$ where $I_{T_1} = \langle \alpha\beta\rangle$. Also note that $Q_{T_2} \cong Q_{T_3} \cong Q$ and $\Lambda_{T_2} \cong \Lambda_{T_3} \cong \Lambda$ where $Q$ is the quiver from Example~\ref{stringsA3} and $\Lambda$ is the algebra from Example~\ref{stringsA3}.

\begin{figure}[h]
$$\begin{array}{ccccccccccccccc}
\includegraphics[scale=.75]{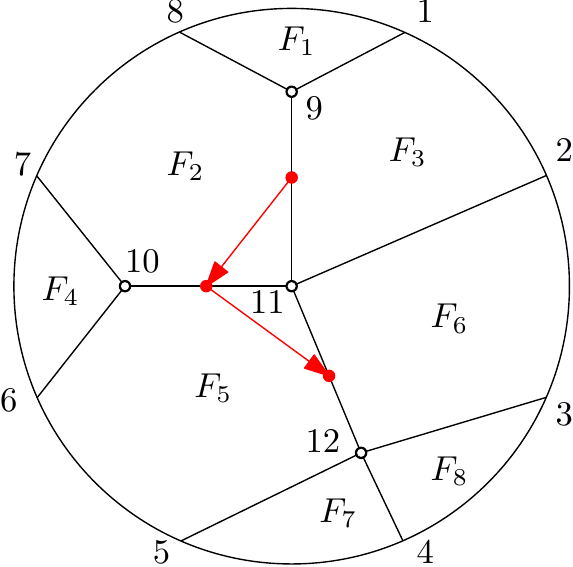} & \includegraphics[scale=.75]{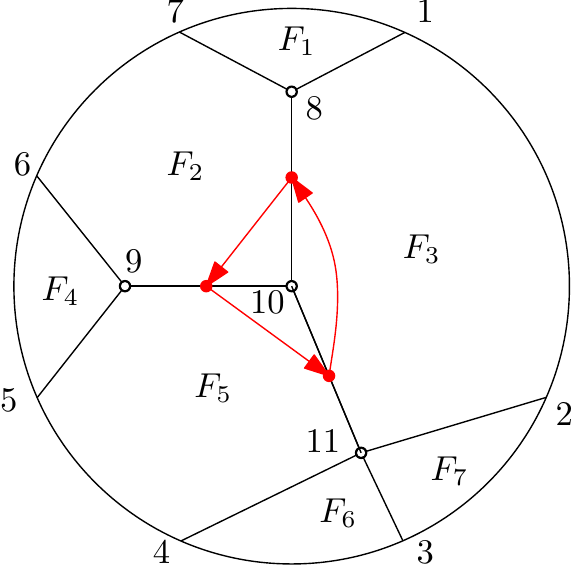} & \includegraphics[scale=.75]{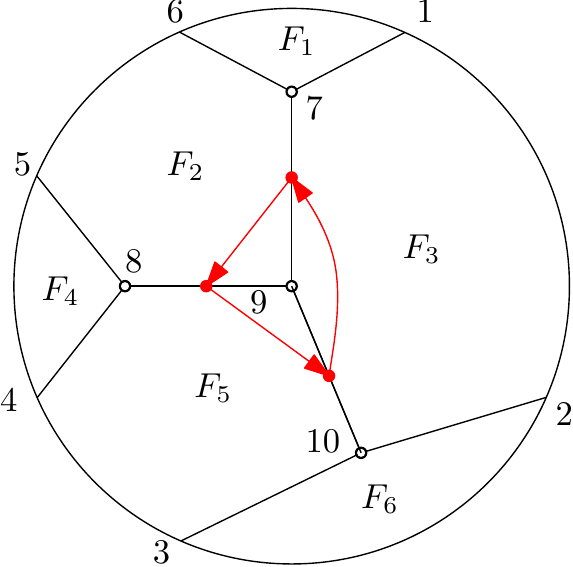}\\
T_1 & T_2 & T_3
\end{array}$$
\caption{}
\label{LambdaTFig}
\end{figure}
\end{example}

We remark that the term tiling algebra first appeared in \cite{simoes2016endomorphism} where a tiling algebra is defined by a partial triangulation of a polygon. The definition of a tiling algebra from \cite{simoes2016endomorphism} agrees with our definition of $\Lambda_\mathcal{P}$ in Section~\ref{Sec:Poly_Sub} which is canonically isomorphic to $\Lambda_T$.

\begin{proposition}\label{lambdarepfin}
The algebra $\Lambda_T$ is a gentle algebra. Furthermore, the algebra $\Lambda_T$ is representation-finite and its indecomposables are exactly the string modules.
\end{proposition}
\begin{proof}
The first assertion follows from \cite[Proposition 3.2]{simoes2016endomorphism}. To prove the second assertion, it is enough to observe that any string $w$ in $\Lambda_T$ can be regarded as a full, connected subquiver of $Q_T$ that avoids the relations imposed by $I_T$. In particular, $w$ has at most one arrow from any cycle in $Q_T$ so $w$ is not a cyclic string. This description of the strings in $\Lambda_T$ implies that there are only finitely many strings in $\Lambda_T$ and there are no bands in $\Lambda_T$. Thus $\Lambda_T$ is representation-finite. \end{proof}




\begin{corollary}\label{stringsegmentbij}
The following hold for the tiling algebra $\Lambda_T$.
\begin{enumerate}
\item[1.] Assume $M(w) := ((V_i)_{i \in Q_0}, (\varphi_\alpha)_{\alpha\in Q_1})$ is a string module of $\Lambda_T$. Then $\dim_\Bbbk(V_i) = 1$ if $i \in \text{supp}(M(w))$ and $\dim_\Bbbk(V_i) = 0$ otherwise.
\item[2.] The map $\text{ind}(\Lambda_T) \longrightarrow \text{Seg}(T)$ defined by $$M(w) \longmapsto s_w := (v_0,\ldots, v_t)$$ where each $v_i$ is a vertex of $T$ belonging to some $e_j \in \text{supp}(M(w))$ and where each pair $v_i$ and $v_{i+1}$ belongs to a common $e_j \in \text{supp}(M(w))$ is a bijection.
\end{enumerate}
\end{corollary}
\begin{proof}
Assertion $1.$ follows from the proof of Proposition~\ref{lambdarepfin}.

To prove assertion 2., we note that, as in the proof of Proposition~\ref{lambdarepfin}, any string module $M(w) \in \text{ind}(\Lambda_T)$ can be regarded as a full, connected subquiver of $Q_T$ that avoids the relations imposed by $I_T$. With this identification, we observe that $M(w)$ is equivalent to a sequence of interior vertices $(v_0, \ldots, v_t)$ of $Q_T$ with the property that any two edges $(v_{i-1}, v_i)$ and $(v_i, v_{i+1})$ are contained in a common face of $T$. Thus the given map is a bijection.
\end{proof}

We now present a description of the spaces of extensions between indecomposable $\Lambda_T$-modules. These results, especially Theorems~\ref{crossingnonsplitext}, appear to be new. These results generalize, in the finite representation type case, the description of extensions between indecomposables found in \cite{cs14}. The proofs of the following results depend on several lemmas presented in Section~\ref{ingredients}.

\begin{proposition}\label{separatesegext}
Let $0 \to M(u) \stackrel{f}{\to} X  \stackrel{g}{\to} M(v) \to 0$ be an extension where $\text{supp}(M(u)) \cap \text{supp}(M(v)) = \emptyset$ and where $s_u$ and $s_v$ do not have any common vertices. Then the given extension is split so $\text{Ext}^1_{\Lambda_T}(M(v), M(u)) = 0$.
\end{proposition}
\begin{proof}
Since $s_u$ and $s_v$ have no common vertices, there is no arrow $\alpha \in (Q_T)_1$ such that $u \stackrel{\alpha}{\leftarrow} v$ is a string in $\Lambda_T$. By exactness of the given sequence and by Lemma~\ref{disjointExt}, it is clear that $X = M(u) \oplus M(v)$. Thus the given sequence is split.
\end{proof}

\begin{proposition}\label{Prop:commonendext}
Let $M(u), M(v) \in \text{ind}(\Lambda_T\text{-mod})$ where $s_u$ and $s_v$ either share an endpoint and agree along a segment or they have a common vertex that is an endpoint of at most one of $s_u$ and $s_v$. Then $\text{Ext}_{\Lambda_T}^1(M(v), M(u)) = 0.$
\end{proposition}

\begin{proof}
Let $0 \to M(u) \stackrel{f}{\to} X  \stackrel{g}{\to} M(v) \to 0$ be an extension. By Lemma~\ref{Lemma:endptoverlapext} $i)$, $X$ has at least two summands $M(y)$ and $M(z)$ for some nonempty strings $y$ and $z$ in $\Lambda_T$. By Lemma~\ref{Lemma:endptoverlapext} $ii)$, without loss of generality, we have that $M(y) = M(u)$ and $M(z) = M(v)$ so the given sequence is split.\end{proof}

\begin{theorem}\label{shareendptext}
Let $M(u),M(v) \in \text{ind}(\Lambda_T\text{-mod})$ where $s_u$ and $s_v$ agree only at an endpoint. Then there is a nonsplit extension $\xi = 0 \to M(u) \to M(u\stackrel{\alpha}{\leftarrow} v)  \to M(v) \to 0$ if and only if there exists an arrow $\alpha \in (Q_T)_1$ such that $u \stackrel{\alpha}{\leftarrow} v$ is a string in $\Lambda_T.$ In this case, $\xi$ is the unique nonsplit extension of $M(u)$ by $M(v)$. 
\end{theorem}
\begin{proof}
Assume that there exists an arrow $\alpha \in (Q_T)_1$ such that $u \stackrel{\alpha}{\leftarrow} v$ is a string in $\Lambda_T$. Thus $M(u \stackrel{\alpha}{\leftarrow} v)$ is a string module and so $\xi$ is a nonsplit extension.

Assume that there does not exist an arrow $\alpha \in (Q_T)_1$ such that $u \stackrel{\alpha}{\leftarrow} v$ is a string in $\Lambda_T$. Let $0 \to M(u) \stackrel{f}{\to} X  \stackrel{g}{\to} M(v) \to 0$ be an extension. Lemma~\ref{disjointExt} implies that $X = M(u) \oplus M(v)$ so all such extensions are split.

The last assertion follows from the fact that $\dim_\Bbbk\text{Ext}^1_{\Lambda_T}(M(v),M(u)) = 1$ by Lemma~\ref{dimext}. 
\end{proof}

\begin{theorem}\label{crossingnonsplitext}
Suppose that $\text{supp}(M(u)) \cap \text{supp}(M(v)) \neq \emptyset$, and let $w$ denote the unique maximal string supported on $\text{supp}(M(u)) \cap \text{supp}(M(v)).$ Furthermore, assume that the segments $s_u$ and $s_v$ do not have any common endpoints. Write $u = u^{(1)}\leftrightarrow w \leftrightarrow u^{(2)}$ and $v = v^{(1)} \leftrightarrow w \leftrightarrow v^{(2)}$ for some strings $u^{(1)}, u^{(2)}, v^{(1)},$ and $v^{(2)}$ in $\Lambda_T$ some of which may be empty. Then $\text{Ext}^1_{\Lambda_T}(M(v), M(u)) \neq 0$ if and only if $u = u^{(1)}\leftarrow w \rightarrow u^{(2)}$ and $v = v^{(1)} \rightarrow w \leftarrow v^{(2)}$. Additionally, in this case, $$0 \to M(u) \to M(u^{(1)} \leftarrow w \leftarrow v^{(2)}) \oplus M(v^{(1)} \rightarrow w \rightarrow u^{(2)}) \to M(v) \to 0$$ is the unique nonsplit extension of $M(u)$ by $M(v)$.
\end{theorem}
\begin{proof}
Assume that $u = u^{(1)}\leftarrow w \rightarrow u^{(2)}$ and $v = v^{(1)} \rightarrow w \leftarrow v^{(2)}$ for some strings $u^{(1)}, u^{(2)}, v^{(1)},$ and $v^{(2)}$ in $\Lambda_T$. Note that the segments $s_u$ and $s_v$ have no common endpoints.   This means that $M(u^{(1)} \leftarrow w \leftarrow v^{(2)})$ is not isomorphic to $M(u)$ or $M(v)$ and the same is true for $M(v^{(1)} \rightarrow w \rightarrow u^{(2)})$. Thus $$0 \to M(u) \to M(u^{(1)} \leftarrow w \leftarrow v^{(2)}) \oplus M(v^{(1)} \rightarrow w \rightarrow u^{(2)}) \to M(v) \to 0$$ is a nonsplit extension. This implies that $\text{Ext}^1_{\Lambda_T}(M(v), M(u)) \neq 0$.

Conversely, assume that $\text{Ext}^1_{\Lambda_T}(M(v), M(u)) \neq 0$. Let $0 \to M(u) \stackrel{f}{\to} X  \stackrel{g}{\to} M(v) \to 0$ be a nonsplit extension and let $X = \oplus_{i = 1}^k X_i$ be a direct sum decomposition of $X$ into indecomposables. By Corollary~\ref{Cor:middletermext}, we have that $X = M(u^{(1)}\leftrightarrow w \leftrightarrow v^{(2)}) \oplus M(v^{(1)}\leftrightarrow w \leftrightarrow u^{(2)}).$ Since the given sequence is exact, we must have that $(u^{(1)}\leftrightarrow w \leftrightarrow v^{(2)}) = (u^{(1)}\leftarrow w \leftarrow v^{(2)})$ and $(v^{(1)}\leftrightarrow w \leftrightarrow u^{(2)}) = (v^{(1)}\rightarrow w \rightarrow u^{(2)}).$ Thus $u = u^{(1)}\leftarrow w \rightarrow u^{(2)}$ and $v = v^{(1)} \rightarrow w \leftarrow v^{(2)}$.

The last assertion follows from the fact that $\dim_\Bbbk\text{Ext}^1_{\Lambda_T}(M(v),M(u)) = 1$ by Lemma~\ref{dimext}.  \end{proof}

\subsection{Homomorphisms and extensions between string modules}\label{ingredients}
In this section, we present the technical facts required to prove Propositions~\ref{separatesegext} and \ref{Prop:commonendext} and Theorems~\ref{shareendptext} and \ref{crossingnonsplitext}. We prove Lemma~\ref{intcomponents}, which is used in the statement of Theorem~\ref{crossingnonsplitext}, Lemma~\ref{2termextlemma}, and Corollary~\ref{Cor:middletermext}. We omit the proofs of Lemma~\ref{injsurj1dim}, \ref{dimhom}, and \ref{dimext} as they are nearly identical to that of \cite[Lemma 9.2]{garver2015lattice}, \cite[Lemma 9.3]{garver2015lattice}, and \cite[Lemma 9.4]{garver2015lattice}, respectively.

\begin{lemma}\label{intcomponents}
Let $M(u), M(v) \in \text{ind}(\Lambda_T\text{-mod})$ with $\text{supp}(M(u)) \cap \text{supp}(M(v)) \neq \emptyset$. Then $w = x_1 \leftrightarrow x_2 \cdots x_{k-1}\leftrightarrow x_k$ where $\text{supp}(M(u)) \cap \text{supp}(M(v)) = \{x_i\}_{i \in [k]}$ is a string in $\Lambda$. Furthermore, $w$ is the unique maximal string along which $u$ and $v$ agree.
\end{lemma}
\begin{proof}
Any string in $\Lambda_T$ includes at most two vertices from any oriented cycle in $Q_T$. Thus a string $u = y_1 \leftrightarrow y_2 \cdots y_{s-1} \leftrightarrow y_s$ is the shortest path connecting $y_1$ and $y_s$ in the underlying graph of $Q_T$. This implies that for any $y_i$ and $y_j$ appearing in $u$, the string $y_i \leftrightarrow y_{i+1} \cdots y_{j-1} \leftrightarrow y_j$ is the shortest path connecting $y_i$ and $y_j$ in the underlying graph of $Q_T$. Therefore if $\text{supp}(M(u)) \cap \text{supp}(M(v)) \neq \emptyset$, then $w = x_1 \leftrightarrow x_2 \cdots x_{k-1}\leftrightarrow x_k$ where $\text{supp}(M(u)) \cap \text{supp}(M(v)) = \{x_i\}_{i \in [k]}$ is a string in $\Lambda_T$. Clearly, $w$ is the unique maximal string along which $u$ and $v$ agree.
\end{proof}

\begin{lemma}\label{injsurj1dim}
{Let $M(u),M(v) \in \text{ind}(\Lambda_T\text{-mod})$. If $M(u) \hookrightarrow M(v)$ or $M(u) \twoheadrightarrow M(v)$, then $$\dim_\Bbbk\text{Hom}_{\Lambda_T}(M(u), M(v)) = 1.$$}
\end{lemma}

\begin{lemma}\label{dimhom}
Let $M(u),M(v) \in \text{ind}(\Lambda_T\text{-mod}).$ Then $\dim_\Bbbk\text{Hom}_{\Lambda_T}(M(u),M(v)) \le 1.$ Additionally, assume $\text{Hom}_{\Lambda_T}(M(u),M(v)) \neq 0$, but $M(u)$ is not a submodule of $M(v)$ and $M(u)$ does not surject onto $M(v)$. Then there exists a string $w$ in $\Lambda_T$ distinct from both $u$ and $v$ such that $M(u) \twoheadrightarrow M(w) \hookrightarrow M(v).$
\end{lemma}

\begin{lemma}\label{endpointoverlapchar}
Assume $s_u$ and $s_v$ agree at an endpoint. Then $s_u$ and $s_v$ agree along a segment $s_w$ if and only if either $\text{Hom}_{\Lambda_T}(M(u),M(v)) \neq 0$ or $\text{Hom}_{\Lambda_T}(M(v),M(u)) \neq 0.$
\end{lemma}

\begin{proof}
Assume $s_u$ and $s_v$ agree along a segment $s_w$. By Lemma~\ref{intcomponents}, assume that $s_w$ is the unique largest segment along which $s_u$ and $s_v$ agree. We have that either $u = u^{(1)} \leftarrow w$ and $v = v^{(1)} \rightarrow w$ or $u = u^{(1)} \rightarrow w$ and $v = v^{(1)} \leftarrow w$. In the former case, $\text{Hom}_{\Lambda_T}(M(u),M(v)) \neq 0$. In the latter case, $\text{Hom}_{\Lambda_T}(M(v),M(u)) \neq 0.$ 

The converse statement is obvious.
\end{proof}

\begin{lemma}\label{dimext}
Let $M(u),M(v) \in \text{ind}(\Lambda_T\text{-mod}).$ Then $\dim_\Bbbk\text{Ext}^1_{\Lambda_T}(M(u),M(v)) \le 1.$ 
\end{lemma}

Next, we present four results, each of which is crucial to classifying extensions between indecomposable $\Lambda_T$-modules. Lemma~\ref{disjointExt} is used in the proof of Proposition~\ref{separatesegext} and Theorem~\ref{shareendptext}. Corollary~\ref{Cor:middletermext}, which is used in the proof Theorem~\ref{crossingnonsplitext}, follows from Lemma~\ref{2termextlemma}. Lemma~\ref{2termextlemma} establishes several restrictions on which indecomposable $\Lambda_T$-modules can appear as middle terms of a nonsplit extension between two indecomposables whose corresponding segments agree along a segment, but have no shared endpoints. Lastly, Lemma~\ref{Lemma:endptoverlapext} is used in the proof of Proposition~\ref{Prop:commonendext}.

\begin{lemma}\label{disjointExt}
Let $0 \to M(u) \stackrel{f}{\to} X  \stackrel{g}{\to} M(v) \to 0$ be an extension where $\text{supp}(M(u)) \cap \text{supp}(M(v)) = \emptyset$. Assume that there does not exist an arrow $\alpha \in (Q_T)_1$ such that $u \stackrel{\alpha}{\leftarrow} v$ is a string in $\Lambda_T$ and let $X = \oplus_{i = 1}^k X_i$ be a direct sum decomposition of $X$ in to indecomposables (i.e. $X_i \in \text{ind}(\Lambda_T\text{-mod})$ for each $i \in [k]$). Then none of the modules $X_i$ have any of the following properties

$\begin{array}{rl}
i) & \text{$\text{supp}(X_i) \cap \text{supp}(M(u)) \neq \emptyset$ and $\text{supp}(X_i) \cap \text{supp}(M(v)) \neq \emptyset$}\\
ii) & \text{$\text{supp}(X_i) \subsetneq \text{supp}(M(u))$}\\
iii) & \text{$\text{supp}(X_i) \subsetneq \text{supp}(M(v))$.}
\end{array}$
\end{lemma}
\begin{proof}

Suppose some $X_i$ satisfies $i)$. Then we can write $X_i = M(w)$, $u = u^\prime{\leftrightarrow} w^\prime,$ and $v = w^{\prime\prime} \leftrightarrow v^{\prime\prime}$ where $w = w^\prime {\leftrightarrow} w^{\prime\prime}$  is a string in $\Lambda_T$. By assumption, $w = w^\prime {\leftrightarrow} w^{\prime\prime} = w^\prime \stackrel{\beta}{\rightarrow} w^{\prime\prime}$.  Observe that the direction of $\beta$ implies that $\text{Hom}_{\Lambda_T}(M(u), M(w)) = 0$ and $\text{Hom}_{\Lambda_T}(M(w), M(v)) = 0$. Since $\text{supp}(M(u)) \cap \text{supp}(M(v)) = \emptyset$, $\{\text{supp}(X_i)\}_{i = 1}^k$ is a set partition of the set $\text{supp}(X)$. Thus we have that $M(w) \cap \text{im}(f) = 0$, but $M(w) \subset \text{ker}(g)$. This contradicts that the given sequence is exact. 

As none of the $X_i$ satisfy $i)$, we can separate these modules into those supported on $M(u)$ and those supported on $M(v)$. We denote the former modules by $\{M(u^{(j)})\}_{j = 1}^{s}$ and the latter by $\{M(v^{(\ell)})\}_{\ell = 1}^{t}$. 

Suppose $M(u^{(j)})$ satisfies $ii)$. Then there exist $M(u^{(j^\prime)})$ for some $j^\prime \neq j$ such that $u^{(j)} \stackrel{\beta}{\longleftrightarrow} u^{(j^\prime)}$ is a string in $\Lambda_T$ supported on $u$. Thus if $u^{(j)} \stackrel{\beta}{\longleftarrow} u^{(j^\prime)}$ (resp. $u^{(j)} \stackrel{\beta}{\longrightarrow} u^{(j^\prime)}$) is a string in $\Lambda_T$, we have that $\text{Hom}_{\Lambda_T}(M(u), M(u^{(j)})) = 0$ (resp. $\text{Hom}_{\Lambda_T}(M(u), M(u^{(j^\prime)})) = 0$). This implies that there exists a summand $M(u^{(j^{\prime\prime})})$ of $X$ such that $M(u^{(j^{\prime\prime})}) \cap \text{im}(f) = 0.$ However, $M(u^{(j^{\prime\prime})}) \subset \text{ker}(g)$ since $\text{supp}(M(u^{(j^{\prime\prime})})) \cap \text{supp}(M(v)) = \emptyset.$ This contradicts that the given sequence is exact. The proof that there are no summands $M(v^{(\ell)})$ of $X$ that satisfy $iii)$ is similar so we omit it.
\end{proof}

\begin{lemma}\label{2termextlemma}
Let $M(u), M(v) \in \text{ind}(\Lambda_T\text{-mod})$ where $s_u$ and $s_v$ have no common endpoints. Let $0 \to M(u) \stackrel{f}{\to} X  \stackrel{g}{\to} M(v) \to 0$ be a nonsplit extension where $\text{supp}(M(u)) \cap \text{supp}(M(v)) \neq \emptyset$, and let $w$ denote the unique maximal string supported on $\text{supp}(M(u)) \cap \text{supp}(M(v)).$ Let $X = \oplus_{i=1}^k X_i$ be a direct sum decomposition of $X$ into indecomposables and write $u = u^{(1)}\leftrightarrow w \leftrightarrow u^{(2)}$ and $v = v^{(1)} \leftrightarrow w \leftrightarrow v^{(2)}$ for some strings $u^{(1)}, u^{(2)}, v^{(1)},$ and $v^{(2)}$ in $\Lambda_T$ some of which may be empty. Then the following hold.

\begin{itemize}
\item[$i)$] $X$ is not indecomposable.
\item[$ii)$] There is no $X_i$ such that $\text{supp}(X_i) \cap \text{supp}(M(x)) \neq \emptyset$ for any $x \in \{w, u^{(1)}, u^{(2)}\},$ assuming that both $u^{(1)}$ and $u^{(2)}$ are nonempty strings.
\item[$iii)$] There is no $X_i$ such that $\text{supp}(X_i) \cap \text{supp}(M(x)) \neq \emptyset$ for any $x \in \{w, v^{(1)}, v^{(2)}\},$ assuming that both $v^{(1)}$ and $v^{(2)}$ are nonempty strings.
\item[$iv)$] There is no $X_i$ such that $\text{supp}(X_i) \subsetneq \text{supp}(M(x))$ where $x \in \{u^{(1)}, u^{(2)}, v^{(1)}, v^{(2)}\}.$ Thus each $X_i$ satisfies $\text{supp}(X_i) \cap \text{supp}(M(w)) \neq \emptyset$.
\item[$v)$] If $X_i$ and $x \in \{w, u^{(1)}, u^{(2)}, v^{(1)}, v^{(2)}\}$ satisfy $\text{supp}(X_i) \cap \text{supp}(M(x)) \neq \emptyset$, then $\text{supp}(M(x)) \subset \text{supp}(X_i)$.
\end{itemize}
\end{lemma}
\begin{proof}
We first show that each $X_i$ satisfies $X_i \not \cong M(u)$ and $X_i \not \cong M(v)$ since the given extension is nonsplit. Without loss of generality, suppose a summand $X_i$ of $X$ satisfies $X_i \cong M(u)$. Since $s_u$ and $s_v$ have no common endpoints, $\text{im}(f) = X_i$. By dimension considerations and the fact that $g$ is surjective, $M(v)$ is also a summand of $X$. Thus the given sequence is split, a contradiction. 

$i)$ We observe that by exactness, $\underline{\text{dim}}_\Bbbk(X) = \underline{\text{dim}}_\Bbbk(M(u)) + \underline{\text{dim}}_\Bbbk(M(v))$. Since $\text{supp}(M(u)) \cap \text{supp}(M(v)) \neq \emptyset$, Lemma~\ref{stringsegmentbij} $1.$ implies that $X$ is not a string module and therefore not indecomposable.

$ii)$ Suppose that such an $X_i$ exists. Then $\text{supp}(M(w)) \subset \text{supp}(X_i)$.  Now note that since $X_i \not \cong M(u)$ and $X_i \not \cong M(v)$, we can assume that $\text{dim}_\Bbbk(M(u^{(1)})) \ge 1$ or $\text{dim}_\Bbbk(M(u^{(2)})) \ge 1$. Without loss of generality, we assume the former. This implies that $\text{supp}(X_i) \cap\text{supp}(M(u^{(1)})) \subsetneq \text{supp}(M(u^{(1)}))$ and $\text{supp}(X_i) \cap \text{supp}(M(u^{(1)})) \neq \emptyset$. The fact that $\text{dim}_\Bbbk(M(u^{(1)})) \ge 1$ also implies that we can write $u^{(1)} = x^{(1)}\leftrightarrow x^{(2)}$ for some nonempty strings $x^{(1)}$ and $x^{(2)}$ in $\Lambda_T$ where $\text{supp}(M(x^{(2)})) = \text{supp}(X_i) \cap \text{supp}(M(u^{(1)}))$ and $u = x^{(1)}\leftrightarrow x^{(2)} \leftrightarrow w \leftrightarrow u^{(2)}$. 

Suppose $u^{(1)} = x^{(1)}\leftarrow x^{(2)}$. Now write $x^{(1)} = x^{(1)}_1 \leftrightarrow x^{(1)}_2 \leftrightarrow \cdots \leftrightarrow x^{(1)}_\ell$ so that $$u^{(1)} = x^{(1)}\leftarrow x^{(2)} = (x^{(1)}_1 \leftrightarrow x^{(1)}_2 \leftrightarrow \cdots \leftrightarrow x^{(1)}_\ell) \leftarrow x^{(2)}.$$ In this case, $\text{Hom}_{\Lambda_T}(M(u), X_j) = 0$ if $X_j$ is any summand of $X$ where $\text{supp}(X_j) \subset \text{supp}(M(x^{(1)}))$ and $x^{(1)}_\ell \in \text{supp}(X_j).$ Thus any such $X_j$ satisfies $X_j \cap \text{im}(f) = 0$. One also observes that $\text{supp}(M(x^{(1)})) \cap \text{supp}(M(v)) = \emptyset$ so $\text{Hom}_{\Lambda_T}(X_j, M(v)) = 0$. Therefore, any such $X_j \subset \text{ker}(g).$ This means that if such a summand $X_j$ exists, then the given sequence is not exact.

We show that there must be a summand $X_j$ of $X$ satisfying $\text{supp}(X_j) \subset \text{supp}(M(x^{(1)}))$ and whose string contains $x^{(1)}_\ell$. First note that by the exactness of the given sequence, there must exist a summand $X_j$ of $X$ whose support contains $x^{(1)}_\ell$ and thus intersects $\text{supp}(M(x^{(1)}))$. To complete the proof, it is enough to show that, without loss of generality, there is no string $y$ in $\Lambda_T$ such that $\text{supp}(M(y)) \cap \text{supp}(M(x^{(1)})) \neq \emptyset$ and $\text{supp}(M(y)) \cap \text{supp}(M(v^{(1)})) \neq \emptyset$. To show this, it is enough to observe that the segments $s_{x^{(1)}}$ and $s_{v^{(1)}}$ have no common vertices, since $x^{(2)}$ is a nonempty string. We obtain a contradiction.

We now have that $u^{(1)} = x^{(1)}\rightarrow x^{(2)}$. This implies that $\text{Hom}_{\Lambda_T}(M(u), X_i) = 0$. Let us express $X_i$ as $X_i = ((V_i)_{i \in Q_0}, (\varphi_\alpha)_{\alpha\in Q_1})$. By exactness and dimension considerations, the module $X_i$ is the only summand of $X$ satisfying $\text{supp}(X_i) \cap \text{supp}(M(x^{(2)})) \neq \emptyset.$ Thus if $\lambda \in V_i$ is nonzero and $i \in \text{supp}(X_i) \cap \text{supp}(M(x^{(2)}))$, then $\lambda \not \in \text{im}(f)$. This contradicts that $f$ is injective. 

$iii)$ The proof of this assertion is similar to the proof of assertion $ii)$ so we omit it. 

$iv)$ It suffices to show that there does not exist a summand $X_i$ of $X$ such that $\text{supp}(X_i) \subsetneq \text{supp}(M(v^{(1)}))$. Suppose there exists such a summand $X_i$. Then there exist summands $M(x)$ and $M(y)$ of $X$ where $x \leftrightarrow y$ is a string in $\Lambda_T$ where $\text{supp}(M(x)) \subsetneq \text{supp}(M(v^{(1)}))$ and $\text{supp}(M(y)) \cap \text{supp}(M(v^{(1)})) \neq \emptyset.$ If the string $(x\leftrightarrow y) = (x\leftarrow y)$, then $\text{Hom}_{\Lambda_T}(M(y), M(v)) = 0$. Let us express $M(y)$ as $M(y) = ((V_i)_{i \in Q_0}, (\varphi_\alpha)_{\alpha\in Q_1})$. Then any nonzero $\lambda \in V_i$ where $i \in \text{supp}(M(y)) \cap \text{supp}(M(v^{(1)}))$ satisfies $\lambda \in \text{ker}(g).$ Since $\lambda$ does not belong to any summand besides $M(y)$, we have that $g$ is not surjective, a contradiction. If the string $(x\leftrightarrow y) = (x\rightarrow y)$, then $\text{Hom}_{\Lambda_T}(M(x), M(v)) = 0.$ Similarly, this implies that $M(x) \subset \text{ker}(g)$, which contradicts that $g$ is surjective.

$v)$ We first prove the assertion for any $x \in \{u^{(1)}, u^{(2)}, v^{(1)}, v^{(2)}\}$. As in the proof of $iv)$, it suffices to prove this for $x = v^{(1)}.$ Suppose that there exists $X_i$ such that $\text{supp}(X_i) \cap \text{supp}(M(v^{(1)})) \neq \emptyset$ and $\text{supp}(M(v^{(1)})) \not \subset \text{supp}(X_i).$ By $iv)$, we have that $\text{supp}(X_i) \cap \text{supp}(M(w)) \neq \emptyset$. Now by exactness of the given sequence, there exists another summand $X_j$ of $X$ such that $\text{supp}(X_j) \subset \text{supp}(M(v^{(1)}))\backslash \text{supp}(X_i) \subset \text{supp}(M(v^{(1)})).$ This contradicts $iv)$.

By assertion $iv)$, each summand $X_i$ satisfies $\text{supp}(X_i)\cap \text{supp}(M(w)) \neq \emptyset$. Thus it is enough to show that there are no summands $X_i$ such that $\text{supp}(X_i) \subsetneq \text{supp}(M(w)).$ Suppose there exists such a summand $X_i = M(y^{(2)})$. We can assume, without loss of generality, that there is another summand $X_j = M(y^{(1)})$ of $X$ such that 

\begin{itemize}
\item $y^{(1)}\leftrightarrow y^{(2)}$ is a string in $\Lambda_T$, 
\item $\text{supp}(M(y^{(1)})) \cap \text{supp}(M(v^{(1)})) \neq \emptyset$,
\item $\text{supp}(M(y^{(1)})) \cap \text{supp}(M(w)) \neq \emptyset$.
\end{itemize}

Suppose that $(y^{(1)}\leftrightarrow y^{(2)}) = (y^{(1)}\rightarrow y^{(2)}).$ Then $\text{Hom}_{\Lambda_T}(M(y^{(1)}), M(v)) = 0.$ Let us express $M(y^{(1)})$ as $M(y^{(1)}) = ((V_i)_{i \in Q_0}, (\varphi_\alpha)_{\alpha\in Q_1})$. The for any nonzero $\lambda \in V_i$ where $i \in \text{supp}(M(y^{(1)})) \cap \text{supp}(M(v^{(1)}))$ satisfies $\lambda \in \text{ker}(g).$ Since $M(y^{(1)})$, is the only summand containing $\lambda$, this contradicts that $g$ is surjective.

Now suppose $(y^{(1)}\leftrightarrow y^{(2)}) = (y^{(1)}\leftarrow y^{(2)})$ and write $y^{(2)} = y_1^{(2)} \leftrightarrow \cdots \leftrightarrow y_\ell^{(2)}$. Then $\text{Hom}_{\Lambda_T}(M(y^{(2)}),M(v)) = 0.$ This means that any other summand $M(y^{(3)})$ of $X$ where $(y^{(1)}\leftarrow y^{(3)})$ is a string in $\Lambda_T$ and $y_1^{(2)} \in \text{supp}(M(y^{(3)}))$ has the property that $\text{Hom}_{\Lambda_T}(M(y^{(3)}),M(v)) = 0.$ Since $M(y^{(1)})$ is the only summand of $X$ whose support intersects $\text{supp}(M(y^{(1)}))\cap \text{supp}(M(v^{(1)}))$ and since $\text{supp}(M(y^{(1)})) \subset \text{supp}(M(v))$, we have that there is an inclusion $M(y^{(1)}) \hookrightarrow M(v).$ Since the given sequence is exact, there must exist a summand $M(z) = ((V_i)_{i \in Q_0}, (\varphi_\alpha)_{\alpha\in Q_1})$ of $X$ where $z$ satisfies

\begin{itemize}
\item $\text{supp}(M(z)) \cap \text{supp}(M(y^{(1)})) \neq \emptyset$ where any nonzero $\lambda \in V_i$ for $i \in \text{supp}(M(z)) \cap \text{supp}(M(y^{(1)}))$ satisfies $\lambda \not \in \text{ker}(g) = \text{im}(f)$, and
\item $\text{supp}(M(z)) \cap \text{supp}(M(y^{(2)})) \neq \emptyset$ where any nonzero $\lambda \in V_i$ for $i \in \text{supp}(M(z)) \cap \text{supp}(M(y^{(2)}))$ satisfies $\lambda \in \text{im}(f)$. 
\end{itemize}

\noindent However, since $(y^{(1)}\leftrightarrow y^{(2)}) = (y^{(1)}\leftarrow y^{(2)})$ there are no homomorphisms from $M(u)$ to $M(z)$ satisfying these properties. Thus there are no summands $X_i$ of $X$ such that $\text{supp}(X_i) \subsetneq \text{supp}(M(w)).$
\end{proof}

\begin{corollary}\label{Cor:middletermext}
Let $M(u), M(v) \in \text{ind}(\Lambda_T\text{-mod})$ where $s_u$ and $s_v$ have no common endpoints. Let $0 \to M(u) \stackrel{f}{\to} X  \stackrel{g}{\to} M(v) \to 0$ be a nonsplit extension where $\text{supp}(M(u)) \cap \text{supp}(M(v)) \neq \emptyset$, and let $w$ denote the unique maximal string supported on $\text{supp}(M(u)) \cap \text{supp}(M(v)).$ Let $X = \oplus_{i=1}^k X_i$ be a direct sum decomposition of $X$ into indecomposables and write $u = u^{(1)}\leftrightarrow w \leftrightarrow u^{(2)}$ and $v = v^{(1)} \leftrightarrow w \leftrightarrow v^{(2)}$ for some strings $u^{(1)}, u^{(2)}, v^{(1)},$ and $v^{(2)}$ in $\Lambda_T$ some of which may be empty. Then $X = M(u^{(1)}\leftrightarrow w \leftrightarrow v^{(2)}) \oplus M(v^{(1)}\leftrightarrow w \leftrightarrow u^{(2)})$. \end{corollary}

\begin{proof}
By Lemma~\ref{2termextlemma} $i)$, $X$ has at least two indecomposable summands. By Lemma~\ref{2termextlemma} $iv)$ and $v)$, $X$ has exactly two summands, $M(y)$ and $M(z)$, where $\text{supp}(M(w)) \subset \text{supp}(M(y))$ and $\text{supp}(M(w)) \subset \text{supp}(M(z))$. By exactness of the given sequence and by Lemma~\ref{2termextlemma} $v)$, for any $x \in \{u^{(1)}, u^{(2)}, v^{(1)}, v^{(2)}\}$ we have that $\text{supp}(M(x))$ is contained in $\text{supp}(M(y))$ or $\text{supp}(M(z))$. By combining Lemma~\ref{2termextlemma} $ii)$ and $iii)$, we have that $M(y) = M(u^{(1)}\leftrightarrow w \leftrightarrow v^{(2)})$ and $M(z) = M(v^{(1)}\leftrightarrow w \leftrightarrow u^{(2)}).$
\end{proof}


\begin{lemma}\label{Lemma:endptoverlapext}
Let $M(u), M(v) \in \text{ind}(\Lambda_T\text{-mod})$ where $s_u$ and $s_v$ either share an endpoint and agree along a segment or they have a common vertex that is an endpoint of at most one of $s_u$ and $s_v$. If $0 \to M(u) \stackrel{f}{\to} X  \stackrel{g}{\to} M(v) \to 0$ is an extension and $X = \oplus_{i = 1}^k X_i$ is a direct sum decomposition into indecomposables, then the following hold.

\begin{itemize}
\item[$i)$] $X$ is not indecomposable.
\item[$ii)$] There is no $X_i$ such that $\text{supp}(X_i) \subsetneq \text{supp}(M(x))$ where $x \in \{u, v\}.$
\end{itemize}
\end{lemma}

\begin{proof}
Note that only Lemma~\ref{2termextlemma} $ii)$ and $iii)$ relied on the assumption that the given extension was nonsplit. Thus one proves these assertions by adapting the proofs of Lemmas~\ref{2termextlemma} $i)$, $iv)$, and $v)$, since these did not depend on Lemma~\ref{2termextlemma} $ii)$ and $iii)$.
\end{proof}

\subsection{Oriented flip graphs and torsion-free classes}\label{subsec:torsf}

In this section, we recall the definition of torsion-free classes and their lattice structure. After that, we show that oriented flip graphs are isomorphic as posets to the lattice of torsion-free classes of $\Lambda_T$ ordered by inclusion and torsion classes of $\Lambda_T$ ordered by reverse inclusion.

Let $\Lambda$ be a finite dimensional $\Bbbk$-algebra. A full, additive subcategory $\mathcal{C} \subset \Lambda$-mod is \textbf{extension closed} if for any objects $X,Y \in \mathcal{C}$ satisfying $0 \to X \to Z \to Y \to 0$ one has $Z \in \mathcal{C}$. We say $\mathcal{C}$ is \textbf{quotient closed} (resp. \textbf{submodule closed}) if for any $X \in \mathcal{C}$ satisfying $X \stackrel{\alpha}{\longrightarrow} Z$ where $\alpha$ is a surjection (resp. $Z \stackrel{\beta}{\longrightarrow} X$ where $\beta$ is an injection), then $Z \in \mathcal{C}$. A full, additive subcategory $\mathcal{T} \subset \Lambda$-mod is called a \textbf{torsion class} if $\mathcal{T}$ is quotient closed and extension closed. Dually, a full, additive subcategory $\mathcal{F} \subset \Lambda$-mod is called a \textbf{torsion-free class} if $\mathcal{F}$ is extension closed and submodule closed. 

Let $\text{tors}(\Lambda)$ (resp. $\text{torsf}(\Lambda)$) denote the lattice of torsion classes  (resp. of torsion-free classes) of $\Lambda$ ordered by inclusion. We have the following proposition, which shows that a torsion class of $\Lambda$ uniquely determines a torsion-free class of $\Lambda$ and vice versa. Given $\mathcal{T}$ a torsion class and $\mathcal{F}$ its corresponding torsion-free class, we say that the data $(\mathcal{T}, \mathcal{F})$ is a \textbf{torsion pair}.

\begin{proposition}\label{torsbij}
\cite[Prop. 1.1 a)]{iyama.reiten.thomas.todorov:latticestrtors} The maps $$\begin{array}{rcl} \text{tors}(\Lambda) & \stackrel{(-)^\perp}{\longrightarrow} & \text{torsf}(\Lambda)\\ \mathcal{T} & \longmapsto &  \mathcal{T}^\perp := \{X \in \Lambda\text{-mod}: \ \text{Hom}_{\Lambda}(\mathcal{T},X) = 0\}\end{array}$$
and
$$\begin{array}{rcl} \text{torsf}(\Lambda) & \stackrel{^{\perp}(-)}{\longrightarrow} & \text{tors}(\Lambda)\\ \mathcal{F} & \longmapsto &  {}^{\perp}\mathcal{F} := \{X \in \Lambda\text{-mod}: \ \text{Hom}_{\Lambda}(X,\mathcal{F}) = 0\}\end{array}$$ are inverse bijections. 
\end{proposition} 

The lattices $\text{tors}(\Lambda)$ and $\text{torsf}(\Lambda)$ have the following description of the meet and join operations. 
\begin{proposition}\cite[Prop. 1.3]{iyama.reiten.thomas.todorov:latticestrtors}\label{meetandjointors} Let $\Lambda$ be a finite dimensional algebra. Then $\text{tors}(\Lambda)$ and $\text{torsf}(\Lambda)$ are complete lattices. The join and meet operations are described as follows

$\begin{array}{rl}
a) & \text{Let } \{\mathcal{T}_i\}_{i \in I} \subset \text{tors}(\Lambda) \text{ be a collection of torsion classes. Then we have } \bigwedge_{i\in I}\mathcal{T}_i = \bigcap_{i \in I} \mathcal{T}_i\\
& \text{and }  \bigvee_{i\in I}\mathcal{T}_i = {}^{\perp}\left(\bigcap_{i \in I} \mathcal{T}_i^\perp\right). \\
b) & \text{Let } \{\mathcal{F}_i\}_{i \in I} \subset \text{torsf}(\Lambda) \text{ be a collection of torsion-free classes. Then we have } \bigwedge_{i\in I}\mathcal{F}_i = \bigcap_{i \in I} \mathcal{F}_i\\
& \text{and }  \bigvee_{i\in I}\mathcal{F}_i = \left(\bigcap_{i \in I} {}^{\perp}\mathcal{F}_i\right)^{\perp}.
\end{array}$

\end{proposition}

\begin{lemma}\cite[Prop. 1.4 a), c)]{iyama.reiten.thomas.todorov:latticestrtors}\label{standarddual}
The maps $$\begin{array}{rcl} \text{tors}(\Lambda) & \stackrel{D(-)}{\longrightarrow} & \text{torsf}(\Lambda^{\text{op}}) \cong \text{torsf}(\Lambda)^{\text{op}} \\ \mathcal{T} & \longmapsto &  D{\mathcal{T}} \end{array}$$
and
$$\begin{array}{rcl} \text{torsf}(\Lambda) & \stackrel{D(-)}{\longrightarrow} & \text{tors}(\Lambda^{\text{op}}) \cong \text{tors}(\Lambda)^{\text{op}}\\ \mathcal{F} & \longmapsto &  D\mathcal{F} \end{array}$$ are isomorphisms of lattices where $D(-):= \text{Hom}_\Lambda(-,\Bbbk)$ is the \textbf{standard duality}. Furthermore, the functor $D((-)^\perp): \text{tors}(\Lambda) \to \text{tors}(\Lambda^{\text{op}})$ is an anti-isomorphism of posets.
\end{lemma}

\begin{theorem}\label{Thm:TorsfOrFlipIso}
For any tree $T$, we have that $\overrightarrow{FG}(T) \cong \text{torsf}(\Lambda_T)$ and $\overrightarrow{FG}(T) \cong \text{tors}(\Lambda_T)^{\text{op}}$ where $\text{tors}(\Lambda_T)^{\text{op}}$ denotes the lattice of torsion classes ordered by reverse inclusion.
\end{theorem}
\begin{proof}
By Lemma~\ref{standarddual}, it is enough to prove that $\overrightarrow{FG}(T) \cong \text{torsf}(\Lambda_T)$.  Furthermore, by Theorem~\ref{thm_eta_phi_main} (2), we have that $\overrightarrow{FG}(T) \cong \pi_{\downarrow}(\text{Bic}(T))$ so it is enough to show that the latter is isomorphic to $\text{torsf}(\Lambda_T)$.

We claim that the map $$\begin{array}{rcl} \pi_{\downarrow}(\text{Bic}(T)) & \stackrel{\zeta}{\longrightarrow} & \text{torsf}(\Lambda_T)\\ \pi_{\downarrow}(X) & {\longmapsto} &  \mathcal{F}:=\text{add}(\oplus_{s_u} M(u): \ s_u \in \pi_{\downarrow}(X))  \end{array}$$ is an isomorphism of posets where $\text{add}(\oplus_{i= 1}^k X_i)$ for any finite set of $\Lambda_T$-modules $X_i$ denotes the smallest full, additive subcategory of $\Lambda_T$-mod closed under taking summands of $\oplus_{i= 1}^k X_i$. Furthermore, we claim that the inverse of this map is given by $$\begin{array}{rcl} \text{torsf}(\Lambda_T) & \stackrel{\delta}{\longrightarrow} & \pi_{\downarrow}(\text{Bic}(T))\\ {\mathcal{F}} = \text{add}(\oplus_{i \in [k]} M(w^{(i)})) & \longmapsto & \pi_\downarrow(\{s_{w^{(1)}}, \ldots, s_{w^{(k)}}\}). \end{array}$$ We can see that these maps are order-preserving, since $\pi_\downarrow$ is order-preserving by Lemma~\ref{lem_map_properties} (7). Assuming that $\zeta(\pi_{\downarrow}(X))$ is a torsion-free class and that $\delta(\mathcal{F}) \in \pi_\downarrow(\text{Bic}(T)),$ we have that $\delta = \zeta^{-1}$ as $\pi_\downarrow$ is an idempotent map (see Lemma~\ref{lem_map_properties} (5)).

We first show that $\delta(\mathcal{F}) \in \pi_\downarrow(\text{Bic}(T))$ where ${\mathcal{F}} = \text{add}(\oplus_{i \in [k]} M(w^{(i)}))$. Let $s_u, s_v \in \{s_{w^{(1)}}, \ldots, s_{w^{(k)}}\}$ and assume $s_u \circ s_v \in \text{Seg}(T).$ Then, up to reversing the roles of $u$ and $v$, $u \leftarrow v$ is a string in $\Lambda_T$ so there is an extension $0 \to M(u) \to M(u \leftarrow v) \to M(v) \to 0$. Since $\mathcal{F}$ is extension closed, $M(u\leftarrow v) \in \mathcal{F}$ so $s_u \circ s_v = s_{(u\leftarrow v)} \in \{s_{w^{(1)}}, \ldots, s_{w^{(k)}}\}$. Thus $\{s_{w^{(1)}}, \ldots, s_{w^{(k)}}\}$ is closed. Since $\mathcal{F}$ is submodule closed, there are no extensions of the form $0 \to M(u) \to M(u \leftarrow v) \to M(v) \to 0$ where $s_u, s_v \not \in \{s_{w^{(1)}}, \ldots, s_{w^{(k)}}\}$, but $s_{(u \leftarrow v)} \in \{s_{w^{(1)}}, \ldots, s_{w^{(k)}}\}$. Thus $\{s_{w^{(1)}}, \ldots, s_{w^{(k)}}\}$ is co-closed.

Next, we show that $\mathcal{F}:=\text{add}(\oplus_{s_u} M(u): \ s_u \in \pi_{\downarrow}(X))$ is a torsion-free class. We begin by showing that it is submodule closed. Assume that there is an inclusion $M(v) \hookrightarrow M(u)$ where $M(u) \in \mathcal{F}$. Write $s_u = (x_0, \ldots, x_\ell)$ and orient this segment from $x_0$ to $x_\ell$. Let $s_v = (x_{i}, \ldots, x_j)$ where we can assume that $0 < i$ and $j < \ell$. The inclusion $M(v) \hookrightarrow M(u)$ implies that $u = u^{(1)} \rightarrow v \leftarrow u^{(2)}$ for some nonempty strings $u^{(1)}$ and $u^{(2)}$ in $\Lambda_T$. Now we have that $s_v$  turns right (resp. left) at $x_i$ (resp. at $x_j$). Thus  $s_v \in C_{s_u} \subset X$. This implies that $C_{s_v} \subset C_{s_u} \subset X$ so $s_v \in \pi_\downarrow(X)$. We obtain that $M(v) \in \mathcal{F}$.

Now suppose $f: M(v) \hookrightarrow X = \oplus_{i \in [k]}M(w^{(i)})^{a_i}$ for some $a_i \ge 0$ and $M(v)$ does not include into any summand of $X$. Furthermore, suppose any indecomposable $M(u)$ with $\text{dim}_\Bbbk(M(u)) < \text{dim}_\Bbbk(M(v))$ that includes into an object of $\mathcal{F}$ belongs to $\mathcal{F}$. Let $M(w^{(i)})$ be a summand of $X$ where the component map $g: M(v) \to M(w^{(i)})$ of $f$ is nonzero. By Lemma~\ref{dimhom}, we can assume that there exists a nonempty string $w$ in $\Lambda_T$ not equal to $u$ or $w^{(i)}$ such that $M(v) \twoheadrightarrow M(w) \hookrightarrow M(w^{(i)})$. By the previous paragraph, $M(w) \in \mathcal{F}$. Now express $v$ as $v = v^{(1)} \leftarrow w \rightarrow v^{(2)}$ where, without loss of generality, both $v^{(1)}$ and $v^{(2)}$ are nonempty. This implies that $M(v^{(i)}) \hookrightarrow X$ so $M(v^{(i)}) \in \mathcal{F}$ for $i = 1,2$ since $\text{dim}_\Bbbk(M(v^{(i)})) < \text{dim}_\Bbbk(M(v))$. Observe that we have an extension $0 \to M(v^{(2)}) \to M(w\rightarrow v^{(2)}) \to M(w) \to 0$, which shows that $M(w\rightarrow v^{(2)}) \in \mathcal{F}$ since $s_{(w\rightarrow v^{(2)})} = s_w\circ s_{v^{(2)}} \in \pi_\downarrow(X).$ This implies that we have an extension $0 \to M(v^{(1)}) \to M(v) \to M(w \rightarrow v^{(2)}) \to 0$, which shows that $M(v) \in \mathcal{F}$ since $s_{v} = s_{v^{(1)}} \circ s_w\circ s_{v^{(2)}} \in \pi_\downarrow(X).$ We conclude that $\mathcal{F}$ is submodule closed.

Lastly, we show that $\mathcal{F}$ is extension closed. Since $\pi_\downarrow(X)$ is closed, it is easy to see that $\mathcal{F}$ is extension closed with respect to extensions whose nonzero terms are indecomposable. By our description of nonsplit extensions in $\Lambda_T$-mod (see Section~\ref{Sec:stringalgtree}), it suffices to show that if $M(u), M(v) \in \mathcal{F}$ where $u = u^{(1)} \leftarrow w \rightarrow u^{(2)}$ and $v = v^{(1)} \rightarrow w \leftarrow v^{(2)}$ and $$0 \to M(u) \to M(u^{(1)} \leftarrow w \leftarrow v^{(2)}) \oplus M(v^{(1)} \rightarrow w \rightarrow u^{(2)}) \to M(v) \to 0$$ is the nonsplit extension defined by these modules, then $M(u^{(1)} \leftarrow w \leftarrow v^{(2)}), M(v^{(1)} \rightarrow w \rightarrow u^{(2)}) \in \mathcal{F}$. We show $M(u^{(1)} \leftarrow w \leftarrow v^{(2)}) \in \mathcal{F}$ and the proof that $M(v^{(1)} \rightarrow w \rightarrow u^{(2)}) \in \mathcal{F}$ is very similar. Notice that $M(u^{(1)}) \hookrightarrow M(u)$ and $M(w \leftarrow v^{(2)}) \hookrightarrow M(v)$ so $M(u^{(1)}), M(w \leftarrow v^{(2)}) \in \mathcal{F}$. Thus we obtain a nonsplit extension $0 \to M(u^{(1)}) \to M(u^{(1)} \leftarrow w \leftarrow v^{(2)}) \to M(w \leftarrow v^{(2)}) \to 0$, which shows that $M(u^{(1)} \leftarrow w \leftarrow v^{(2)}) \in \mathcal{F}.$\end{proof}

\subsection{Noncrossing tree partitions and wide subcategories}\label{Subsec:NCP_and_wide}


In this section, we show that noncrossing tree partitions of a tree $T$ provide a combinatorial model for the wide subcategories of $\Lambda_T$-mod. 

If $\Lambda$ is a finite dimensional $\Bbbk$-algebra, we say that a full, additive subcategory $\mathcal{W} \subset \Lambda$-mod is a  \textbf{wide subcategory} if it is abelian and extension closed. We let $\text{wide}(\Lambda)$ denote the poset of wide subcategories of $\Lambda$-mod partially ordered by inclusion. It is easy to see that the intersection of two wide subcategories is also a wide subcategory, and the zero subcategory (resp. $\Lambda$-mod) is the bottom (resp. top) element $\text{wide}(\Lambda)$. Thus if $\Lambda$ is representation-finite, the poset $\text{wide}(\Lambda)$ is a lattice.  

\begin{theorem}
For any tree $T$, we have the following isomorphisms of posets: $$\begin{array}{rcccl}\text{NCP}(T) & \longrightarrow & \Psi(\overrightarrow{FG}(T)) & \longrightarrow & \text{wide}(\Lambda_T) \\ \textbf{B} & \longmapsto & \overline{\text{Seg}(\textbf{B})} & \longmapsto & \text{add}\left(\oplus_{s_u} M(u): s_u \in \overline{\text{Seg}(\textbf{B})}\right). \end{array}$$
\end{theorem}
\begin{proof}
That the map on the left is an isomorphism follows from Theorem~\ref{Thm:rhocircPsi}. It is clear that the map on the right is order-preserving so it is enough to show that this same map defines a wide subcategories and has an order-preserving inverse.  

We first show that $\mathcal{W} := \text{add}\left(\oplus_{s_u} M(u): s_u \in \overline{\text{Seg}(\textbf{B})}\right) \in \text{wide}(\Lambda_T)$. By Lemma~\ref{HomLambdaFree}, we know that $\mathcal{W}$ is closed under taking kernels and cokernels of maps between modules $M(u), M(v) \in \mathcal{W}$ where $s_u, s_v \in \text{Seg}(\textbf{B})$. Now suppose that $f \in \text{Hom}_{\Lambda_T}(M(u), M(v))$ is nonzero where $s_u \in \overline{\text{Seg}(B)}$, $s_v \in \overline{\text{Seg}(B^\prime)}$, and where $B$ and $B^\prime$ are blocks of $\textbf{B}$. We can further assume that $f$ is neither injective nor surjective. We write $s_u = s_{u^{(1)}}\circ \cdots \circ s_{u^{(k)}}$ and $s_v = s_{v^{(1)}}\circ \cdots \circ s_{v^{(\ell)}}$ where $s_{u^{(1)}}, \ldots, s_{u^{(k)}} \in \text{Seg}(B)$ and $s_{v^{(1)}}, \ldots, s_{v^{(\ell)}} \in \text{Seg}(B^\prime)$. 

Assume $B \neq B^\prime.$ Since $f$ is nonzero, then $s_u \not \in \text{Seg}(B)$ or $s_v \not \in \text{Seg}(B^\prime)$. Without loss of generality, we assume that $s_u \not \in \text{Seg}(B)$. Thus we have an inclusion $M(u^{(t)}) \hookrightarrow M(u)$ for some $t = 1, \ldots, k$. This implies that $\text{Hom}_{\Lambda_T}(M(u^{(t)}),M(v)) \neq 0$, which contradicts Lemma~\ref{HomLambdaFree}.

Now assume $B = B^\prime$, and suppose that any $g \in \text{Hom}_{\Lambda_T}(X,Y)$ has $\text{ker}(g), \text{coker}(g) \in \mathcal{W}$ for any $X, Y \in \mathcal{W}$ with $\text{dim}_\Bbbk(X) + \text{dim}_\Bbbk(Y) < \text{dim}_\Bbbk(M(u)) + \text{dim}_\Bbbk(M(v))$. Define $s_w := s_{w^{(1)}}\circ \cdots \circ s_{w^{(t)}}$ where $\{s_{w^{(1)}}, \ldots, s_{w^{(t)}}\} = \{s_{u^{(1)}}, \ldots, s_{u^{(k)}}\}\cap \{s_{v^{(1)}}, \ldots, s_{v^{(\ell)}}\}$. Now we have that $f$ factors as $M(u) \stackrel{\alpha}{\twoheadrightarrow} M(w) \stackrel{\beta}{\hookrightarrow} M(v)$. Since $f$ is neither injective nor surjective, we know that $\text{dim}_\Bbbk(M(u)) + \text{dim}_\Bbbk(M(w))$ and $\text{dim}_\Bbbk(M(w)) + \text{dim}_\Bbbk(M(v))$ are both less than $\text{dim}_\Bbbk(M(u)) + \text{dim}_\Bbbk(M(v))$. We thus obtain that $\text{ker}(f) = \text{ker}(\alpha) \in \mathcal{W}$ and $\text{coker}(f) = \text{coker}(\beta) \in \mathcal{W}$. We conclude that $\mathcal{W}$ is abelian. 

To see that $\mathcal{W}$ is extension closed, it is enough to show that if $M(u), M(v) \in \mathcal{W}$ where $u = u^{(-)} \leftarrow w \rightarrow u^{(+)}$, $v = v^{(-)} \rightarrow w \leftarrow v^{(+)}$, and $$0 \to M(u) \to M(u^{(-)} \leftarrow w \leftarrow v^{(+)}) \oplus M(v^{(-)} \rightarrow w \rightarrow u^{(+)}) \to M(v) \to 0$$ is the nonsplit extension defined by these modules, then $M(u^{(-)} \leftarrow w \leftarrow v^{(+)}), M(v^{(-)} \rightarrow w \rightarrow u^{(+)}) \in \mathcal{W}$. Notice that the existence of such an extension implies that the segments $s_u$ and $s_v$ are crossing. Now using Lemma~\ref{lem_crossing_div} and \ref{prop_red_green_switch}, we deduce the existence of a crossing between two segments in $\text{Seg}(\textbf{B})$, a contradiction. We thus have that $s_u, s_v \in \overline{\text{Seg}(B)}$ for some block $B$ of $\textbf{B}$.

Write $s_u = s_{u^{(1)}}\circ \cdots \circ s_{u^{(k)}}$ and $s_v = s_{v^{(1)}}\circ \cdots \circ s_{v^{(\ell)}}$ where $s_{u^{(1)}}, \ldots, s_{u^{(k)}}, s_{v^{(1)}}, \ldots, s_{v^{(\ell)}} \in \text{Seg}(B)$. Note that up to reversing the expression, these are the unique expressions for $s_u$ and $s_v$ in terms of elements of $\text{Seg}(\textbf{B})$. This implies that that $s_w = s_{w^{(1)}}\circ \cdots \circ s_{w^{(t)}}$ where $\{s_{w^{(1)}}, \ldots, s_{w^{(t)}}\} = \{s_{u^{(1)}}, \ldots, s_{u^{(k)}}\}\cap \{s_{v^{(1)}}, \ldots, s_{v^{(\ell)}}\}$ so that $M(w) \in \mathcal{W}$.  We now observe that we can write 
$$\begin{array}{rcl}
s_u & = & \underbracket{s_{u^{(1)}}\circ \cdots \circ s_{u^{(i)}}}_{s_{u^{(-)}}} \circ \underbracket{s_{w^{(1)}}\circ \cdots \circ s_{w^{(t)}}}_{s_{w}} \circ \underbracket{s_{u^{(i+t)}}\circ \cdots \circ s_{u^{(k)}}}_{s_{u^{(+)}}}\\
s_v & = & \underbracket{s_{v^{(1)}}\circ \cdots \circ s_{v^{(j)}}}_{s_{v^{(-)}}} \circ \underbracket{s_{w^{(1)}}\circ \cdots \circ s_{w^{(t)}}}_{s_{w}} \circ \underbracket{s_{v^{(j+t)}}\circ \cdots \circ s_{u^{(\ell)}}}_{s_{v^{(+)}}}.
\end{array}$$
One can thus construct extensions showing that $M(u^{(-)} \leftarrow w \leftarrow v^{(+)}), M(v^{(-)} \rightarrow w \rightarrow u^{(+)}) \in \mathcal{W}$. We conclude that $\mathcal{W} \in \text{wide}(\Lambda_T)$.

We now claim that the map $\omega: \text{wide}(\Lambda_T) \longrightarrow \Psi(\overrightarrow{FG}(T))$ defined by $$\mathcal{W} \mapsto \mathcal{S} := \overline{\{s_{u}: \ M(u) \text{ is a simple object of $\mathcal{W}$}\}}$$ is an order-preserving inverse to the map $\Psi(\overrightarrow{FG}(T)) \longrightarrow \text{wide}(\Lambda_T)$. Assuming that $\omega(\mathcal{W}) \in \Psi(\overrightarrow{FG}(T))$, it is clear that $\omega$ is an inverse as a map of sets.

Our earlier argument shows that the elements $M(u^{(i)})$ are the simple objects of $\mathcal{W}$ where $s_{u^{(i)}} \in \text{Seg}(\textbf{B})$. Thus to prove that $\mathcal{S} = \overline{\text{Seg}(\textbf{B})}$ for some $\textbf{B} \in \text{NCP}(T)$, it is enough to show that any two distinct segments $s_u$ and $s_v$ are noncrossing where $M(u), M(v) \in \mathcal{W}$ are simple objects in $\mathcal{W}$. Note that $\text{Hom}_{\Lambda_T}(M(u), M(v)) = \text{Hom}_{\Lambda_T}(M(v), M(u)) = 0$, since $M(u)$ and $M(v)$ are simple objects and $\mathcal{W}$ is a wide subcategory. 

If $s_u$ and $s_v$ share an endpoint, then Lemma~\ref{endpointoverlapchar} implies that $s_u$ and $s_v$ do not agree along a segment. Thus they are noncrossing in this case.

If $s_u$ and $s_v$ are crossing, then by Theorem~\ref{crossingnonsplitext}, and up to reversing the roles of $u$ and $v$ they define a unique nonsplit extension $$0 \to M(u) \to M(u^{(-)} \leftarrow w \leftarrow v^{(+)}) \oplus M(v^{(-)} \rightarrow w \rightarrow u^{(+)}) \to M(v) \to 0$$ where $u = u^{(-)} \leftarrow w \rightarrow u^{(+)}$, $v = v^{(-)} \rightarrow w \leftarrow v^{(+)}$. Using this description of the strings $u$ and $v$, we notice that there is map $f \in \text{Hom}_{\Lambda_T}(M(u), M(v))$ where $\text{im}(f) = M(w)$, a contradiction. Thus $s_u$ and $s_v$ are noncrossing. 

Next, we show that $\omega$ is order-preserving. Since any two simple objects of $\mathcal{W} \in \text{wide}(\Lambda_T)$ correspond to noncrossing segments, the segment defined by any indecomposable object of $\mathcal{W}$ can be expressed as a concatenation of segments corresponding to simple objects of $\mathcal{W}$. That is, the segments of $\omega(\mathcal{W})$ are in bijection with the indecomposable objects of $\mathcal{W}$. Thus if $\mathcal{W}_1 \subset \mathcal{W}_2$, one has $\omega(\mathcal{W}_1) \subset \omega(\mathcal{W}_2)$.\end{proof}

\begin{lemma}\label{HomLambdaFree}
Let $\textbf{B} \in \text{NCP}(T)$ and let $M(u), M(v)$ be two distinct indecomposable $\Lambda_T$-modules whose corresponding segments appear in $\text{Seg}(B)$ and $\text{Seg}(B^\prime)$, respectively, for some blocks $B$ and $B^\prime$ of $\textbf{B}$. Then one has $\text{Hom}_{\Lambda_T}(M(u), M(v)) = 0$ and $\text{Hom}_{\Lambda_T}(M(v), M(u)) = 0.$ 
\end{lemma}

\begin{proof}
First assume $B = B^\prime$. Since $M(u)$ and $M(v)$ are distinct, the corresponding segments $s_u$ and $s_v$ share at most one vertex of $T$. This means $u$ and $v$ are supported on disjoint sets of vertices of $Q_T$ so the statement holds. Thus we can assume that $s_u \in \text{Seg}(B)$ and $s_v \in \text{Seg}(B^\prime)$ where $B$ and $B^\prime$ are distinct blocks of $\textbf{B}$. Since $\textbf{B} \in \text{NCP}(T)$, this implies that $s_u$ and $s_v$ have no common endpoints. 

Let $\gamma_u$ and $\gamma_v$ be left admissible curves for $s_u$ and $s_v$, respectively, witnessing that $s_u$ and $s_v$ are noncrossing. Write $s_w = [a,b]$ for the unique maximal segment along which $s_u$ and $s_v$ agree, if it exists, and orient $\gamma_u$ and $\gamma_v$ from $a$ to $b$. Without loss of generality, we have two cases:

\begin{itemize}
\item[i)] $\text{supp}(M(u)) \subsetneq \text{supp}(M(v))$
\item[ii)] $\text{supp}(M(v))\backslash\text{supp}(M(u)) \neq \emptyset$ and $\text{supp}(M(u))\backslash\text{supp}(M(v)) \neq \emptyset$
\end{itemize}

Suppose $\text{supp}(M(u)) \subsetneq \text{supp}(M(v))$. Here $s_w = s_u$. By Lemma~\ref{Lemma:crossingcurves} (1), with $s_u$ playing the role of $t$, we have that $\gamma_v$ either turns left at both $a$ and $b$ or it turns right at both $a$ and $b$. This means that either $v = v^{(1)} \leftarrow u \leftarrow v^{(2)}$ or $v = v^{(1)} \rightarrow u \rightarrow v^{(2)}$ for some nonempty strings $v^{(1)}$ and $v^{(2)}$ in $\Lambda_T.$ Thus $\text{Hom}_{\Lambda_T}(M(u), M(v)) = 0$ and $\text{Hom}_{\Lambda_T}(M(v), M(u)) = 0.$

Now suppose that $\text{supp}(M(v))\backslash\text{supp}(M(u)) \neq \emptyset$ and $\text{supp}(M(u))\backslash\text{supp}(M(v)) \neq \emptyset$. We can assume that $a$ (resp. $b$) is an endpoint of $s_v$ (resp. $s_u$). Thus we can write $s_v = [a,b] \circ s_{v^{\prime}}$ and $s_u = s_{u^{\prime}} \circ [a,b]$ for some nonempty segments $s_{v^\prime}, s_{u^\prime} \in \text{Seg}(T)$. By Lemma~\ref{Lemma:crossingcurves} (2), with $[a,b]$ playing the role of $t$, we have that either $\gamma_v$ turns right at $b$ and $\gamma_u$ turns left at $a$ or $\gamma_v$ turns left at $b$ and $\gamma_u$ turns right at $a$. Thus either $v = w\rightarrow v^\prime$ and $u = u^\prime \leftarrow w$ or $v = w \leftarrow v^\prime$ and $u = u^\prime \rightarrow w.$ We conclude that $\text{Hom}_{\Lambda_T}(M(u), M(v)) = 0$ and $\text{Hom}_{\Lambda_T}(M(v), M(u)) = 0.$\end{proof}

\section{Polygonal subdivisions}\label{Sec:Poly_Sub}

In this section, we show how oriented flip graphs can be equivalently described using certain decompositions of a convex polygon $P \subset \mathbb{R}^2$ into smaller convex polygons called polygonal subdivisions. The notion of a flip between two facets of the reduced noncrossing complex will translate into a type of \textit{flip} between polygonal subdivisions of $P$. After that, we show that the polygonal subdivision corresponding to the top element of an oriented flip graph is obtained by \textit{rotating} the arcs in the polygonal subdivision corresponding to the bottom element. We show that oriented exchange graphs of quivers that are mutation-equivalent to type $\mathbb{A}$ Dynkin quivers are examples of oriented flip graphs. Lastly, we show that the Stokes poset of quadrangulations are also examples of oriented flip graphs.

A \textbf{polygonal subdivision} $\mathcal{P} = \{P_i\}_{i \in [\ell]}$ of a polygon $P$ is a family of polygons $P_1,\ldots,P_\ell$ such that
\begin{itemize}
\item $\ds\bigcup_{i=1}^l P_i=P$
\item $P_i\cap P_j$ is a face of $P_i$ and $P_j$ for all $i,j$, and
\item every vertex of $P_i$ is a vertex of $P$ for all $i$.
\end{itemize}

\noindent Equivalently, we can define a polygonal subdivision of $P$ to be a collection of pairwise noncrossing \textbf{diagonals} of $P$ (i.e. curves in $\mathbb{R}^2$ connecting two vertices of $P$) up to endpoint fixing isotopy.

\begin{figure}[h]
$$\begin{array}{ccccccccccccccc}
\raisebox{.25in}{\includegraphics[scale=1]{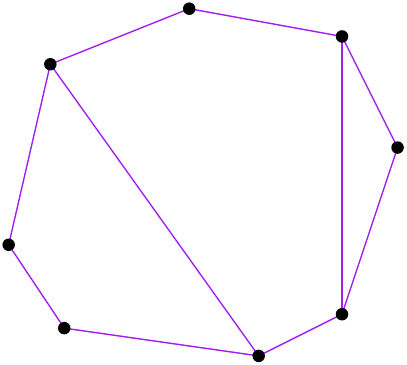}} & & & & \includegraphics[scale=1]{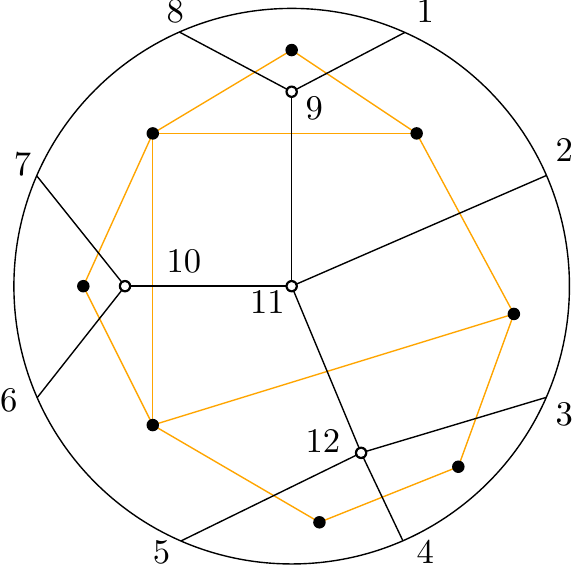}
\end{array}$$
\caption{Two examples of polygonal subdivisions where the latter is drawn with its corresponding tree.}
\label{poly1}
\end{figure}

\begin{remark}\label{polytreeduality}
Trees and polygonal subdivisions are dual. Given any tree $T$ embedded in $D^2$, it defines a polygonal subdivision $\mathcal{P}$ as follows. Let $P$ be a polygon with vertex set $\{v_F: F \text{ is a face of } T\}$ and where $v_{F_1}$ is connected to $v_{F_2}$ by an edge of $P$ if and only if there is an edge of $T$ that is incident to both $F_1$ and $F_2$. Using the data of the embedding of $T$, the resulting collection of polygons $\mathcal{P}$ is a polygonal subdivision of $P$. It is straightforward to verify that this construction can be reversed. We show an example of this duality in Figure~\ref{poly1}.
\end{remark}

Given a polygonal subdivision $\mathcal{P} = \{P_i\}_{i \in [\ell]}$ of a polygon $P$, there is a natural bound quiver $(Q_{\mathcal{P}}, I_{\mathcal{P}})$ that we associate to $\mathcal{P}$. Define $Q_{\mathcal{P}}$ to be the quiver whose vertices are in bijection with edges in $\mathcal{P}$ belonging to two distinct polygons $P_i, P_j \in \mathcal{P}$ and whose arrows are exactly those of the form $\epsilon_1 \stackrel{\alpha}{\longrightarrow} \epsilon_2$ satisfying:

$\begin{array}{rl}
i) & \text{$\epsilon_1$ and $\epsilon_2$ share a vertex of $P$,}\\
ii) & \text{$\epsilon_2$ is clockwise from $\epsilon_1$.} 
\end{array}$

\noindent The admissible ideal $I_{\mathcal{P}}$ is, by definition, generated by the relations $\alpha\beta$ where $\alpha: \epsilon_2 \longrightarrow \epsilon_3$, $\beta: \epsilon_1 \longrightarrow \epsilon_2$, and $\epsilon_1, \epsilon_2,$ and $\epsilon_3$ all belong to a common polygon $P_i \in \mathcal{P}$. We also define $\Lambda_{\mathcal{P}} := \Bbbk Q_\mathcal{P}/I_\mathcal{P}$. The following lemma is easy to verify using Remark~\ref{polytreeduality}.

\begin{lemma}
Let $T$ be a tree embedded in $D^2$ and let $\mathcal{P}$ be the corresponding polygonal subdivision. Then there are natural isomorphisms $Q_T \cong Q_\mathcal{P}$ and $\Lambda_T \cong \Lambda_\mathcal{P}.$
\end{lemma}

\begin{remark}
In \cite{simoes2016endomorphism}, it is shown that tiling algebras can be defined without reference to a combinatorial model such as a polygonal subdivision, and any tiling algebra so defined naturally gives rise to a polygonal subdivision. More generally, it is shown in \cite{schroll2015trivial} that any gentle algebra gives rise to a certain finite graph. If the algebra is a tiling algebra, then its polygonal subdivision from \cite{simoes2016endomorphism} is exactly the finite graph associated to it in \cite{schroll2015trivial}.
\end{remark}

\begin{remark}\label{Remark:triang} 
When $\mathcal{P} = \{P_i\}_{i \in [\ell]}$ is a \textbf{triangulation} of a polygon $P$ (i.e. each polygon $P_i$ is a triangle), the definition of the algebra $\Lambda_\mathcal{P}$ agrees with the definition of the \textbf{Jacobian algebra} \cite{dwz1} associated to the triangulation. Moreover, the triangulations of $P$ are exactly those polygonal subdivisions whose corresponding tree has only degree 3 interior vertices. When $\mathcal{P} = \{P_i\}_{i \in [\ell]}$ is an \textbf{$(m+2)$-angulation} of $P$ where $m \ge 1$ (i.e. each polygon $P_i$ is an $(m+2)$-gon), the algebra $\Lambda_\mathcal{P}$ is an \textbf{$m$-cluster-tilted algebra} of type $\mathbb{A}$ as was shown in \cite{murphy2010derived}.

Additionally, the class of tiling algebras also contains the \textbf{surface algebras} when the surface is the disk. These algebras were introduced in \cite{david2012algebras} and studied further in \cite{david2014derived} and \cite{amiot2016derived}.
\end{remark}

Now let $T$ be a tree embedded in $D^2$. Using Remark~\ref{polytreeduality}, let $\mathcal{P}_T$ be the polygonal subdivision of the polygon $P_T$ defined by $T$ and let $\{v_F: F \text{ is a face of } T\}$ be the set of vertices of the polygon ${P}_T$. There is an obvious bijection between elements of $\{v_F: F \text{ is a face of } T\}$ and the set of boundary vertices of $T$ given by sending $v_F$ to the counterclockwise most leaf of $T$ in face $F$. Using this bijection and the fact that any arc of $T$ is completely determined by the leaves of $T$ it connects, we obtain the following.

\begin{proposition}\label{Prop:NCfacetasPolySub}
Let $T$ be a tree embedded in $D^2$. The map sending each arc in a facet $\mathcal{F} \in \widetilde{\Delta}^{NC}(T)$ to its corresponding diagonal of $P_T$ defines a polygonal subdivision $\mathcal{P}(\mathcal{F})$ of $P_T$. This map defines an injection from the facets of $\widetilde{\Delta}^{NC}(T)$ to the set of polygonal subdivisions of $P_T$.
\end{proposition}

\begin{figure}[h]
\includegraphics[scale=1]{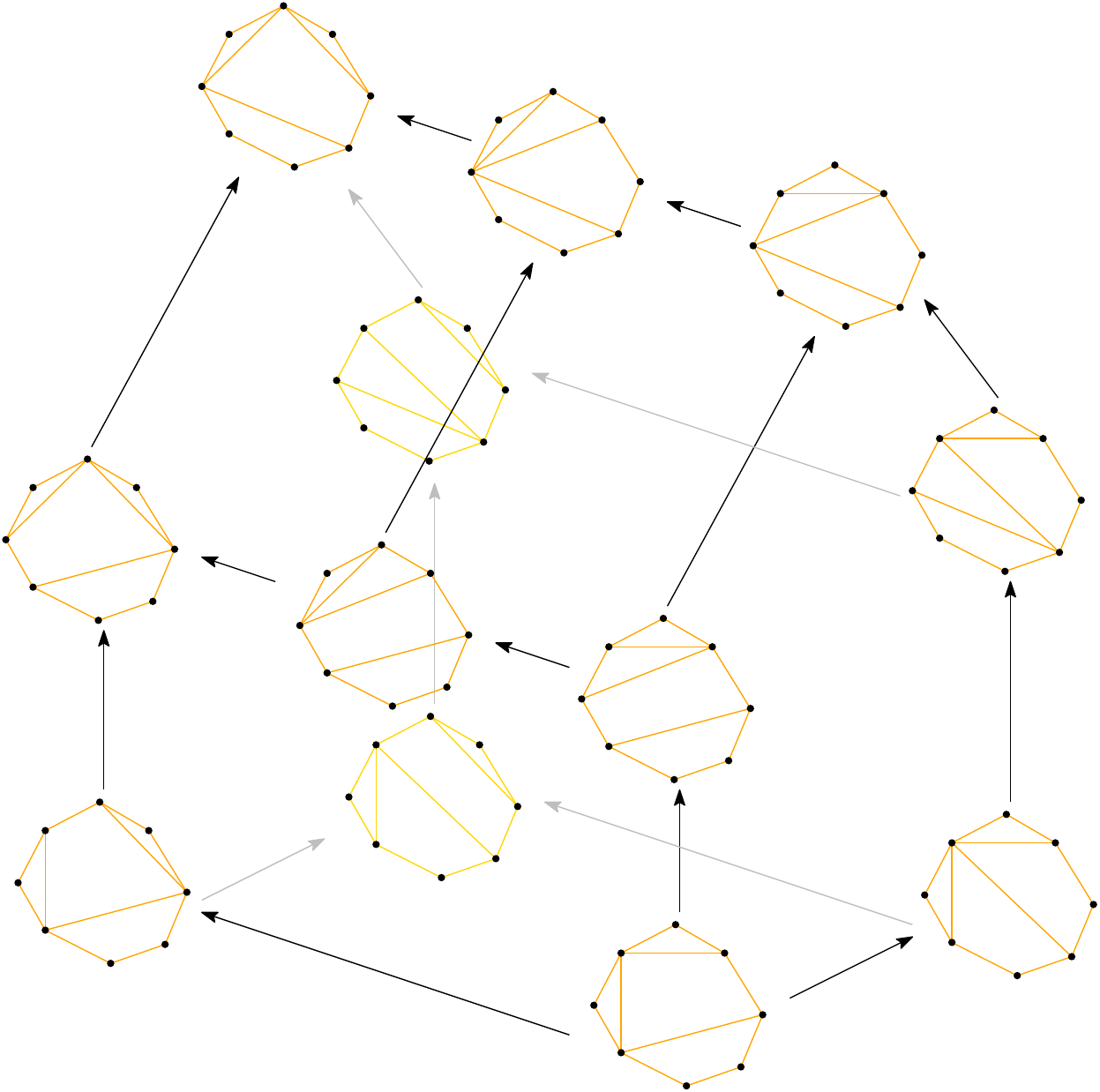}
\caption{}
\label{polysubex1}
\end{figure}

\begin{example}
By Proposition~\ref{Prop:NCfacetasPolySub}, we can identify the vertices of $\overrightarrow{FG}(T)$ with a certain subset of the polygonal subdivisions of $P_T$. In Figure~\ref{polysubex1}, we show the oriented flip graph from Figure~\ref{3dimorflipgraph} with its vertices represented by the corresponding polygonal subdivisions of $P_T$.
\end{example}

Next, we show how the polygonal subdivisions corresponding to the top and bottom elements of an oriented flip graph compare to each other. Note that there is a natural cyclic action on the diagonals of the polygon $P_T$. If $\alpha$ is a diagonal of $P_T$, we define the \textbf{rotation} of $\alpha$, denoted $\varrho(\alpha)$, to be the diagonal of $P$ whose endpoints are the vertices of $P$ immediately clockwise from the endpoints of $\alpha$ (see Figure~\ref{rotationexample}). If $\mathcal{P}(\mathcal{F})$ is a polygonal subdivision of $P_T$, we let $\varrho(\mathcal{P}(\mathcal{F}))$ denote the polygonal subdivision of $P_T$ obtained by applying $\varrho$ to each diagonal in $\mathcal{P}(\mathcal{F})$.

\begin{figure}
\includegraphics[scale=1]{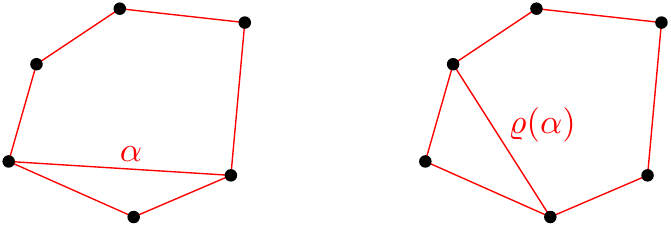}
\caption{The effect of $\varrho$ on a diagonal $\alpha$.}
\label{rotationexample}
\end{figure} 

\begin{theorem}\label{univtagrotation}
Let $T$ be a tree embedded in $D^2$. Then the bottom element (resp. top element) of $\overrightarrow{FG}(T)$ corresponds to the polygonal subdivision $\mathcal{P}_T$ (resp. $\varrho(\mathcal{P}_T)$).
\end{theorem}

\begin{proof}
Let $\mathcal{F}_1$ and $\mathcal{F}_2$ be the facets of $\widetilde{\Delta}^{NC}(T)$ corresponding to the bottom and top elements $\overrightarrow{FG}(T)$, respectively. Using Proposition~\ref{Prop:NCfacetasPolySub}, we let $\mathcal{P}(\mathcal{F}_1)$ and $\mathcal{P}(\mathcal{F}_2)$ be the corresponding polygonal subdivisions of $P_T$. It is clear that $\mathcal{F}_2 = \eta(\text{Seg}(T))$ and $\mathcal{F}_1 = \eta(\emptyset)$. 

Let $p_{(v,F)}$ be any arc of $T$ that appears in $\eta(\text{Seg}(T))$ (resp. $\eta(\emptyset)$). Let $u$ be any interior vertex of $T$ that appears in $p_{(v,F)}$, and orient the arc $p_{(v,F)}$ from $v$ to $u$. By the definition of $\eta$, the arc $p_{(v,F)}$ must turn left (resp. right) at $u$. 

Next, let $e = (v_1,v_2)$ be an edge of $T$ whose endpoints are internal vertices of $T$, and let $F$ and $G$ be the two faces of $T$ that are incident to $e$ and satisfy $(v_1, F)$ (resp. $(v_2, G)$) is immediately clockwise from $(v_1,G)$ (resp. $(v_2, F)$). Define $p:=p_{(v_1,G)} \in \eta(\emptyset)$ and $q:= p_{(v_1,F)} \in \eta(\text{Seg}(T))$ and let $\alpha_p \in \mathcal{P}(\mathcal{F}_1)$ and $\alpha_q \in \mathcal{P}(\mathcal{F}_2)$ be the diagonals corresponding to $p$ and $q,$ respectively. If we write $p = (u_1, \ldots, u_k, v_1, v_2, u_{k+1}, \ldots, u_{r})$ and $q = (w_1, \ldots, w_\ell, v_1, v_2, w_{\ell + 1}, \ldots, w_s)$, then the argument in the previous paragraph implies that the corners contained in $q$ are $(w_2, F), \ldots, (w_\ell, F), (v_1, F), (v_2, G), (w_{\ell+1}, G), \ldots, (w_{s-1}, G)$ and the corners contained in $p$ are $(u_2, G), \ldots, (u_k, G), (v_1, G), (v_2, F), (u_{k+1}, F), \ldots, (u_{r-1}, F).$ Thus we have that $\alpha_q = \varrho(\alpha_p)$ and $\mathcal{P}(\mathcal{F}_1) = \mathcal{P}_T$. The desired result follows.\end{proof}

As we mentioned in Example~\ref{Example_assoc}, the flip graph of a tree $T$ with only degree 3 internal vertices is isomorphic to the dual associahedron. By this identification and by Proposition~\ref{Prop:NCfacetasPolySub}, we obtain an orientation of the 1-skeleton of the associahedron. This orientation adds the data of a ``sign" to the operation of performing a single flip between two triangulations  $\mathcal{P}(\mathcal{F}_1), \mathcal{P}(\mathcal{F}_2)$ of $P_T$. 

It turns out that this oriented version of flipping between triangulations has been described by Fomin and Thurston (we refer the reader to \cite{fomin2012cluster} for more details). Given any triangulation $\mathcal{P}(\mathcal{F}_1)$ of $P_T$, one adds some additional curves $(L_1, \ldots, L_n)$ to $\mathcal{P}(\mathcal{F}_1)$ (here $n = \#(Q_T)_0$), called an \textbf{elementary lamination} (see \cite[Definition 17.2]{fomin2012cluster}), and records the \textbf{shear coordinates} \cite[Definition 12.2]{fomin2012cluster} (i.e. integer vectors indicating the number of certain crossings of arcs in $\mathcal{P}(\mathcal{F})$ and the curves $(L_1, \ldots, L_n)$). The elementary lamination is a collection of curves that are slightly deformed versions of the arcs in $\mathcal{P}_T$ and the shear coordinates are the \textbf{c}-vectors appearing in the \textbf{c}-matrix of the ice quiver corresponding to $\mathcal{P}(\mathcal{F}_1)$. Then there is a directed edge $\mathcal{P}(\mathcal{F}_1) \to \mathcal{P}(\mathcal{F}_2)$ in $\overrightarrow{FG}(T)$ if and only if $\mathcal{P}(\mathcal{F}_2)$ is obtained from $\mathcal{P}(\mathcal{F}_1)$ by performing a single diagonal flip on an arc $\alpha$ in $\mathcal{P}(\mathcal{F}_1)$ and the shear coordinate of $\alpha$ is positive in $\mathcal{P}(\mathcal{F}_1)$. We thus obtain following proposition.

\begin{proposition}\label{orexandorflip}
If $T$ is a tree whose internal vertices have degree 3, then $\overrightarrow{FG}(T) \cong \overrightarrow{EG}(\widehat{Q}_T)$ and this isomorphism commutes with flips and mutations.
\end{proposition}

\begin{remark}
A version of Theorem~\ref{univtagrotation} has been established  by Br\"ustle and Qiu (see \cite{brustle2015tagged}) for oriented exchange graphs defined by quivers arising from triangulations of \textbf{marked surfaces} (see \cite{fomin2008cluster} for more details). By identifying a convex polygon with an unpunctured disk, Theorem~\ref{univtagrotation} recovers their result in the case where one considers oriented flip graphs of a tree arising from a polygonal subdivision of an unpunctured disk. In their language, $\varrho$ is the \textbf{universal tagged rotation} of the marked surface.
\end{remark}

The Stokes poset defined by Chapoton in \cite{chapoton:stokes} is a partial order on a family of quadrangulations which are ``compatible'' with a given quadrangulation $Q$. The compatibility condition was defined by Baryshnikov as follows \cite{baryshnikov2001stokes}. Let $P$ be a $(2n)$-gon whose vertices lie on a circle. The vertices of $P$ are colored black and white, alternating in color around the circle. Let $P^{\pr}$ be the same polygon, rotated slightly clockwise. A \textbf{quadrangulation} is a polygonal subdivision into quadrilaterals. Fix a quadrangulation $Q$ of $P$. A quadrangulation $Q^{\pr}$ of $P^{\pr}$ is compatible with $Q$ if for each diagonal $q\in Q$ and $q^{\pr}\in Q^{\pr}$ such that $q$ and $q^{\pr}$ intersect, the white endpoint of $q^{\pr}$ appears clockwise from the white endpoint of $q$ before the black endpoint of $q$.

Let $T$ be the tree dual to $P$. We may assume that the leaves of $T$ are the vertices of $Q^{\pr}$. If $p$ is a geodesic between two leaves of $T$ that does not take a sharp turn at an interior vertex then it crosses a pair of opposite sides of some quadrilateral in $Q$. As a result, $p$ cannot be part of a quadrangulation compatible with $Q$. Let $\Delta_Q$ be the simplicial complex on the diagonals of $Q^{\pr}$ whose facets are quadrangulations compatible with $Q$. Then $\Delta_Q$ is a pure subcomplex of $\Delta^{NC}(T)$ of the same dimension. The complex $\Delta_Q$ is thin by Proposition 1.1 of \cite{chapoton:stokes}. Since the dual graph of $\Delta^{NC}(T)$ is connected, it follows that $\Delta_Q$ and $\Delta^{NC}(T)$ are isomorphic. Moreover, the orientation on the flips of quadrangulations defined in Section~1.3 of \cite{chapoton:stokes} coincides with $\ora{FG}(T)$. Consequently, we deduce the following proposition.

\begin{proposition}
If every interior vertex of $T$ has degree 4, then the poset $\ora{FG}(T)$ is isomorphic to the Stokes poset of quadrangulations compatible with the quadrangulation $\mathcal{P}_T$. 
\end{proposition}

\section{Simple-minded collections}\label{Sec:SMCs} In this section, we interpret noncrossing tree partitions in terms of the representation theory of $\Lambda_T$ using \text{simple-minded collections} in the bounded derived category of $\Lambda_T$, denoted $\mathcal{D}^b(\Lambda_T)$. We show that the data of a noncrossing tree partition and its Kreweras complement is equivalent to that of a certain type of simple-minded collection.

Simple-minded collections were originally used by Rickard \cite{rickard2002equivalences} in the construction of derived equivalences of symmetric algebras from stable equivalences. A standard example of a simple-minded collection in representation theory is a complete set of non-isomorphic simple $\Lambda$-modules regarded as elements of $\mathcal{D}^b(\Lambda)$. Note that any $\Lambda$-module $X$ becomes an element of $\mathcal{D}^b(\Lambda)$ by mapping it to the stalk complex concentrated in degree 0 whose degree 0 term is $X$. Additionally, in \cite{koenig2014silting}, simple-minded collections were useful in computing spaces of Bridgeland stability conditions \cite{bridgeland2007stability}. 

Here we recall some of the definitions we will need in order to study simple-minded collections. For a more complete presentation of the notions of derived categories and triangulated categories, we refer the reader to Chapter 1 of \cite{kashiwara2013sheaves}.

Let $\Lambda$ be a finite dimensional $\Bbbk$-algebra (or, more generally, a ring). By a \textbf{complex}, we mean a diagram of finitely generated $\Lambda$-modules $$\begin{array}{rcl}
X & = & \cdots \stackrel{d^{-2}_X}{\longrightarrow} X^{-1} \stackrel{d^{-1}_X}{\longrightarrow} X^0 \stackrel{d^{0}_X}{\longrightarrow} X^1 \stackrel{d^{1}_X}{\longrightarrow} X^2 \stackrel{d^{2}_X}{\longrightarrow} \cdots
\end{array}$$ that satisfies $d^{i+1}_X\circ d^i_X = 0$ for each $i \in \mathbb{Z}$. We say that the $\Lambda$-module $X^i$ in the complex $X$ is in \textbf{degree} $i$. We refer to the $\Lambda$-module homomorphisms $d_X^i: X^{i} \to X^{i-1}$ as \textbf{differentials}. If the only nonzero module of a complex $X$ is in degree $i$, we say that $X$ is a \textbf{stalk complex} concentrated in degree $i$. Given a complex $X$, it is natural to define the  \textbf{shift} of $X$, denoted $X[1]$, where $$\begin{array}{rcl}
X[1] & = & \cdots \stackrel{-d^{-1}_X}{\longrightarrow} X^{0} \stackrel{-d^{0}_X}{\longrightarrow} X^1 \stackrel{-d^{1}_X}{\longrightarrow} X^2 \stackrel{-d^{2}_X}{\longrightarrow} X^3 \stackrel{-d^{3}_X}{\longrightarrow} \cdots
\end{array}$$ and where in $X[1]$ the module in degree $i$ is $X^{i+1}.$ Now let $f: X \to Y$ be a morphism of complexes. We define the \textbf{mapping cone} or \textbf{cone} of $f$, denoted $\text{Cone}(f)$, to be the componentwise direct sum of complexes $$\begin{array}{rcl}
X[1] \oplus Y & = & \cdots \stackrel{d^{-2}_{\text{Cone}(f)}}{\longrightarrow} X^{0}\oplus Y^{-1} \stackrel{d^{-1}_{\text{Cone}(f)}}{\longrightarrow} X^1\oplus Y^{0} \stackrel{d^{0}_{\text{Cone}(f)}}{\longrightarrow} X^2\oplus Y^1 \stackrel{d^{1}_{\text{Cone}(f)}}{\longrightarrow} X^3\oplus Y^2 \stackrel{d^{2}_{\text{Cone}(f)}}{\longrightarrow} \cdots
\end{array}$$ with differential given by $$d^{i}_{\text{Cone}(f)} = \left[\begin{array}{cc} -d_X^{i+1} & 0\\ f^{i+1} & d^i_Y \end{array}\right].$$ Dually, one defines the \textbf{cocone} of $f$, denoted $\text{Cocone}(f)$.

The bounded derived category of $\Lambda$ has objects given by complexes $X$ of $\Lambda$-modules with $X^i = 0$ when $|i|$ is sufficiently large. Two objects $X$ and $Y$ in $\mathcal{D}^b(\Lambda)$ are isomorphic if and only if $X$ and $Y$ are \textbf{quasi-isomorphic} (i.e. there exists a morphism of complexes $\varphi: X \to Y$ that induces an isomorphism $H^k(X) \to H^k(Y)$ for all $k$). The category $\mathcal{D}^b(\Lambda)$, which is a triangulated category, also has the property that any triangle is isomorphic to a triangle of the form $$X \stackrel{f}{\longrightarrow} Y \longrightarrow \text{Cone}(f) \longrightarrow X[1].$$ One can also show that any triangle in $\mathcal{D}^b(\Lambda)$ is isomorphic to one of the form $$X[-1] {\longrightarrow} \text{Cocone}(f) \longrightarrow X \stackrel{f}{\longrightarrow} Y.$$

In this paper, we will be interested in understanding collections of objects from $\mathcal{D}^b(\Lambda)$ where the spaces of morphisms between any two objects in such a collection satisfy certain strong constraints. Morphism spaces between objects in derived categories can be very complicated. However, the objects in the collections we will study turn out to be stalk complexes. In this situation, the problem of understanding morphisms between such objects in $\mathcal{D}^b(\Lambda)$ is more tractable, as the following well-known proposition shows.  

\begin{proposition}
Let $X, Y \in \mathcal{D}^b(\Lambda)$ be stalk complexes concentrated in degree 0. Then $$\text{Hom}_{\mathcal{D}^b(\Lambda)}(X[i], Y[j]) = \text{Ext}^{j-i}_{\Lambda}(X,Y).$$
\end{proposition}

We now give the main definition of this section.

\begin{definition}\label{smcdefin}
Let $\mathcal{C}$ be a triangulated category. A collection $\{X_1,\ldots, X_n\}$ of objects of $\mathcal{C}$ is said to be \textbf{simple-minded} if the following hold for any $i,j \in [n]$:
\begin{itemize}
\item[i)] $\text{Hom}_\mathcal{C}(X_i, X_j[k]) = 0 \text{ for any } k < 0,$
\item[ii)] $\text{Hom}_\mathcal{C}(X_i, X_j) = \left\{\begin{array}{rcl} \Bbbk & : & \text{if } i = j\\ 0 & : & \text{otherwise,}  \end{array}\right.$
\item[iii)] $\mathcal{C} = \text{thick}\langle X_1, \ldots, X_n\rangle$ (i.e. the smallest triangulated category containing $X_1, \ldots, X_n$  and closed under taking summands of objects is $\mathcal{C}$). One says that the objects $\{X_1, \ldots, X_n\}$ form a \textbf{thick subcategory} of $\mathcal{C}$.
\end{itemize}
Now let $\Lambda$ be a finite dimensional $\Bbbk$-algebra and consider a simple-minded collection $\{X_1,\ldots, X_n\}$ in $\mathcal{D}^b(\Lambda)$. If for each $i \in [n]$ one has $H^k(X_i) = 0$ for any $k \neq 0, -1$, we say the collection is $\textbf{2-term}$. We let 2-smc($\Lambda$) denote the set of isomorphism classes of 2-term simple-minded collections of $\mathcal{D}^b(\Lambda).$
\end{definition}

It turns out that, as the following lemma shows, it is easy to say what objects can appear in a 2-term simple minded collection in $\mathcal{D}^b(\Lambda_T)$.

\begin{lemma}\label{stalks0and-1}
Let $\mathcal{X} = \{X_1, \ldots, X_n\} \in 2\text{-smc}(\Lambda_T)$. Each $X_i \in \mathcal{X}$ is isomorphic to a stalk complex of an indecomposable $\Lambda_T$-module concentrated in degree $0$ or $-1$.
\end{lemma}
\begin{proof}
By \cite[Remark 4.11]{by14}, each $X \in \mathcal{X}$ is isomorphic to a stalk complex of a $\Lambda_T$-module concentrated in degree $0$ or $-1$. Suppose $X \in \mathcal{X}$ is of the form $X \cong M[1]$ where $M \in \Lambda_T\text{-mod}$. Now we have that $$\begin{array}{rcl}
\text{End}_{\Lambda_T}(M) & = & \text{Hom}_{\Lambda_T}(M, M) \\
& = & \text{Hom}_{\mathcal{D}^b(\Lambda_T)}(M,M) \\
& = & \text{Hom}_{\mathcal{D}^b(\Lambda_T)}(M[1],M[1]) \\
& = & \Bbbk
\end{array}$$ where the last equality follows from the fact that $\mathcal{X} \in 2\text{-smc}(\Lambda_T)$. Since $\text{End}_{\Lambda_T}(M)$ is a local ring, $M$ is indecomposable. The proof is similar when $X \cong M$ for some $M \in \Lambda_T\text{-mod}$.
\end{proof}

From Lemma~\ref{stalks0and-1}, we have that any 2-term simple-minded collection $\mathcal{X} = \{X_1, \ldots, X_n\}$ in $\mathcal{D}^b(\Lambda_T)$ can be regarded as a collection of segments of $T$. We define $\text{Seg}(\mathcal{X}) = \{s_1, \ldots, s_n\}$ to be this collection where $s_i \in \text{Seg}(\mathcal{X})$ corresponds to $X_i \in \mathcal{X}$. Moreover, we can write $\text{Seg}(\mathcal{X}) = \text{Seg}^0(\mathcal{X}) \sqcup \text{Seg}^{-1}(\mathcal{X})$ where $$\text{Seg}^{i}(\mathcal{X}) := \{s_j \in \text{Seg}(\mathcal{X}): \ X_j \text{ is concentrated in degree $i$}\}.$$

The simple-minded collection $\mathcal{X}$ also naturally defines a graph lying on $D^2$ as follows. Let $\mathcal{SEG}(\mathcal{X})$ be the graph whose vertices are the internal vertices of $T$ and whose edges are admissible curves $\gamma_i$ defined by the segments $s_i \in \text{Seg}(\mathcal{X})$ up to endpoint fixing isotopy where if $s_i \in \text{Seg}^0(\mathcal{X})$ (resp. $s_i \in \text{Seg}^{-1}(\mathcal{X})$) then $\gamma_i$ is a green- (resp. red-) admissible curve. By abuse of notation, we will write $\mathcal{SEG}(\mathcal{X}) = \{\gamma_1, \ldots, \gamma_n\}$. It will also be useful to define $\mathcal{SEG}^{0}(\mathcal{X})$ (resp. $\mathcal{SEG}^{-1}(\mathcal{X})$) to be the subgraph of $\mathcal{SEG}(\mathcal{X})$ consisting of green- (resp. red-) admissible curves from $\mathcal{SEG}(\mathcal{X})$.

Our next main theorem, which we now state, gives a combinatorial classification of the 2-term simple-minded collections for the algebras $\Lambda_T$. This theorem implies that the data of a noncrossing tree partition paired with its Kreweras complement is equivalent to that for $\mathcal{SEG}(\mathcal{X})$ for a unique $\mathcal{X} \in 2\text{-smc}(\Lambda_T)$. 

\begin{theorem}\label{Thm:ncpsmcbijection}
There is a bijection $\theta: \{(\textbf{B},\text{Kr}(\textbf{B}))\}_{\textbf{B} \in \text{NCP}(T)} \longrightarrow 2\text{-smc}(\Lambda_T)$ given by $$(\textbf{B},\text{Kr}(\textbf{B})) \stackrel{\theta}{\longmapsto} \{M(u)[1]: s_u \in \text{Seg}(B) \text{ where } B \in \textbf{B}\} \sqcup \{M(v): s_v \in \text{Seg}(B^\prime) \text{ where } B^\prime \in \text{Kr}(\textbf{B})\}.$$
\end{theorem}

\begin{figure}
$$\includegraphics[scale=1.2]{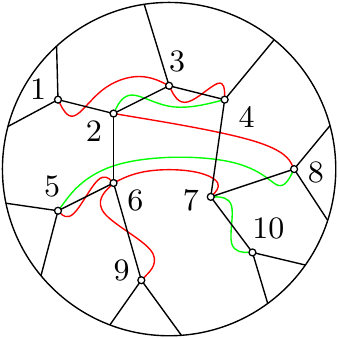} \ \ \raisebox{.7in}{$\longleftrightarrow$} \ \ \raisebox{.7in}{$\begin{array}{ccccccccccccc} \{M(w_{(1,3)})[1], & M(w_{(3,4)})[1], & M(w_{(2,8)})[1],\\ M(w_{(5,6)})[1], & M(w_{(6,7)})[1], & M(w_{(6,9)})[1], \\ M(w_{(2,4)}), & M(w_{(5,8)}), & M(w_{(7,10)})\}\end{array}$}$$
\caption{The noncrossing tree partition $\textbf{B} = (\{1,3,4\}, \{2,8\}, \{5, 6, 7, 9\},\{10\})$ with its Kreweras complement $\text{Kr}(\textbf{B}) = (\{1\}, \{2,4\}, \{3\}, \{5,8\}, \{6\}, \{7, 10\}, \{9\})$ and its corresponding simple-minded collection via the map $\theta$ in Theorem~\ref{Thm:ncpsmcbijection}. Here $w_{(i,j)}$ denotes the string corresponding to the segment of $T$ connecting $i$ and $j$.}
\label{noncrossingcurves}
\end{figure}

\begin{proof}
The image of $\theta$ lies in $2\text{-smc}(\Lambda_T)$ by Lemma~\ref{HomLambdaFree}, Lemma~\ref{Lemma:homvanishing}, and Lemma~\ref{generatesDb}. 

Next, decompose $\text{Seg}^0(\mathcal{X})$ and $\text{Seg}^{-1}(\mathcal{X})$ into segment-connected subsets of maximal size as follows: $$\text{Seg}^{0}(\mathcal{X}) = \bigsqcup_{i = 1}^\ell\text{Seg}^0_i(\mathcal{X}) \ \text{ and } \ \text{Seg}^{-1}(\mathcal{X}) = \bigsqcup_{i = 1}^k\text{Seg}^{-1}_i(\mathcal{X}).$$ In Section~\ref{smctoncp}, we construct a map from $\epsilon: 2\text{-smc}(\Lambda_T) \longrightarrow \{(\textbf{B},\text{Kr}(\textbf{B}))\}_{\pi \in \text{NCP}(T)}$ by $$\mathcal{X} \stackrel{\epsilon}{\longmapsto} (\textbf{B}_{\mathcal{X}}, \text{Kr}(\textbf{B}_{\mathcal{X}}))$$ where $\textbf{B}_{\mathcal{X}} := (B_1, \ldots, B_k)$ and where $B_i := \{\text{vertices of $T$ that are endpoints of segments in $\text{Seg}^{-1}_i(\mathcal{X})$}\}.$ It follows from Proposition~\ref{partitionassociatedtoX} that $\textbf{B}_{\mathcal{X}} \in \text{NCP}(T)$ and that any block $B_i^\prime$ in $\text{Kr}(\textbf{B}_{\mathcal{X}}) = (B_1^\prime, \ldots, B_\ell^\prime)$ satisfies $B_i^\prime = \{\text{vertices of $T$ that are endpoints of segments in $\text{Seg}^{0}_i(\mathcal{X})$}\}.$

It is easy to see that $\epsilon = \theta^{-1}.$
\end{proof}

\subsection{Mutation of simple-minded collections}\label{Sec:MutSMC}

Here we recall the notion of \textbf{mutation} of simple-minded collections and interpret this as a combinatorial operation on configurations of admissible curves. Our interpretation of mutation will be a key ingredient in showing that a 2-term simple-minded collection gives rise to a noncrossing tree partition paired with its Kreweras complement. 

Mutation was first introduced in \cite[Section 8.1]{KS} for spherical collections and generalized in \cite{koenig2014silting} to Hom-finite, Krull-Schmidt triangulated categories. This notion is defined using the language of \textbf{approximations}, which we now briefly review.

Let $\mathcal{C}$ be an arbitrary category (not necessarily triangulated), and let $\mathcal{A}$ be any subcategory of $\mathcal{C}$. We say that a morphism $f: C \to A$ where $C \in \mathcal{C}$ and $A \in \mathcal{A}$ is a \textbf{left $\mathcal{A}$-approximation} of $C$ if for any morphism $g: C \to A^\prime$ where $A^\prime \in \mathcal{A}$ one has $g = g^\prime f$ for some morphism $g^\prime : A \to A^\prime.$ Dually, one defines the notion of a \textbf{right $\mathcal{A}$-approximation} of $C$. Additionally, we say that $f: C \to A$ where $C \in \mathcal{C}$ and $A \in \mathcal{A}$ is \textbf{left minimal} morphism if for every morphism $g: A \to A$ that satisfies $gf = f$ one has that $g$ is an isomorphism. Dually, one defines \textbf{right minimal morphisms}. A morphism $f: C\to A$ (resp. $f: A \to C$) is a \textbf{left minimal $\mathcal{A}$-approximation} (resp. \textbf{right minimal $\mathcal{A}$-approximation}) if $f$ is left minimal and is a left $\mathcal{A}$-approximation (resp. right minimal and is a right $\mathcal{A}$-approximation).

Let $\mathcal{X} = \{X_1, \ldots, X_n\}$ be a simple-minded collection in $\mathcal{D}^b(\Lambda)$ where $\Lambda$ is an arbitrary finite dimensional $\Bbbk$-algebra. Let $\text{ext}(X_k)$ denote the \textbf{extension closure} of $X_k$ in  $\mathcal{D}^b(\Lambda)$ (i.e. the smallest subcategory of $\mathcal{D}^b(\Lambda)$ that contains $X_k$ and is closed under extensions). We define the \textbf{left mutation} of $\mathcal{X}$ to be $\mu_k^+(\mathcal{X}) := \{X_1^+, \ldots, X_n^+\}$ where $$\begin{array}{rcl} X_i^+ & := & \left\{\begin{array}{rcl} X_k[1] & : & \text{if $i = k$}\\
\text{Cone}(g^+_i: X_i[-1] \to X_{k,i}) & : & \text{if $i \neq k$}  \end{array}\right. \end{array}$$ where $g^+_i$ is a left minimal $\text{ext}(X_k)$-approximation. It is known that such approximations exist and that $\mu_k^+(\mathcal{X})$ is a simple-minded collection in $\mathcal{D}^b(\Lambda)$ (see \cite[Section 7.2]{koenig2014silting}). Dually, one defines the \textbf{right mutation} of $\mathcal{X}$, denoted $\mu_k^-(\mathcal{X})$. The resulting collection $\mu_k^-(\mathcal{X}):=\{X_1^-, \ldots, X_n^-\}$ has objects given by $$\begin{array}{rcl} X_i^- & := & \left\{\begin{array}{rcl} X_k[-1] & : & \text{if $i = k$}\\
\text{Cocone}(g_i^-: X_{k,i} \to X_i[1]) & : & \text{if $i \neq k$}  \end{array}\right. \end{array}$$ where $g_i^-$ is a right minimal $\text{ext}(X_k)$-approximation. It follows from \cite[Proposition 7.6 (a)]{koenig2014silting} that the $\mu_k^{-}\mu_k^+(\mathcal{X}) = \mathcal{X}$ and $\mu_k^{+}\mu_k^-(\mathcal{X}) = \mathcal{X}$.

\begin{remark}\label{Remark:leftrightmutation}
Let $\mathcal{X} = \{X_1, \ldots, X_n\} \in \text{2-smc}(\Lambda_T)$. By Lemma~\ref{stalks0and-1}, we have that $\mu_k^+(\mathcal{X}) \in \text{2-smc}(\Lambda_T)$ (resp. $\mu_k^-(\mathcal{X}) \in \text{2-smc}(\Lambda_T)$) if and only if $X_k$ is a stalk complex of an indecomposable concentrated in degree 0 (resp. $-1$). Using Proposition~\ref{Prop:commonendext}, we have that, when performing the mutation $\mu_k^+$ (resp. $\mu_k^-$) on $\mathcal{X}$, $\text{ext}(X_k) = \text{add}(X_k)$ (resp. $\text{ext}(X_k) = \text{add}(X_k[1])$).
\end{remark}

\begin{lemma}\label{Cone_Lemma}
Let $\mathcal{X} = \{X_1, \ldots, X_n\} = \{M(u^{(1)})[1], \ldots, M(u^{(n_1)})[1]\} \sqcup \{M(v^{(1)}), \ldots, M(v^{(n_2)})\} \in \text{2-smc}(\Lambda_T)$ and let $g_i^+: X_i[-1] \to X_{k,i}$ and $g_i^-: X_{k,i} \to X_i[1]$ be approximations used in the mutations $\mu_k^+(\mathcal{X})$ and $\mu_k^-(\mathcal{X})$. Then if $X_k = M(v^{(j)})$, we have $$\begin{array}{rcccl} X^+_i & = & \text{Cone}(g_i^+) & \cong & \left\{\begin{array}{rcl} M(v^{(j)} \leftarrow v^{(j^\prime)}) & : & \text{$\text{Ext}^1_{\Lambda_T}(X_i, X_{k}) \neq 0$ and $X_i = M(v^{(j^\prime)})$ where}\\ & & \text{supp}(M(v^{(j)}))\cap \text{supp}(M(v^{(j^\prime)})) = \emptyset,\\ M(w) & : & \text{$\text{Hom}_{\Lambda_T}(X_i[-1], X_{k}) \neq 0$ and $X_i = M(u^{(j^\prime)})[1]$ where} \\ & & \text{$\text{supp}(M(w)) = \text{supp}(M(v^{(j)}))\backslash\text{supp}(M(u^{(j^\prime)})) \text{ and}$}\\ & & \text{supp}(M(u^{(j^\prime)})) \subset \text{supp}(M(v^{(j)})),\\ M(w)[1] & : & \text{$\text{Hom}_{\Lambda_T}(X_i[-1], X_{k}) \neq 0$ and $X_i = M(u^{(j^\prime)})[1]$ where} \\ & & \text{$\text{supp}(M(w)) = \text{supp}(M(u^{(j^\prime)}))\backslash\text{supp}(M(v^{(j)}))$ and} \\ & & \text{supp}(M(u^{(j)})) \subset \text{supp}(M(v^{(j^\prime)})), \\ X_i & : & \text{otherwise.}  \end{array}\right.\end{array}$$ 
If $X_k = M(u^{(j)})[1]$, we have
$$\begin{array}{rcccl} X^-_i & = & \text{Cocone}(g_i^-) & \cong & \left\{\begin{array}{rcl} M(u^{(j^\prime)} \leftarrow u^{(j)})[1] & : & \text{$\text{Ext}^1_{\Lambda_T}(X_k, X_{i}) \neq 0$ and $X_i = M(u^{(j^\prime)})[1]$ where}\\ & & \text{supp}(M(u^{(j)}))\cap \text{supp}(M(u^{(j^\prime)})) = \emptyset, \\ M(w)[1] & : & \text{$\text{Hom}_{\Lambda_T}(X_k, X_{i}[1]) \neq 0$ and $X_i = M(v^{(j^\prime)})$ where} \\ & & \text{$\text{supp}(M(w)) = \text{supp}(M(u^{(j)}))\backslash\text{supp}(M(v^{(j^\prime)}))$ and}\\ & & \text{supp}(M(v^{(j^\prime)})) \subset \text{supp}(M(u^{(j)})), \\ M(w) & : & \text{$\text{Hom}_{\Lambda_T}(X_k, X_{i}[1]) \neq 0$ and $X_i = M(v^{(j^\prime)})$ where} \\ & & \text{$\text{supp}(M(w)) = \text{supp}(M(v^{(j^\prime)}))\backslash\text{supp}(M(u^{(j)}))$ and}\\ & & \text{supp}(M(u^{(j)})) \subset \text{supp}(M(v^{(j^\prime)})),  \\ X_i & : & \text{otherwise.} \end{array}\right. \end{array}$$
\end{lemma}

Lemma~\ref{Cone_Lemma} shows how mutation of a 2-term simple-minded collection $\mathcal{X}$ of $\mathcal{D}^b(\Lambda_T)$ can be understood combinatorially as an operation on admissible curves in  $\mathcal{SEG}(\mathcal{X})$. In Figure~\ref{Fig:typesofflip}, we illustrate the possible ways that mutation can effect $\mathcal{SEG}(\mathcal{X})$. Lemma~\ref{Cone_Lemma} also shows that $\mu_k^+(\mathcal{X})$ differs from $\mathcal{X}$ by at most three objects.

\begin{figure}
\includegraphics[scale=2]{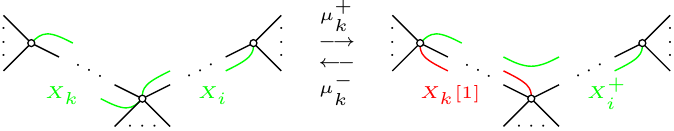}
\includegraphics[scale=2]{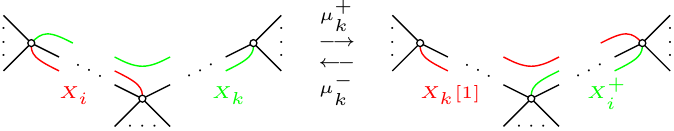}
\includegraphics[scale=2]{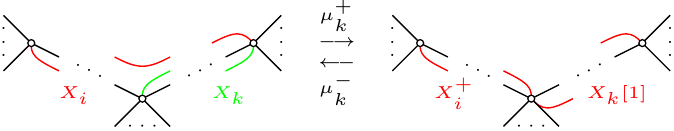}
\caption{The three types of nontrivial transformations.}
\label{Fig:typesofflip}
\end{figure}

\begin{proof}[Proof of Lemma~\ref{Cone_Lemma}]
It is easy to see that $X_{k,i}$ is isomorphic to $X_k$ or $0$, since $g_i^+$ is a left minimal $\text{add}(X_k)$-approximation. Note that the map $g_i^+$ defines the triangle $X_i[-1] \stackrel{g_i^+}{\longrightarrow} X_{k,i} \longrightarrow \text{Cone}(g_i^+) \longrightarrow X_i$ in $\mathcal{D}^b(\Lambda_T)$. This triangle gives rise to the long exact sequence $$0 \longrightarrow H^{-1}(\text{Cone}(g_i^+)) \longrightarrow H^0(X_i[-1]) \stackrel{(g_i^+)^*}{\longrightarrow} H^0(X_{k,i}) \longrightarrow H^0(\text{Cone}(g_i^+)) \longrightarrow H^1(X_i[-1]) \longrightarrow 0,$$ which, by Lemma~\ref{stalks0and-1}, vanishes outside of the terms shown. This sequence becomes $$0 \longrightarrow H^{-1}(\text{Cone}(g_i^+)) \longrightarrow H^{-1}(X_i) \stackrel{(g_i^+)^*}{\longrightarrow} H^0(X_{k,i}) \longrightarrow H^0(\text{Cone}(g_i^+))\longrightarrow H^0(X_i) \longrightarrow 0.$$

Now note that since $X_i$ is a stalk complex concentrated in degree $0$ or $-1$, we have the following two cases $$\begin{array}{rcl}\text{Hom}_{\mathcal{D}^b(\Lambda_T)}(X_i[-1], X_{k,i}) & = & \left\{\begin{array}{lcl}\text{Ext}^1_{\Lambda_T}(X_i, X_{k,i}) & : & \text{if } H^0(X_i) = X_i,\\ \text{Hom}_{\Lambda_T}(X_i[-1], X_{k,i}) & : & \text{if } H^{-1}(X_i) = X_i.\end{array}\right. \end{array}$$ We first consider the case when $H^0(X_i) = X_i$. By Lemma~\ref{dimext}, $\text{dim}_\Bbbk\text{Ext}^1_{\Lambda_T}(X_i, X_{k,i}) \le 1$. Suppose that $\text{dim}_\Bbbk\text{Ext}^1_{\Lambda_T}(X_i, X_{k,i}) = 0.$ This means that $g_i^+: X_i[-1] \to X_{k,i}$ is the zero map. Since $g_i^+$ is a left minimal morphism, this implies that $X_{k,i} = 0$. Then the long exact sequence implies that $H^{-1}(\text{Cone}(g_i^+)) \cong H^{-1}(X_i) = 0$ and $H^0(\text{Cone}(g_i^+)) \cong H^0(X_i) = X_i.$ Thus we obtain that $\text{Cone}(g_i^+) \cong X_i$.

Next, suppose that $\text{dim}_\Bbbk\text{Ext}^1_{\Lambda_T}(X_i, X_{k,i}) = 1.$ Since $g_i^+$ is a left minimal morphism, we know that $g_i^+$ is nonzero and thus $X_{k,i} = X_k.$ Assume $X_i$ is concentrated in degree 0 and write $X_k = M(v^{(j)})$, $X_i = M(v^{(j^\prime)})$. Since $\mathcal{X}$ is a simple-minded collection, $\text{Hom}_{\Lambda_T}(M(v^{(j)}), M(v^{(j^\prime)})) = 0$ and $\text{Hom}_{\Lambda_T}(M(v^{(j^\prime)}), M(v^{(j)})) = 0$. Thus Theorems~\ref{shareendptext} and \ref{crossingnonsplitext} imply that $0 \to M(v^{(j)}) \to M(v^{(j)} \leftarrow v^{(j^\prime)}) \to M(v^{(j^\prime)}) \to 0$ is the unique nonsplit extension of $M(v^{(j^\prime)})$ by $M(v^{(j)})$ up to equivalence of extensions. 

Let $M(v^{(j)}) {\longrightarrow} M(v^{(j)} \leftarrow v^{(j^\prime)}) \longrightarrow M(v^{(j^\prime)}) \stackrel{\xi}{\longrightarrow} M(v^{(j)})[1]$ be the triangle in $\mathcal{D}^b(\Lambda_T)$ defined by this nonsplit extension where $\xi$ is the class of this extension in $\text{Ext}^1_{\Lambda_T}(M(v^{(j^\prime)}), M(v^{(j)}))$. As $\text{dim}_\Bbbk\text{Ext}^1_{\Lambda_T}(M(v^{(j^\prime)}), M(v^{(j)})) = 1,$ we know that $\xi \neq 0$. Furthermore, we have that $g_i^+ = c\cdot\xi$ for some $c \in \Bbbk\backslash\{0\}$. Thus we have the following isomorphism of triangles in $\mathcal{D}^b(\Lambda_T)$ $$\xymatrix @R=1pc
{
M(v^{(j^\prime)})[-1]\ar@{->}[rr]^{-\xi}\ar@{=}[dd] & & M(v^{(j)})\ar@{->}[dd]^{(-c)\cdot 1}\ar@{->}[rr] & & M(v^{(j)} \leftarrow v^{(j^\prime)})\ar@{->}[rr]\ar@{-->}[dd]^{\cong} & & M(v^{(j^\prime)})\ar@{=}[dd] \\
&  & &  \\
M(v^{(j^\prime)})[-1]\ar@{->}[rr]^{c\cdot \xi} & & M(v^{(j)})\ar@{->}[rr] & & \text{Cone}(c\cdot \xi)\ar@{->}[rr] & & M(v^{(j^\prime)}) }$$ This implies that $\text{Cone}(g_i^+) \cong M(v^{(j)} \leftarrow v^{(j^\prime)}).$

Next, we consider the case when $H^{-1}(X_i) = X_i.$ By Lemma~\ref{dimhom}, $\text{dim}_\Bbbk\text{Hom}_{\Lambda_T}(X_i[-1], X_{k,i}) \le 1$. Suppose that $\text{dim}_\Bbbk\text{Hom}_{\Lambda_T}(X_i[-1], X_{k,i}) = 0.$ This means that $g_i^+: X_i[-1] \to X_{k,i}$ is the zero map. Since $g_i^+$ is a left minimal morphism, this implies that $X_{k,i} = 0$. Then the long exact sequence implies that $H^{-1}(\text{Cone}(g_i^+)) \cong H^{-1}(X_i) = X_i$ and $H^0(\text{Cone}(g_i^+)) \cong H^0(X_i) = 0.$ Thus we obtain that $\text{Cone}(g_i^+) \cong X_i$.

Now suppose that $\text{dim}_\Bbbk\text{Hom}_{\Lambda_T}(X_i[-1], X_{k,i}) = 1.$ Since $g_i^+$ is a left minimal morphism, we know that $g_i^+$ is nonzero and thus $X_{k,i} = X_k.$ Thus if we write $X_i[-1] = M(u^{(j^\prime)})$ and $X_k = M(v^{(j)})$, we have that $\text{supp}(M(u^{(j^\prime)})) \cap \text{supp}(M(v^{(j)})) \neq \emptyset.$ Furthermore, since $\mathcal{X}$ is a simple-minded collection, we have that $$\text{Ext}^1_{\Lambda_T}(M(u^{(j^\prime)}), M(v^{(j)})) = \text{Ext}^1_{\Lambda_T}(X_i[-1], X_k) = \text{Hom}_{\mathcal{D}^b(\Lambda_T)}(X_i, X_k) = 0.$$ Thus Theorem~\ref{crossingnonsplitext} implies that the segments $s_{u^{(j)}}$ and $s_{v^{(j^\prime)}}$ must share an endpoint. As $\text{Hom}_{\Lambda_T}(M(u^{(j^\prime)}), M(v^{(j)})) \neq 0$, the two segments must agree along a segment.

We know from Lemma~\ref{stalks0and-1} that $\text{Cone}(g_i^+)$ must be isomorphic in $\mathcal{D}^b(\Lambda_T)$ to either $M(w)$ or $M(w)[1]$ for some $M(w) \in \text{ind}(\Lambda_T\text{-mod})$ in order to have $\mu_k^{+}(\mathcal{X}) \in \text{2-smc}(\Lambda_T)$. This implies that either $\text{ker}((g_i^+)^*) = 0$ or $\text{coker}((g_i^+)^*) = 0$. In the former case $\text{Cone}(g_i^+) \cong M(w)$ where $\text{supp}(M(w)) = \text{supp}(M(v^{(j)}))\backslash\text{supp}(M(u^{(j^\prime)}))$. In the latter case $\text{Cone}(g_i^+) \cong M(w)[1]$ where $\text{supp}(M(w)) = \text{supp}(M(u^{(j^\prime)}))\backslash\text{supp}(M(v^{(j)}))$.

The computation of $\text{Cocone}(g_i^-)$ is similar so we omit it.
\end{proof}

\subsection{From simple-minded collections to noncrossing tree partitions}\label{smctoncp}

In this section, we show how any 2-term simple-minded collection gives rise to a noncrossing tree partition paired with its Kreweras complement.

Using left mutation, we can endow $2\text{-smc}(\Lambda_T)$ with a poset structure by regarding it as the transitive closure of the relation $\mathcal{X}_1 \lessdot \mathcal{X}_2$ if and only if $\mathcal{X}_2 = \mu_{k}^+(\mathcal{X}_1)$ for some $k \in [n]$. Perhaps surprisingly, this poset can be understood more globally. In \cite[Proposition 7.9]{koenig2014silting} it is shown that the partial order on $(2\text{-smc}(\Lambda_T), <)$ can be described as follows. If $\mathcal{X}_1 = \{X_1^{(1)}, \ldots, X_n^{(1)}\}, \mathcal{X}_2 = \{X_1^{(2)}, \ldots, X_n^{(2)}\} \in 2\text{-smc}(\Lambda_T),$ then $$\mathcal{X}_1 \le \mathcal{X}_2 \ \ \text{if and only if} \ \ \text{Hom}_{\mathcal{D}^b(\Lambda_T)}(X^{(1)}_i, X_j^{(2)}[m]) =0$$ for any $m < 0$ and any $i, j \in [n].$ The next proposition shows that the poset $(2\text{-smc}(\Lambda_T), <)$ has an even richer structure.

\begin{proposition}\label{Prop:2smclatticeprop}
The poset $(2\text{-smc}(\Lambda_T), <)$ is a finite lattice whose unique minimal (resp. maximal) element is $\{M(i): i \in (Q_T)_0\}$ (resp. $\{M(i)[1]: i \in (Q_T)_0\}$).
\end{proposition}
\begin{proof}
We will show that $(2\text{-smc}(\Lambda_T), <)$ is isomorphic to the lattice of torsion-free classes $\text{torsf}(\Lambda_T)$. The lattice $\text{torsf}(\Lambda_T)$ is finite since $\Lambda_T$ is representation-finite. 

By \cite[Theorem 3.1]{happel1996tilting} and \cite[Proposition 2.3]{woolf2010stability}, the poset $\text{torsf}(\Lambda_T)$ is isomorphic to the poset of \textbf{bounded $t$-structures} $(\mathcal{C}^{\le 0}_1, \mathcal{C}^{\ge 0}_1)$ on $\mathcal{D}^b(\Lambda_T)$ that satisfy $\mathcal{C}^{\le 0}[1] \subset \mathcal{C}^{\le0}_1 \subset \mathcal{C}^{\le0}$ or equivalently, $\mathcal{C}^{\ge 0}[1] \subset \mathcal{C}^{\ge0}_1 \subset \mathcal{C}^{\ge 0}$ where $$\mathcal{C}^{\le0} := \{X \in \mathcal{D}^b(\Lambda_T): \ H^i(X) = 0 \text{ for } i > 0\} \ \ \text{ and } \ \ \mathcal{C}^{\ge0} := \{X \in \mathcal{D}^b(\Lambda_T): \ H^i(X) = 0 \text{ for } i \le -1\}.$$ In the latter poset, bounded $t$-structures are partially ordered by inclusion: $$(\mathcal{C}_1^{\le0}, \mathcal{C}_1^{\ge0}) \le (\mathcal{C}_2^{\le0}, \mathcal{C}_2^{\ge0}) \ \text{ if and only if } \ \mathcal{C}_1^{\ge0} \subset \mathcal{C}_2^{\ge0}, \ \text{or equivalently, } \mathcal{C}_1^{\le0} \supset \mathcal{C}_2^{\le0}$$The isomorphism sends a torsion-free class $\mathcal{F}$ and its corresponding torsion class $\mathcal{T}$ to the bounded $t$-structure $(\mathcal{C}^{\prime\le 0}, \mathcal{C}^{\prime\ge 0})$ where $$\mathcal{C}^{\le0}_1 := \{X \in \mathcal{D}^b(\Lambda_T): \ H^i(X) = 0 \text{ for } i > 0, \ H^0(X) \in \mathcal{T}\}$$ and $$\mathcal{C}^{\ge0}_1 := \{X \in \mathcal{D}^b(\Lambda_T): \ H^i(X) = 0 \text{ for } i < -1, \ H^{-1}(X) \in \mathcal{F}\}.$$ Now, by \cite[Corollary 4.3]{by14} and the remarks following its proof, this poset of bounded $t$-structures is isomorphic to $(2\text{-smc}(\Lambda_T), <)$.

Remark~\ref{Remark:leftrightmutation} shows that the unique minimal (resp. maximal) element of $(2\text{-smc}(\Lambda_T), <)$ is $\{M(i): i \in (Q_T)_0\}$ (resp. $\{M(i)[1]: i \in (Q_T)_0\}$).
\end{proof}

\begin{proposition}\label{graphofSEG(X)}
Let $\mathcal{X} \in 2\text{-smc}(\Lambda_T)$. The graph $\mathcal{SEG}(\mathcal{X})$ is a \textbf{noncrossing} tree (i.e. any two admissible curves in $\mathcal{SEG}(\mathcal{X})$ are noncrossing in the sense of Lemma~\ref{Lemma:crossingcurves}). 
\end{proposition}
\begin{proof}
It is clear that $\mathcal{SEG}(\{M(i): i \in (Q_T)_0\})$ is a noncrossing tree. By Proposition~\ref{Prop:2smclatticeprop}, for any $\mathcal{X} \in 2\text{-smc}(\Lambda_T)$ there exists a sequence of left mutations such that $\mathcal{X} = \mu_{i_k}^+\circ \cdots \circ \mu_{i_1}^+(\{M(i): i \in (Q_T)_0\})$. By Lemma~\ref{Cone_Lemma}, we have that if $\mathcal{X}_2 = \mu_k^+(\mathcal{X}_1)$ and $\mathcal{SEG}(\mathcal{X}_1)$ is a tree, then $\mathcal{SEG}(\mathcal{X}_2)$ is a tree.

It remains to prove that if $\mathcal{X}_2 = \mu_k^+(\mathcal{X}_1)$ where $\mathcal{X}_1, \mathcal{X}_2 \in 2\text{-smc}(\Lambda_T)$ and $\mathcal{SEG}(\mathcal{X}_1) = \{\gamma_1, \ldots, \gamma_n\}$ is noncrossing, then $\mathcal{SEG}(\mathcal{X}_2) = \{\gamma_1^+, \ldots, \gamma_n^+\}$ is noncrossing. It is clear that the admissible curves in $\mathcal{SEG}(\mathcal{X}_1)\backslash\mathcal{SEG}(\mathcal{X}_2)$ are noncrossing. Write $\mathcal{X}_1 = \{X_1, \ldots, X_n\}$, $\text{Seg}(\mathcal{X}_1) = \{s_1, \ldots, s_n\}$, and $\text{Seg}(\mathcal{X}_2) = \{s_1^+, \ldots, s_n^+\}$. Without loss of generality, we can assume $k = 1$ and then $\mathcal{X}_2 =  \{X_1[1], X_{2}^+, \ldots, X_n^+\}$. By Lemma~\ref{Cone_Lemma}, $\mathcal{X}_2$ differs from $\mathcal{X}_1$ in at most three objects. This implies that, without loss of generality, $X_i^+ = X_i$ if $i \not \in \{1,2,3\}.$ Furthermore, the description of mutation in Lemma~\ref{Cone_Lemma} shows that the admissible curves in $\mathcal{SEG}(\mathcal{X}_2)\backslash\mathcal{SEG}(\mathcal{X}_1)$ are noncrossing. Thus it suffices to show any admissible curve from $\mathcal{SEG}(\mathcal{X}_2)\backslash\mathcal{SEG}(\mathcal{X}_1)$ and any admissible curve from $\mathcal{SEG}(\mathcal{X}_1)\backslash\mathcal{SEG}(\mathcal{X}_2)$ are noncrossing. Note that from our interpretation of mutation in terms of admissible curves (see Figure~\ref{Fig:typesofflip}), we see that there is no curve in $\mathcal{SEG}(\mathcal{X}_2)\backslash\mathcal{SEG}(\mathcal{X}_1)$ that crosses one from $\mathcal{SEG}(\mathcal{X}_1)\backslash\mathcal{SEG}(\mathcal{X}_2)$ in the sense that the two have a common endpoint $z(w,F)$ for some corner $(w,F)$ of $T$.

Next, we show that if $\gamma_\ell^+ \neq \gamma_\ell \in \mathcal{SEG}(\mathcal{X}_2)\backslash\mathcal{SEG}(\mathcal{X}_1)$, then $\gamma_\ell^+$ and any $\gamma_i^+ = \gamma_i \in \mathcal{SEG}(\mathcal{X}_1)\backslash\mathcal{SEG}(\mathcal{X}_2)$ are noncrossing. Write $s_\ell^+ = s_{w^{(\ell,+)}}$ and $s_i = s_{w^{(i)}}$ for some strings $w^{(\ell,+)}$ and $w^{(i)}$ in $\Lambda_T.$ Let $s_w = [a,b]$ be the unique maximal string along which $s_{w^{(\ell,+)}}$ and $s_{w^{(i)}}$ agree and orient $\gamma_\ell^+$ and $\gamma_i$ from $a$ to $b$. 

 Assume $s_{w^{(\ell,+)}}$ and $s_{w^{(i)}}$ share an endpoint and that $a$ is the shared endpoint. In this situation, one of $\gamma_\ell^+$ and $\gamma_i$ is red-admissible and the other is green-admissible. We assume $\gamma_\ell^+$ is green-admissible and $\gamma_i$ is red-admissible, and the following argument can be adapted to the case where $\gamma_\ell^+$ is red-admissible and $\gamma_i$ is green-admissible. Since $\mathcal{X}_2$ is a simple-minded collection, Definition~\ref{smcdefin} i) implies that $$\text{Hom}_{\Lambda_T}(M({w^{(\ell,+)}}),M({w^{(i)}})) = \text{Hom}_{\mathcal{D}^b(\Lambda_T)}(M({w^{(\ell,+)}}),M({w^{(i)}})[1][-1]) = 0$$ and so by Lemma~\ref{endpointoverlapchar}, there is a nonzero morphism $f: M({w^{(i)}}) \to M({w^{(\ell,+)}})$. Thus $w^{(i)} =w \rightarrow u^{(i)}$ and $w^{(\ell,+)} = w \leftarrow u^{(\ell)}$ for some strings $u^{(i)}$ and $u^{(\ell)}$ in $\Lambda_T$, one of which may be empty. This implies that $\gamma_\ell^+$ turns left at $b$ or $\gamma_i$ turns right at $b$. By Lemma~\ref{Lemma:crossingcurves} (c) (with $\gamma_\ell^+$ playing the role of $\gamma$), we have that $\gamma_\ell^+$ and $\gamma_i$ are noncrossing.
 
Now suppose that $s_{w^{(\ell,+)}}$ and $s_{w^{(i)}}$ do not share an endpoint. Assume that $\gamma_\ell^+$ is green-admissible and $\gamma_i$ is red-admissible. The following argument can be adapted to the case when $\gamma_\ell^+$ is red-admissible and $\gamma_i$ is green-admissible. Then since $\mathcal{X}_2$ is a simple-minded collection, Definition~\ref{smcdefin} ii) implies that $$\text{Ext}^1_{\Lambda_T}(M({w^{(\ell,+)}}),M({w^{(i)}})) = \text{Hom}_{\mathcal{D}^b(\Lambda_T)}(M({w^{(\ell,+)}}),M({w^{(i)}})[1]) = 0.$$ By Theorem~\ref{crossingnonsplitext} and the structure of $Q_T$, we have that one the following holds: 

\begin{itemize}
\item[a)] $w^{(i)} = u^{(i,1)} \leftarrow w \leftarrow u^{(i,2)}$ and $w^{(\ell,+)} = v^{(\ell,1)} \stackrel{\alpha}{\rightarrow} w \stackrel{\alpha}{\rightarrow} v^{(\ell,2)}$ where $u^{(i,1)}$ and $u^{(i,2)}$ are nonempty strings and $v^{(\ell,1)}$ and $v^{(\ell,2)}$ may be empty strings,
\item[b)] $w^{(i)} = u^{(i,1)} \leftarrow w$ and $w^{(\ell,+)} = v^{(\ell,1)} \stackrel{\alpha}{\rightarrow} w \stackrel{\beta}{\rightarrow} v^{(\ell,2)}$ where $u^{(i,1)}$ and $v^{(\ell,2)}$ are nonempty strings and $v^{(\ell,1)}$ may be an empty string,
\item[c)] $w^{(i)} = w \leftarrow u^{(i,2)}$ and $w^{(\ell,+)} = v^{(\ell,1)} \stackrel{\beta}{\rightarrow} w \stackrel{\alpha}{\rightarrow} v^{(\ell,2)}$ where $u^{(i,1)}$ and $v^{(\ell,1)}$ are nonempty strings and $v^{(\ell,2)}$ may be an empty string,
\item[$\text{a}^\prime$)] $w^{(i)} = u^{(i,1)} \rightarrow w \rightarrow u^{(i,2)}$ and $w^{(\ell,+)} = v^{(\ell,1)} \stackrel{\alpha}{\leftarrow} w \stackrel{\alpha}{\leftarrow} v^{(\ell,2)}$ where $u^{(i,1)}$ and $u^{(i,2)}$ are nonempty strings and $v^{(\ell,1)}$ and $v^{(\ell,2)}$ may be empty strings, 
\item[$\text{b}^\prime$)] $w^{(i)} = u^{(i,1)} \rightarrow w$ and $w^{(\ell,+)} = v^{(\ell,1)} \stackrel{\alpha}{\leftarrow} w \stackrel{\beta}{\leftarrow} v^{(\ell,2)}$ where $u^{(i,1)}$ and $v^{(\ell,2)}$ are nonempty strings and $v^{(\ell,1)}$ may be an empty string, or
\item[$\text{c}^\prime$)] $w^{(i)} = w \rightarrow u^{(i,2)}$ and $w^{(\ell,+)} = v^{(\ell,1)} \stackrel{\beta}{\leftarrow} w \stackrel{\alpha}{\leftarrow} v^{(\ell,2)}$ where $u^{(i,2)}$ and $v^{(\ell,1)}$ are nonempty strings and $v^{(\ell,2)}$ are may be an empty string.\end{itemize}
Here the orientation of the arrows labeled $\beta$ is determined by the fact that $\text{Ext}^1_{\Lambda_T}(M({w^{(\ell,+)}}),M({w^{(i)}})) = 0$, while the orientation of the arrows labeled $\alpha$ is determined by the structure of $Q_T$. Note that we cannot have $w^{(i)} = u^{(i,1)} \rightarrow w \leftarrow u^{(i,2)}$ for some nonempty strings $u^{(i,1)}$ and $u^{(i,2)}$, otherwise the structure of $Q_T$ implies that $$\text{Hom}_{\mathcal{D}^b(\Lambda_T)}(M(w^{(\ell,+)}), M(w^{(i)})[1][-1]) = \text{Hom}_{\Lambda_T}(M(w^{(\ell,+)}), M(w^{(i)})) \neq 0,$$ and this contradicts that $\mathcal{X}_2$ is a simple-minded collection. Using Lemma~\ref{Lemma:crossingcurves}, it is straightforward to verify that in each of these cases the admissible curves $\gamma_2^+$ and $\gamma_i$ are noncrossing.

Finally, assume $s_{w^{(2,+)}}$ and $s_{w^{(i)}}$ do not share an endpoint and $\gamma_2^+$ and $\gamma_i$ are of the same color. As $\mathcal{X}_2$ is a simple-minded collection, we know that $\text{Hom}_{\Lambda_T}(M(w^{(i)}), M(w^{(2,+)})) = 0.$ Thus $\text{Ext}^1_{\Lambda_T}(M(w^{(2,+)}), M(w^{(i)})) = 0$ by Theorem~\ref{crossingnonsplitext}. We obtain the same family of cases as in the previous paragraph and, as above, it is routine to verify from these that $\gamma_2^+$ and $\gamma_i$ are noncrossing.\end{proof}

\begin{proposition}\label{partitionassociatedtoX}
Let $\mathcal{X} \in 2\text{-smc}(\Lambda_T)$. There exists $\textbf{B}_{\mathcal{X}} = (B_1, \ldots, B_k) \in \text{NCP}(T)$ with Kreweras complement $\text{Kr}(\textbf{B}_{\mathcal{X}}) = (B^\prime_1, \ldots, B_\ell^\prime)$ such that 

\begin{itemize}
\item[i)] $\text{Seg}^{-1}(\mathcal{X}) = \bigsqcup_{i = 1}^k\text{Seg}(B_i)$ 
\item[ii)] $\text{Seg}^{0}(\mathcal{X}) = \bigsqcup_{i = 1}^\ell\text{Seg}(B_i^\prime).$
\end{itemize}
\end{proposition}
\begin{proof}
i) Write $\mathcal{SEG}^{-1}(\mathcal{X}) = \bigsqcup_{i = 1}^k \mathcal{SEG}^{-1}_i(\mathcal{X})$ where each $\mathcal{SEG}^{-1}_i(\mathcal{X})$ is a connected component of $\mathcal{SEG}^{-1}(\mathcal{X})$. Also, let $\text{Seg}^{-1}_i(\mathcal{X})$ denote the set of segments defined by $\mathcal{SEG}^{-1}_i(\mathcal{X})$. 

We claim that any two segments in $\text{Seg}^{-1}_i(\mathcal{X})$ either have no common vertices or they agree only at an endpoint of each. Since $\text{Hom}_{\mathcal{D}^b(\Lambda_T)}(X_{s},X_t) = 0$ for any objects in $\mathcal{X}$ and since any $\mathcal{SEG}^{-1}_i(\mathcal{X})$ is connected, Lemma~\ref{endpointoverlapchar} implies that there are no segments in $\text{Seg}^{-1}_i(\mathcal{X})$ that share an endpoint and agree along a segment. 

Suppose that $s_1, s_2 \in \text{Seg}^{-1}_i(\mathcal{X})$ agree along a segment, but have no common endpoints. Let $\gamma_1$ and $\gamma_2$ be the edges of $\mathcal{SEG}^{-1}_i(\mathcal{X})$ whose segments are $s_1$ and $s_2$, respectively. Since $\mathcal{SEG}^{-1}_i(\mathcal{X})$ is a tree, let $(\gamma^{(1)}, \ldots, \gamma^{(r)})$ with $\gamma^{(j)} \in \mathcal{SEG}^{-1}_i(\mathcal{X})$ denote the unique sequence of edges connecting an endpoint of $s_1$ to an endpoint of $s_2$. Let $(s^{(1)}, \ldots, s^{(r)})$ with $s^{(j)} \in \text{Seg}^{-1}_i(\mathcal{X})$ denote the sequence of segments defined by $(\gamma^{(1)}, \ldots, \gamma^{(r)})$. We assume $s^{(1)}$ (resp. $s^{(r)}$) agrees with $s_1$ (resp. $s_2$) at an endpoint, and, by the previous paragraph, we can assume that $s^{(j)}$ and $s^{(j+1)}$ agree only at endpoints for each $j$. Now from the structure of $T$, we have that $s^{(1)}$ agrees with $s_1$ along a segment or $s^{(r)}$ agrees with $s_2$ along a segment. In either situation we reach a contradiction.

We now have that each $\text{Seg}^{-1}_i(\mathcal{X})$ is an inclusion-minimal set of segments. Since $\mathcal{SEG}^{-1}_i(\mathcal{X})$ is a connected component of $\mathcal{SEG}^{-1}(\mathcal{X})$, we observe that $\text{Seg}^{-1}_i(\mathcal{X})$ is segment-connected. Thus for each $i \in [k]$, we define $$B_i := \{ v \in T : \ \text{$v$ is an endpoint of some segment in $\text{Seg}^{-1}_i(\mathcal{X})$}\},$$ and we obtain that $\text{Seg}^{-1}_i(\mathcal{X}) = \text{Seg}(B_i).$ By Proposition~\ref{graphofSEG(X)}, this implies that $\textbf{B}_{\mathcal{X}} := (B_1, \ldots, B_k) \in \text{NCP}(T).$

The proof of ii) is similar so we omit it. We remark that the noncrossing tree partition corresponding to $\mathcal{SEG}^{0}(\mathcal{X}) = \bigsqcup_{i = 1}^\ell \mathcal{SEG}^{0}_i(\mathcal{X})$ is defined as $\textbf{B}^\prime := (B_1^\prime, \ldots, B_\ell^\prime)$ where $$B^\prime_i := \{ v \in T : \ \text{$v$ is an endpoint of some segment in $\text{Seg}^{0}_i(\mathcal{X})$}\}.$$

Lastly, we know that $\mathcal{SEG}(\mathcal{X})$ is a noncrossing tree by Proposition~\ref{graphofSEG(X)}. Furthermore, we have that the green segments in $\text{Seg}^{-1}(\mathcal{X}) = \bigsqcup_{i = 1}^k\text{Seg}(B_i)$ and the red segments in $\text{Seg}^{0}(\mathcal{X}) = \bigsqcup_{i = 1}^\ell\text{Seg}(B_i^\prime)$ define a red-green tree. Thus Corollary~\ref{cor_kreweras_rgtree} implies that $\textbf{B}^\prime = \text{Kr}(\textbf{B}_{\mathcal{X}})$. \end{proof}

\subsection{From noncrossing tree partitions to simple-minded collections} In this section, we present three lemmas whose combined result shows that the image of the map $\theta$, as defined in Theorem~\ref{Thm:ncpsmcbijection}, lies in $2\text{-smc}(\Lambda_T)$.

\begin{lemma}\label{Lemma:homvanishing}
Let $(\textbf{B}, \text{Kr}(\textbf{B})) \in \text{NCP}(T)^2$ and let $M(u)$ (resp. $M(v)$) be an indecomposable $\Lambda_T$-module whose corresponding segment appears in $\text{Seg}(B)$ for some block $B$ of $\textbf{B}$ (resp. of $\text{Kr}(\textbf{B})$). Then 

$\begin{array}{rcl}
(1) & \text{$\text{Hom}_{\mathcal{D}^b(\Lambda_T)}(M(u)[1], M(v)[k]) = 0$  for any $k \le 0$,}\\
(2) & \text{$\text{Hom}_{\mathcal{D}^b(\Lambda_T)}(M(v), M(u)[1][k]) = 0$  for any $k \le 0$.}
\end{array}$
\end{lemma}

\begin{proof}For each part, we assume that $\textbf{B}$ is not the top or bottom element of $\text{NCP}(T)$, otherwise the statements hold vacuously. In each part, whenever we assume that $s_u = [y_1, y_2]$ and $s_v = [x_1, x_2]$ agree along a segment, we let $s_w = [a,b]$ denote the unique maximal segment along which they agree. Furthermore, we let $\gamma_u$ and $\gamma_v$ be admissible curves for $s_u$ and $s_v$, respectively, that witness the fact that $s_u \in \text{Seg}(B)$ for some block $B$ of $\textbf{B}$ and $s_v \in \text{Seg}(B^\prime)$ for some block $B^\prime$ of $\text{Kr}(\textbf{B})$, and orient this curves from $a$ to $b$.

(1) We have that $\text{Hom}_{\mathcal{D}^b(\Lambda_T)}(M(u)[1], M(v)[k]) = \text{Ext}_{\Lambda_T}^{k-1}(M(u), M(v)) = 0$, since $k - 1 \le -1.$

(2) Since $\text{Hom}_{\mathcal{D}^b(\Lambda_T)}(M(v), M(u)[1][k]) = \text{Ext}_{\Lambda_T}^{k+1}(M(v), M(u)) = 0$ for $k \le -2$, it is enough to show that $\text{Hom}_{\Lambda_T}(M(v), M(u)) = 0$ and $\text{Ext}^1_{\Lambda_T}(M(v), M(u)) = 0.$ 

We first show that $\text{Hom}_{\Lambda_T}(M(v), M(u)) = 0.$ Suppose that $s_v$ and $s_u$ have no common endpoints. We claim that $\nu :=\{\{x_1,x_2\}, \{y_1,y_2\}, \{i\}: \ i \in V^o\backslash\{x_1, x_2, y_1, y_2\} \}$ is a noncrossing tree partition. Since $\gamma_v$ and $\gamma_u$ do not cross and since $s_v$ and $s_u$ have no common endpoints, we can replace $\gamma_v$ with a red-admissible curve $\gamma_v^\prime$ representing $s_v$ that does not cross $\gamma_u$. Thus $\nu \in \text{NCP}(T)$.  Now by Lemma~\ref{HomLambdaFree}, we have that $\text{Hom}_{\Lambda_T}(M(v), M(u)) = 0.$

Now suppose the segments $s_u$ and $s_v$ share an endpoint. Since $s_u \in \text{Seg}(B)$ for some block $B$ of $\textbf{B}$ and $s_v \in \text{Seg}(B^\prime)$ for some block $B^\prime$ of $\text{Kr}(\textbf{B})$, they are distinct and thus share exactly one endpoint. We can assume that $s_u$ and $s_v$ agree along some segment, otherwise we are done. Since $s_u$ and $s_v$ agree along $s_w$, we must have that $v = v^\prime \leftrightarrow w$ and $u = u^\prime \leftrightarrow w$ for some strings $u^{\prime}$ and $v^{\prime}$ in $\Lambda_T$, at least one of which is nonempty. Assume $a$ is the common endpoint of $s_u$ and $s_v$. By Lemma~\ref{Lemma:crossingcurves} (3), with $s_w = [a,b]$ playing the role of $t$ and $\gamma_v$ playing the role of $\gamma$, we have that $\gamma_v$ either turns left at $b$ or $\gamma_u$ turns right at $b$. Thus either $v = v^\prime \rightarrow w$ and $u = w$ or $v = w$ or $u \leftarrow w.$ This implies that $\text{Hom}_{\Lambda_T}(M(v), M(u)) = 0.$

Lastly, we show that $\text{Ext}^1_{\Lambda_T}(M(v), M(u)) = 0$. By Proposition~\ref{separatesegext}, we can restrict to the situation where $s_u$ and $s_v$ have at least one common vertex of $T$.  By Proposition~\ref{Prop:commonendext}, we can assume that if  $s_u$ and $s_v$ have only one vertex in common, then that vertex is an endpoint of each. 

Assume $s_u$ and $s_v$ agree only at an endpoint. By Lemma~\ref{shareendptext}, $\text{Ext}^1_{\Lambda_T}(M(v), M(u)) \neq 0$ if and only if there exists an arrow $\alpha \in (Q_T)_1$ such that the string $(u \longleftrightarrow v) = (u \stackrel{\alpha}{\longleftarrow} v)$. Since $s_u \in B \in \textbf{B}$ and $s_v \in B^\prime \in \text{Kr}(\textbf{B})$, any admissible curve $\gamma_u$ (resp. $\gamma_v$) leaves its endpoints from their right (resp. left). Thus the existence of such an arrow $\alpha \in (Q_T)_1$ implies that $\gamma_u$ and $\gamma_v$ leave their common endpoint from a common corner of $T$, and such a configuration is not allowed.

Now assume $s_u$ and $s_v$ agree along a segment, but they have no common endpoints. Now we can write $u = u^{(1)}\leftrightarrow w \leftrightarrow u^{(2)}$ and $v = v^{(1)} \leftrightarrow w \leftrightarrow v^{(2)}$ for some strings $u^{(1)}, u^{(2)}, v^{(1)}$, and $v^{(2)}$ in $\Lambda_T$ where 

\begin{itemize}
\item[i)] $u^{(1)}$ and $u^{(2)}$ are nonempty or
\item[ii)] $v^{(1)}$ and $v^{(2)}$ are nonempty or
\item[iii)] $u^{(1)}$ and $v^{(2)}$ are nonempty and $u^{(2)}$ and $v^{(1)}$ are empty or
\item[iv)] $v^{(1)}$ and $u^{(2)}$ are nonempty and $u^{(1)}$ and $v^{(2)}$ are empty.
\end{itemize}

Suppose we are in case i). Since $s_u$ and $s_v$ are noncrossing and since $u^{(1)}$ and $u^{(2)}$ are nonempty, we have from Lemma~\ref{Lemma:crossingcurves} (1) (with $s_w$ playing the role of $t$) that $\gamma_u$ either turns left at $a$ and $b$ or turns right at $a$ and $b$. Thus $u = u^{(1)} \leftarrow w \leftarrow u^{(2)}$ or $u = u^{(1)} \rightarrow w \rightarrow u^{(2)}$. By Theorem~\ref{crossingnonsplitext}, we have that $\text{Ext}^1_{\Lambda_T}(M(v), M(u)) = 0$. In case ii), the analogous arguments shows that $v = v^{(1)} \leftarrow w \leftarrow v^{(2)}$ or $v = v^{(1)} \rightarrow w \rightarrow v^{(2)}$. Thus Theorem~\ref{crossingnonsplitext} implies that $\text{Ext}^1_{\Lambda_T}(M(v), M(u)) = 0$

Suppose we are in case iii). We have from Lemma~\ref{Lemma:crossingcurves} (2) (with $s_w$ playing the role of $t$, $\gamma_v$ playing the role of $\gamma$, and $\gamma_u$ playing the role of $\gamma^\prime$) that either $\gamma_v$ turns left at $b$ and $\gamma_u$ turns right at $a$ or $\gamma_v$ turns right at $b$ and $\gamma_u$ turns left at $a$. This implies that either $u = u^{(1)} \leftarrow w$ and $v = w \rightarrow v^{(2)}$ or $u = u^{(1)} \rightarrow w$ and $v = w \leftarrow v^{(2)}$. By Theorem~\ref{crossingnonsplitext}, we have that $\text{Ext}^1_{\Lambda_T}(M(v), M(u)) = 0$. The analogous argument can be used in case iv).\end{proof}

\begin{lemma}\label{generatesDb}
Let $(\textbf{B}, \text{Kr}(\textbf{B})) \in \text{NCP}(T)^2$. Then the objects $$\{M(u)[1]: s_u \in \text{Seg}(B) \text{ where } B \in \textbf{B}\} \sqcup \{M(v): s_v \in \text{Seg}(B^\prime) \text{ where } B^\prime \in \text{Kr}(\textbf{B})\} \subset \mathcal{D}^b(\Lambda_T)$$ form a thick subcategory of $\mathcal{D}^b(\Lambda_T)$.
\end{lemma}
\begin{proof}
Let $\mathscr{T}$ denote the smallest triangulated category that contains the objects in the statement of the lemma and that is closed under taking summands of its objects. Note that $M(u) \in \mathscr{T}$ for each $u \in \text{Seg}(B)$ where $B \in \textbf{B}$ because $\mathscr{T}$ is closed under taking shifts of objects. Since $\{M(i): i \in (Q_T)_0\}$ is a simple-minded collection, it is enough to show that every indecomposable $\Lambda_T$-module belongs to $\mathscr{T}$. To do so, we use what we call admissible sequences of segments.

We say $(s_{u^{(1)}}, \ldots, s_{u^{(k)}})$ is an \textbf{admissible sequence} of segments for $s = [a,b]$ if the following hold:

\begin{itemize}
\item[i)] $M(u^{(i)}) \in \mathscr{T}$ for each $i \in [k]$,
\item[ii)] $s_{i-1}$ and $s_i$ are segments that share an endpoint,
\item[iii)] vertex $a$ (resp. $b$) is an endpoint of $s_1$ (resp. $s_k$).
\end{itemize}

\noindent Observe that every segment $s = [a,b]$ has an admissible sequence of segments $(s_{u^{(1)}}, \ldots, s_{u^{(k)}})$ of length at most $n$ given by the sequence of segments connecting $a$ and $b$ in the red-green tree defined by $(\textbf{B}, \text{Kr}(\textbf{B}))$. We also remark that since $s = [a,b]$ is a segment, we know that the vertices $a^\prime$ and $b^\prime$ of $T$ that are the endpoints shared by $s_{u^{(i)}}$ and $s_{u^{(i+1)}}$ and by $s_{u^{(j)}}$ and $s_{u^{(j+1)}}$, respectively, for $i > 1$, $j < k$, and $i < j$ define a segment $[a^\prime, b^\prime] \in \text{Seg}(T)$. We prove that if every $s_u \in \text{Seg}(T)$ with an admissible sequence $(s_{u^{(1)}}, \ldots, s_{u^{(k)}})$ has the property that $M(u) \in \mathscr{T}$, then every $s_v \in \text{Seg}(T)$ with an admissible sequence $(s_{v^{(1)}}, \ldots, s_{v^{(k+1)}})$ has $M(v) \in \mathscr{T}$. If $s_u \in \text{Seg}(T)$ has an admissible sequence $(s_{u^{(1)}})$, then $s_u = s_{u^{(1)}}$ and so $M(u) \in \mathscr{T}$.

Now assume that every $s_u \in \text{Seg}(T)$ with an admissible sequence $(s_{u^{(1)}}, \ldots,  s_{u^{(k)}})$ has the property that $M(u) \in \mathscr{T}.$ Let $s_v = [a,b] \in \text{Seg}(T)$ be any segment and let $(s_{v^{(1)}}, \ldots,  s_{v^{(k+1)}})$ be an admissible sequence for $s_v$. Observe that in $(s_{v^{(1)}}, \ldots,  s_{v^{(k+1)}})$ there exists $i \in [k]$ such that, without loss of generality, $s_{v^{(i)}}$ and $s_{v^{(i+1)}}$ are distinct segments that satisfy one of the following 

\begin{itemize}
\item $\text{supp}(M(v^{(i)}))\cap \text{supp}(M(v^{(i+1)})) = \emptyset$ or 
\item $\text{supp}(M(v^{(i)}))\cap \text{supp}(M(v^{(i+1)})) \neq \emptyset.$
\end{itemize}

Suppose that $\text{supp}(M(v^{(i)}))\cap \text{supp}(M(v^{(i+1)})) = \emptyset$. Note that $s_{v^{(i)}}$ and $s_{v^{(i+1)}}$ agree only at an endpoint. By the properties of admissible sequences, this implies that $s_{v^{(i)}} \circ s_{v^{(i+1)}} \in \text{Seg}(T).$ Now we have that up to reversing the roles of $v^{(i)}$ and $v^{(i+1)}$, there is a nonsplit extension $0 \to M(v^{(i)}) \to M(v^{(i)}\leftarrow v^{(i+1)}) \to M(v^{(i+1)}) \to 0$. This means there is a triangle in $\mathcal{D}^b(\Lambda_T)$ given by $M(v^{(i)}) \to M(v^{(i)}\leftarrow v^{(i+1)}) \to M(v^{(i+1)}) \to M(v^{(i)})[1]$ so $M(v^{(i)}\leftarrow v^{(i+1)}) \in \mathscr{T}$. We obtain an admissible sequence $(s_{v^{(1)}}, \ldots, s_{v^{(i-1)}}, s_{(v^{(i)}\leftarrow v^{(i+1)})}, s_{v^{(i+2)}}, \ldots, s_{v^{(k+1)}})$ for $s_v$ of length $k$. By induction, we obtain that $M(v) \in \mathscr{T}$.

Now suppose that $\text{supp}(M(v^{(i)}))\cap \text{supp}(M(v^{(i+1)})) \neq \emptyset.$ Since $s_{v^{(i)}}$ and $s_{v^{(i+1)}}$ share an endpoint, it is easy to see that there is nonzero morphism $f: M(v^{(i)}) \to M(v^{(i+1)})$ or a nonzero morphism $f: M(v^{(i+1)}) \to M(v^{(i)})$. Without loss of generality, we assume the former. We obtain a triangle in $\mathcal{D}^b(\Lambda_T)$ given by $M(v^{(i)}) \stackrel{f}{\to} M(v^{(i+1)}) \to \text{Cone}(f) \to M(v^{(i)})[1]$ whose long exact sequence reduces to the following exact sequence $$0 \longrightarrow \underbracket{H^{-1}(\text{Cone}(f))}_{\cong \text{ker}(f)} \longrightarrow H^0(M(v^{(i)})) \stackrel{f}{\longrightarrow} H^0(M(v^{(i+1)})) \longrightarrow \underbracket{H^0(\text{Cone}(f))}_{\cong \text{coker}(f)} \longrightarrow 0.$$ We now have that $\text{Cone}(f) = M(w^{(1)})[1] \oplus M(w^{(2)})$ where $\text{supp}(M(w^{(1)})) = \text{supp}(M(v^{(i)}))\backslash\text{supp}(M(v^{(i+1)}))$ and $\text{supp}(M(w^{(2)})) = \text{supp}(M(v^{(i+1)}))\backslash\text{supp}(M(v^{(i)}))$. If $\text{supp}(M(w^{(1)})) = \emptyset$ (resp. $\text{supp}(M(w^{(2)})) = \emptyset$), one checks that $(s_{v^{(1)}}, \ldots, s_{v^{(i-1)}}, s_{w^{(2)}}, s_{v^{(i+2)}}, \ldots, s_{v^{(k+1)}})$ (resp. $(s_{v^{(1)}}, \ldots, s_{v^{(i-1)}}, s_{w^{(1)}}, s_{v^{(i+2)}}, \ldots, s_{v^{(k+1)}})$) is an admissible sequence for $s_v$ of length $k$. By induction, we obtain that $M(v) \in \mathscr{T}$.

Finally, suppose that both $\text{supp}(M(w^{(1)})) \neq \emptyset$ and  $\text{supp}(M(w^{(2)})) \neq \emptyset$. Since $\mathscr{T}$ is closed under taking summands of its objects, we have that $M(w^{(1)}), M(w^{(2)}) \in \mathscr{T}.$ From the properties of admissible sequences, we have that the vertices $a^\prime$ and $b^\prime$ of $T$ that are the endpoints shared by $s_{v^{(i-1)}}$ and $s_{v^{(i)}}$ and by $s_{v^{(i+1)}}$ and $s_{u^{(i+2)}}$, respectively, define a segment $[a^\prime, b^\prime] \in \text{Seg}(T)$. This implies that $s_{w^{(1)}}\circ s_{w^{(2)}} \in \text{Seg}(T)$. Thus, up to reversing the roles of $w^{(1)}$ and $w^{(2)}$, there is a nonsplit extension $0 \to M(w^{(1)}) \to M(w^{(1)} \leftarrow w^{(2)}) \to M(w^{(2)}) \to 0$. This extension defines a triangle in $\mathcal{D}^b(\Lambda_T)$ given by $M(w^{(1)}) \to M(w^{(1)} \leftarrow w^{(2)}) \to M(w^{(2)}) \to M(w^{(1)})[1]$. Thus $M(w^{(1)} \leftarrow w^{(2)}) \in \mathscr{T}$. We obtain an admissible sequence $(s_{v^{(1)}}, \ldots, s_{v^{(i-1)}}, s_{w^{(1)} \leftarrow w^{(2)}}, s_{v^{(i+2)}}, \ldots, s_{v^{(k+1)}})$ for $s_v$ of length $k$. By induction, we obtain that $M(v) \in \mathscr{T}$.\end{proof}

\section{Classification of \textbf{c}-matrices}\label{App:cmat}
We now apply our work to obtain a combinatorial classification of the \textbf{c-matrices} of quivers $Q_T$ (see Section~\ref{Sec:oreg}) where the internal vertices of $T$ are all of degree 3. By \cite{nakanishi2012tropical}, the vertices of the oriented exchange graph of $Q_T$ index the clusters in the cluster algebra \cite{fomin2002cluster} defined by $Q_T$. The \textbf{c}-{matrices} \cite{fomin2007cluster} of a quiver $Q$  are related to noncrossing partitions of finite Coxeter groups \cite{reading2007clusters} and many important objects in representation theory \cite{by14}. In \cite{by14}, the \textbf{c}-matrices of quivers were interpreted representation theoretically as certain {simple-minded collections} in the bounded derived category of a finite dimensional algebra $\Lambda$. Our result is that \textbf{c}-matrices of $Q_T$ are classified by noncrossing tree partitions of $T$ paired with their Kreweras complement.

\begin{figure}[t]
$$\begin{array}{cc}
\raisebox{.1in}{\includegraphics[scale=1.1]{A2_cmat.pdf}} & \includegraphics[scale=1.0]{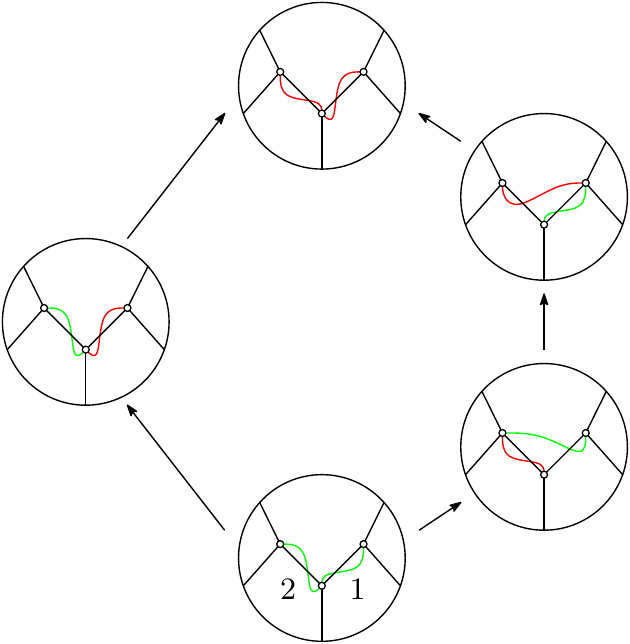}
\end{array}$$ 
\vspace{-.3in}
\caption{The $\textbf{c}$-matrices of $Q=2\leftarrow 1$ and the corresponding noncrossing tree partitions with their Kreweras complements.}
\label{oregA2}
\end{figure}

\begin{theorem}\label{Thm:cmatclassif}
Assume that $T$ is a tree whose internal vertices are of degree 3. 
\begin{enumerate}

\item The map $\varphi: \text{Seg}(T) \to \textbf{c}\text{-vec(}Q\text{)}^+$ defined by $s \mapsto (a_1,\ldots, a_n) \in \mathbb{Z}_{\ge 0}^{n}$, where $a_i := 1$ if the edge corresponding to vertex $i$ of $Q_T$ appears in $s$ and $a_i := 0$ otherwise, is a bijection.

\item The map $\{(\textbf{B}, \text{Kr}(\textbf{B}))\}_{\textbf{B} \in \text{NCP}(T)} \to \textbf{c}\text{-mat(}Q\text{)}$ defined by sending $(\textbf{B}, \text{Kr}(\textbf{B}))$ to the \textbf{c}-matrix $C$ whose negative \textbf{c}-vectors are $\{-\varphi(s): \ s \in \text{Seg}({B}) \text{ where } {B} \in \textbf{B}\}$ and whose positive \textbf{c}-vectors are $\{\varphi(s): \ s \in \text{Seg}(B^\prime) \text{ where } B^\prime \in \text{Kr}(\textbf{B})\}$ is a bijection (see Figure~\ref{oregA2}).
\end{enumerate}
\end{theorem}

\begin{proof}

(1) By Corollary~\ref{stringsegmentbij}, there is a bijection between segments of $T$ and the indecomposable modules of $\Lambda_T$. This bijection sends a segment $s$ to a string module $M(w)$ of $\Lambda_T$ where $w = w_1\leftrightarrow \cdots \leftrightarrow w_k$ has the property that each $w_i$ corresponds to an edge of $T$ whose vertices both appear in $s$. Now consider the map $\underline{\text{dim}}: \Lambda_T\text{-mod} \to \mathbb{Z}^n_{\ge 0}$. By \cite[Theorem 6]{chavez}, the restriction $\underline{\text{dim}}: \text{ind}(\Lambda_T) \to \textbf{c}\text{-vec}(Q)^+$ is a bijection. As the composition $s \mapsto \underline{\text{dim}}(M(w))$ agrees with the map in the assertion, this completes the proof. 

(2) By Theorem~\ref{Thm:ncpsmcbijection}, there is a bijective map $$(\textbf{B}, \text{Kr}(\textbf{B})) \stackrel{\theta}{\longmapsto}  \{M(u)[1]: s_u \in \text{Seg}(B) \text{ where } B \in \textbf{B}\} \sqcup \{M(v): s_v \in \text{Seg}(B^\prime) \text{ where } B^\prime \in \text{Kr}(\textbf{B})\}$$ where the latter belongs to $2\text{-smc}(\Lambda_T)$. Define a map $\Phi: 2\text{-smc}(\Lambda_T) \to \textbf{c}\text{-mat}(Q)$ by $$\{X_1,\ldots, X_n\} \mapsto \{\text{\underline{dim}}(X_1), \ldots, \text{\underline{dim}}(X_n)\}$$ where $\text{\underline{dim}}: \mathcal{D}^b(\Lambda_T) \to \mathbb{Z}^n$ is defined as $\text{\underline{dim}}(X_i) := \sum_{j \in \mathbb{Z}} (-1)^j\text{\underline{dim}}(X_i^j).$ The latter map was shown to be a bijection in \cite{by14}. Using the proof of (1), we see that $$(\textbf{B}, \text{Kr}(\textbf{B})) \stackrel{\Phi \circ \theta}{\longmapsto}  \{-\varphi(s_u): s_u \in \text{Seg}(B) \text{ where } B \in \textbf{B}\} \sqcup \{\varphi(s_v): s_v \in \text{Seg}(B^\prime) \text{ where } B^\prime \in \text{Kr}(\textbf{B})\}$$ and the result follows.
\end{proof}

\bibliographystyle{plain}
\bibliography{bib_treelattice}

\end{document}